\documentclass[a4pitaper,11pt]{amsart}
\usepackage{amsmath,amsthm,amssymb}
\usepackage[mathscr]{eucal}
 \usepackage{cite}
\usepackage{upgreek}
\usepackage[bookmarks,bookmarksnumbered]{hyperref}
\usepackage{enumerate}

\usepackage[dvipsnames]{xcolor}

\numberwithin{equation}{section}

\newcommand{\car}{\curvearrowright}

\newcommand{\G}{\Gamma}

\newcommand{\g}{\gamma}

\newcommand{\ca}{\curvearrowright}

\newcommand{\emm}{\mathcal{M}}

\newcommand{\enn}{\mathcal N}

\newcommand{\la}{\lambda}
\newcommand{\ra}{\rangle}
\newcommand{\El}{\mathcal{L}}

\newcommand{\La}{\Lambda}

\newcommand{\bten}{\bar\otimes}

\newcommand{\cA}{\mathcal A}
\newcommand{\cB}{\mathcal B}

\newcommand{\cD}{\mathcal D}

\newcommand{\cL}{\mathcal L}
\newcommand{\cM}{\mathcal M}
\newcommand{\cN}{\mathcal N}

\newcommand{\cP}{\mathcal P}
\newcommand{\cQ}{\mathcal Q}
\newcommand{\cR}{\mathcal R}

\def\R{\mathcal{R}}
\def\C{\mathcal{C}}

\def\P{\mathcal{P}}
\def\cL{\mathcal{L}}

\def\M{\mathcal{M}}

\def\A{\mathcal{A}}
\def\B{\mathcal{B}}

\def\ra{\rightarrow}

\def\ca{\curvearrowright}
\def\bten{\bar\otimes}
\def\G{\Gamma}

\theoremstyle{plain}
\newtheorem{main}{Theorem}

\newtheorem{mcor}[main]{Corollary}

\newtheorem{theorem}{Theorem}[section]

\newtheorem{lemma}[theorem]{Lemma}
\newtheorem{proposition}[theorem]{Proposition}
\newtheorem{corollary}[theorem]{Corollary}
\theoremstyle{definition}

\newtheorem{definition}[theorem]{Definition}

\newtheorem{assumption}[theorem]{Assumption}
\begin{document}

\title[Superrigid coinduced groups]
{$W^*$ and $C^*$-superrigidity results for coinduced groups}

\author{Ionu\c t Chifan}
\address{14 MacLean Hall, Department of Mathematics, The University of Iowa, IA, 52242, U.S.A.}
\email{ionut-chifan@uiowa.edu}
\thanks{I.C. has been supported in part by the NSF grants DMS-1600688, FRG DMS-1854194 and a CDA award from the University of Iowa}

\author{Alec Diaz-Arias}
\email{alec-diaz-arias@uiowa.edu}
\thanks{A.D-A. was supported in part by the AGEP- supplemental grant}

\author{Daniel Drimbe}
\address{Department of Mathematics, KU Leuven, Celestijnenlaan 200b, B-3001 Leuven, Belgium}
\email{daniel.drimbe@kuleuven.be}
\thanks {D.D. holds the postdoctoral fellowship fundamental research 12T5221N of the Research Foundation - Flanders.}


\begin{abstract} In this paper we explore a generic notion of superrigidity for von Neumann algebras $\cL(G)$ and reduced $C^*$-algebras $C^*_r(G)$ associated with countable discrete groups $G$. This allows us to classify these algebras for various new classes of groups $G$ from the realm of coinduced groups.

\end{abstract}

\maketitle

\section{Introduction}

The work of Murray and von Neumann \cite{MvN36,MvN43} shows that any countable discrete group $\Gamma$ gives rise in a canonical way to a von Neumann algebra, denoted $\cL(\Gamma)$. Namely, $\cL(\Gamma)$ is defined by considering the weak operator closure of the group algebra $\mathbb C[\Gamma]$ acting on the Hilbert space $\ell ^2(\Gamma)$ by left convolution. A main theme in operator algebras is to classify group von Neumann algebras, and hence, to understand to what extend $\cL(\Gamma)$ remembers some information about the group $\Gamma$. This problem is particularly studied when the group $\Gamma$ is {\it icc}, i.e. the conjugacy class of every non-trivial element of $\Gamma$ is infinite, which corresponds to $\cL(\Gamma)$ being a II$_1$ factor. In the amenable case, the classification problem is completed by the work of Connes \cite{Co76} asserting that for all icc amenable groups, their von Neumann algebras are isomorphic. This shows, in particular, that group von Neumann algebras tend to have little memory of the underlying group.

In sharp contrast, the non-amenable case is far more complex and it reveals presence of striking rigidity phenomena. In the last twenty years, Popa's deformation rigidity/theory \cite{Po07} led to the discovery of a plethora of structural results for group von Neumann algebras including many classes of groups $\Gamma$ for which various algebraic and analytical properties can be completely recovered from $\cL(\Gamma)$. In particular, the class of generalized wreath product groups (and more generally, coinduced groups, see Definition \ref{def:coind}) has been extensively studied and many rigidity results have been discovered. We only highlight here two developments in this direction, and refer the reader to the surveys \cite{Va10b,Io12,Io17} for more information. In his seminal work \cite{Po03,Po04}, Popa proved that non-isomorphic icc property (T) groups $\Gamma$ give rise to non-isomorphic von Neumann algebras $\cL(\mathbb Z/2\mathbb Z\wr\Gamma)$. Very recently, Popa and Vaes proved in \cite{PV21} many surprising results for the embedding problem of II$_1$ factors. For instance, to any infinite group $\Gamma$ they defined an icc group $H_\Gamma$ through a generalized wreath product construction such that $\cL(H_\Gamma)$ embeds in $\cL(H_\Lambda)$ if and only if $\Gamma$ is isomorphic to a subgroup of $\Lambda$.

In a remarkable work \cite{IPV10} a decade ago,  Ioana, Popa and Vaes discovered the first classes of groups $G$ that are completely remembered by their von Neumann algebra $\cL(G)$. Subsequently, a few other classes of such groups have been unveiled in \cite{BV12,Be14,CI17,CD-AD20,Dr20b}. 
Despite this progress, much more remains to be done in this area, as identifying new $W^*$-supperigid groups remains a challenging task. 

To properly formalize the rigidity questions and the results we will be studying in this paper, we first recall the notion of group-like $\ast$-isomorphism between von Neumann algebras (resp. $C^*$-algebras).
Let $G$ and $H$ be countable groups. Let $\omega:G \ra \mathbb T$ be a multiplicative character and $\delta : G\ra H$ a group isomorphism. Consider the group von Neumann algebras $\mathcal L(G)$ and $\mathcal L(H)$ and denote by $\{u_g \,:\, g\in G\}$ and $\{v_h \,:\, h\in H\}$ their corresponding canonical group unitaries. Then the map $G\ni u_g\ra \omega(g)v_{\delta(g)}$ canonically extends to a $\ast$-isomorphism denoted by $\Psi_{\omega,\delta}: \mathcal L(G)\ra \mathcal L(H)$. Moreover, one can easily see that this map also implements a $\ast$-isomorphism of their reduced $C^*$-algebras, $\Psi_{\omega,\delta}: C^*_r(G)\ra C_r^*(H)$. When $G=H$, the collection of all the maps $\Psi_{\omega,\delta}$ realizes a copy of the canonical semi-direct product group ${\rm Char}(G)\rtimes {\rm Aut}(G)$ inside the automorphism group of $\mathcal L(G)$. Finally, it is not difficult to see that if $G$ is icc, then $\Psi_{\omega, \delta}$ is not inner, whenever $\omega \in {\rm Char} (G)\setminus \{1\}$ or $\delta\in {\rm Out}(G)\setminus \{1\}$. 

\vskip 0.03in

Following \cite{IPV10}, using this notation, one formalizes a generic notion of superrigidity for group von Neumann algebras.  A countable group $G$ is called \emph{abstractly $W^*$-superrigid} if for any group $H$ and $\ast$-isomorphism $\Theta: \cL(G)\ra \cL(H)$, one can find $\Phi \in {\rm Aut}(\cL(H))$, $\omega\in {\rm Char}(G)$, and $\delta\in {\rm Isom} (G,H)$ so that $\Theta = \Phi\circ \Psi_{\omega,\delta}$. If $\Phi$ can always be chosen to be inner, then one says that $G$ is \emph{$W^*$-superrigid}---a notion already coined in\cite{IPV10}. Finally, abstract $W^*$-superrigidity is equivalent to saying that whenever $\cL(G)\cong \cL(H)$ for an arbitrary countable group $H$, then $G\cong H$. 

In their groundbreaking work \cite{IPV10}, Ioana, Popa and Vaes provided the first examples of groups $G$ that are completely reconstructible from their von Neumann algebras $\cL(G)$.   Their groups are certain generalized wreath products of the form $G= \mathbb Z/n \mathbb Z\wr_I\G$, where $n$ is a square-free integer. When $n\geq 2 $ they showed that the groups $G$ are abstractly $W^*$-superrigid; specifically, any $\ast$-isomorphism $\Theta: \cL(G)\ra \cL(H)$ appears as $\Theta =\Phi\circ \Psi_{\omega, \delta }$,   where $\Phi$  is composition between an inner automorphism of $\cL(H)$ and a certain automorphism induced from the abelian base algebra $\cL(\mathbb Z/n\mathbb Z)$. In addition, when $n=2,3$, then $\Phi$  can be chosen inner and thus in this case the groups $G$ are $W^*$-superrigid.  

\vskip 0.03in

While abstract $W^*$-superrigidity is formally weaker than $W^*$-superrigidity in practice it is significantly more robust to work with as it allows classification of new classes of group factors beyond the ones described above. This is because in general a factor $\cL(G)$ could have an abundance of automorphisms different from  the inner ones and the group-like ones.  This is perhaps best seen when considering any generalized wreath product group $G =A \wr_I \G$. Its von Neumann algebra canonically decomposes as crossed product $\cL(G)=\bar\otimes_I \cL(A) \rtimes_\sigma \G$, where $\sigma$ is the action induced by the generalized Bernoulli shift of $\Gamma$ on $\bar\otimes_I \cL(A)$. Then each automorphism $\phi\in {\rm Aut} (\cL(A)) $ induces a tensor product automorphism  $\otimes_I \phi \in {\rm Aut} (\bar \otimes_I \cL(A)) $. As it obviously commutes with the action  $\sigma$, then it extends uniquely to an automorphism denoted by $\Phi_\phi \in {\rm Aut} (\cL(G))$ that is identity on the subalgebra $\cL(\G)$. The $\Phi_\phi$ is called a Houghton-type automorphism (see Section \ref{Section:HoughtonAut}) and generally it is not inner and does not belong to $\rm Char (G) \rtimes Aut(G)$ (see Proposition \ref{Proposition:notunitary}). Remarkably, this is the case even for all inner automorphisms $\phi= {\rm ad}(u)$ implemented by any non-central, non-group unitary  $u\in \mathscr U(\cL(A))$.

 Thus, abstract $W^*$-superrigidity provides a more inclusive framework for classifying  new families of factors arising from generalized wreath product groups, many with non-amenable base groups.
In \cite{BV12}, Berbec and Vaes found the first examples of left-right wreath product groups of the form $G=\mathbb Z/ 2\mathbb Z\wr_\G (\G\times \G)$  that are W$^*$-superrigid. Motivated by this result and inspired by \cite{IPV10, KV15, CU18, CD-AD20},  we obtain a variety of new examples of abstractly W$^*$-superrigid left-right wreath product groups where the base group $\mathbb Z/2\mathbb Z$ of $G$ is replaced by any non-amenable W$^*$-superrigid group.



\begin{main}\label{Main:leftright}
Let $G=A_0\wr_{\Gamma} (\Gamma\times\Gamma)$ be a left-right wreath product, where $A_0$ is an abstractly $W^*$-superrigid group and $\Gamma$ is an icc, torsion free, bi-exact, weakly amenable, property (T) group. Let $H$ be any countable group and let $\Theta: \cL(G)\ra \cL(H)$ be a $\ast$-isomorphism. 

Then there exist a group isomorphism $\delta:G \ra H$,  a character $\eta:G\ra \mathbb T$, an automorphism  $ \phi\in {\rm Aut}( \cL(\delta(A_0)))$, and a unitary $w\in \cL(H)$ such that  $\Theta= {\rm ad} (w) \circ  \Phi_\phi \circ \Psi_{\eta,\delta}$.

In particular, $G$ is abstractly $W^*$-superrigid and not $W^*$-superrigid.
\end{main}

Theorem \ref{Main:leftright} is an immediate corollary of a much more general result, Theorem \ref{Th:coinduced} in which we completely describe all group von Neumann algebra decompositions of $\cL(A_0\wr_{\Gamma} (\Gamma\times\Gamma))$, where $A_0$ is a non-trivial arbitrary group and $\Gamma$ is an icc, torsion free, bi-exact, weakly amenable, property (T) group (e.g.\ any uniform lattice in $Sp(n,1)$, $n\geq 2$). Note that \cite{IPV10, BV12, Be14} provide numerous wreath product factors $\cL(A_0\wr_I\tilde  \Gamma)$ for which all their group von Neumann algebra decompositions are described. A common feature of these results is that the base group $A_0$ is amenable (even abelian). Hence, we provide the first class of wreath product factors $\cL(A_0\wr_I \tilde \Gamma)$ with arbitrary, possibly non-amenable, base $A_0$ for which we classify all their group von Neumann algebra decompositions.

Theorem \ref{Main:leftright} provides a new style of stability result that applies to a W$^*$-superrigidity notion. Namely, if $\Gamma$ is an icc, torsion-free, non-amenable, bi-exact, weakly amenable, property (T) group and $A_0$ is an {\it arbitrary} non-trivial group, then $A_0$ is abstractly $W^*$-superrigid if and only if the left-right wreath product group $A_0\wr_{\Gamma} (\Gamma\times\Gamma)$ is again abstractly $W^*$-superrigid. This should be compared with the \emph{functorial} result \cite[Theorem 1.1]{IPV10} which asserts that every non-amenable group can be embedded in a $W^*$-superrigid group.

Our second main result provides a fairly large class of coinduced groups that are W$^*$-superrigid. We first recall the coinduction construction. 

\begin{definition}\label{def:coind}
Let $\Gamma_0<\Gamma$ be countable groups and denote $I=\Gamma/\Gamma_0$. Let $\phi:I\to \Gamma$ be a section map and define the associated cocycle $c:\Gamma\times I\to\Gamma_0$ by the formula $c(g,i)=\phi(gi)^{-1}g\phi(i)$, for all $g\in\Gamma$ and $i\in I.$ For a group $A_0$ and for an action by group automorphisms $\Gamma_0\overset{\sigma_0}{\car} A_0$ we define the {\it coinduced action} 
$\Gamma\overset{\sigma}\car A_0^I$ of $\sigma_0$ by the formula
$$
(\sigma_g(x))_{gi}=(\sigma_0)_{c(g,i)}(x_{i}), \text{ for all } g\in\Gamma,i\in I \text{ and } x=(x_i)_{i\in I}\in A_0^I.
$$ 
The semi-direct product $A_0^I\rtimes_\sigma\Gamma$ corresponding to this action is called the {\it coinduced group} associated with $\Gamma_0\car A_0$ and $\Gamma_0<\Gamma$. We observe that the class of coinduced groups contains any generalized wreath product group $A_0\wr_I\Gamma$ with $\Gamma\car I$ transitive. Developing new technical aspects in Popa's deformation/rigidity we are able to classify all group von Neumann algebra decompositions for a fairly large family of coinduced groups.
\end{definition}

\begin{main}\label{Main:coinduced}
Let $G=A_0^I\rtimes_\sigma \tilde\Gamma$ be the associated coinduced group of $\Gamma_0\car A_0$ satisfying the following properties:
\begin{itemize}
    \item $\tilde\Gamma=\Gamma\times\Gamma$, where $\Gamma$ is an icc, torsion free, weakly amenable, bi-exact, property (T) group. 
    \item $\Gamma_0<\tilde\Gamma$ is the diagonal subgroup defined by $\Gamma_0=\{(g,g)|g\in\Gamma\}.$
    \item $A_0$ is icc, torsion free, bi-exact and it contains an infinite property (T) subgroup $K_0$ that is $\Gamma_0$-invariant and has trivial virtual centralizer.
\end{itemize}
Let $H$ be a countable group and let $\Theta: \cL(G)\ra \cL(H)$ be a $\ast$-isomorphism. 

Then there exist subgroups $\Lambda_0<\tilde\Lambda<H$ and $\Sigma_0<H$ such that $\Sigma_0$ is normalized by $\Lambda_0$ and $H=\Sigma_0^I\rtimes_\rho\tilde\Lambda$ is the coinduced group of $\Lambda_0\car \Sigma_0.$ Moreover, there exist a group isomorphism $\delta:\tilde\Gamma\to\tilde\Lambda$ with $\delta(\Gamma_0)=\Lambda_0$, a $*$-isomorphism $\theta:\cL(A_0)\to\cL(\Sigma_0)$ that satisfies $\theta(\sigma_g(a))=\rho_{\delta(g)}(\theta(a))$ for all $g\in \Gamma_0$ and $a\in \cL(A_0)$, a character $\eta: \tilde\Gamma\to\mathbb T$ and a unitary $w\in \cL(H)$ such that $ 
\Theta = {\rm ad} (w)\circ \Phi_\theta \circ \Psi_{\omega,\delta}.$

\end{main}

As a consequence of Theorem \ref{Main:coinduced}, we obtain the following W$^*$-superrigid coinduced groups.

\begin{mcor}\label{stabilitycoind}
We assume that $G=A_0^I\rtimes \tilde\Gamma$ is a coinduced group that satisfies the hypothesis of Theorem \ref{Main:coinduced} and assume in addition that ${\rm Comm}^{(1)}_{A_0\rtimes\Gamma_0}(\Gamma_0)=\Gamma_0$ and $A_0\rtimes\Gamma_0$ is $W^*$-superrigid. 

Then $G$ is $W^*$-superrigid.
\end{mcor}

In contrast to the stability result obtained from Theorem \ref{Main:leftright} for the abstract $W^*$-superrigid notion, Corollary \ref{stabilitycoind} provides a stability result for the W$^*$-superrigidity notion. More precisely, if $\Gamma_0<\tilde\Gamma$ are countable groups and $\Gamma_0\car A_0$ is an action by group automorphisms satisfying certain properties such that $A_0\rtimes\Gamma_0$ is  W$^*$-superrigid, then the associated coinduced group $A_0^I\rtimes\tilde\Gamma$ is again W$^*$-superrigid.
In order to provide some concrete applications of Corollary \ref{stabilitycoind}, we consider the following class of examples. First, recall the following class of W$^*$-superrigid groups from \cite{CD-AD20}.

{\bf Class $\mathscr {A}$}. Let $\G$ be an icc, torsion free, bi-exact, property (T) group and let  $\G_0,\G_1$,\dots,$\G_n$ be isomorphic copies of $\G$. For every $1\leq i\leq n$ consider the action $\G_0 \curvearrowright^{\rho^i} \G_i$ by conjugation, i.e.\ $\rho^i_\g (\la)=\g\la \g^{-1}$ for all $\g\in \G_0, \la\in \G_i$. Then let $\G_0\ca^{\rho} \G_1\ast \G_2\ast ...\ast \G_n$ be the action of $\G_0$ on the free product group $\G_1 \ast \G_2\ast...\ast \G_n$ induced by the canonical free product automorphism $\rho_\g = \rho^1_\g \ast \rho^2_\g \ast ...\ast \rho^n_\g$ for all $\g\in \G_0$. Denote by $\mathscr A$ the collection of all such semi-direct product groups $(\G_1\ast \G_2 \ast ... \ast \G_n) \rtimes_{\rho} \G_0$.

\begin{mcor}\label{examplecoind} 
Let $(\G_1\ast \G_2 \ast... \ast \G_n) \rtimes_{\rho} \G_0\in \mathscr A$. Note that $\Gamma_0$ can be seen as a subgroup of $\tilde\Gamma=\Gamma_0\times\Gamma_0$ via the diagonal embedding. 
  
Then  the associated coinduced group $G=(\G_1\ast \G_2\ast ...\ast \G_n)^I\rtimes \tilde\Gamma$ of $\rho$ is W$^*$-superrigid.

\end{mcor}

The first class of non-amenable {\it C$^*$-superrigid groups}, i.e. groups that can be completely recovered from their reduced $C^*$-algebras, were discovered in \cite[Corollary B]{CI17}. This class of groups consists of uncountable many groups that appear as certain amalgamated free products. Subsequently, other examples of C$^*$-superrigid amalgams, HNN extensions and semi-direct product groups have been found very recently in \cite{CD-AD20}. In this paper, we introduce new families of C$^*$-superrigid groups, from the realm of generalized wreath products as introduced in \cite{IPV10}, and also, coinduced groups in the sense of Definition \ref{def:coind}.

Our results provide a first instance when the $C^*$-superrigidity statement is slightly more complex than the ones presented in \cite{CI17,CD-AD20} as it involves a family of non-trivial automorphisms of Houghton type, see Corollary \ref{Cor:C*sup1} and Theorem \ref{Main:C*sup}. To properly introduce our $C^*$-superrigidity results, we need some new terminology. By analogy with the von Neumann algebra situation, we introduce the following definition.

 \begin{definition}
 A countable group $G$ is called \emph {abstractly $C^*$-superrigid} if the following holds: for any group $H$ and any $\ast$-isomorphism $\Theta:  C^*_r(G)\ra C_r^* (H)$ one can find an automorphism $\theta \in {\rm Aut}(C^*_r(H))$, a character $\eta: G\ra \mathbb T$ and a group isomorphism $\delta : G\ra  H$ so that $\Theta = \theta \circ \Psi_{\eta,\delta}$.  Moreover, if $\theta$ can always be implemented by a unitary in the von Neumann algebra $\cL(H)$, then we say that $G$ is \emph{ C$^*$-superrigid}; in this case, one recovers the C$^*$-superrigidity notion from \cite[Corollary B]{CI17}. Finally, abstract $C^*$-superrigidity  is equivalent to saying that whenever $C_r^*(G)\cong C_r^*(H)$ for an arbitrary  countable group $H$, then $G\cong H$.
\end{definition}

We remark that in the definition of \emph{C$^*$-superrigidity} allowing $\theta$ to be implemented by a unitary in $\cL(H)$, instead of $C^*_r(H)$ as one may be tempted to, is in fact optimal. Indeed, a remarkable result of J. Phillips \cite{Ph87} shows that for every countable icc group $H$ with $\cL(H)$ a factor that does not have Murray and von Neumann’s property Gamma \cite{MvN43}, there exist uncountably many unitaries $u\in \cL(H)\setminus C^*_r(H)$ such that $\theta=  {\rm ad}(u) $ implements an \emph{outer} automorphism of the reduced $C^*$-algebra $C^*_r(H)$. Since in the non-amenable case we consider mostly these type of groups the aforementioned framework becomes optimal in this study. See Section \ref{Sec:Autom} for more details on this.

As a corollary of Theorem \ref{Main:leftright}, we obtain that the following left-right wreath product groups are abstractly C$^*$-superrigid by using the von Neumann algebraic superrigidity property of the groups combined with their unique trace property \cite{BKKO14}.  

\begin{mcor}\label{Cor:C*sup1}
Let $\Gamma$ be an icc, torsion free, weakly amenable, bi-exact, property (T) group. Let $A_0$ be an icc abstractly W$^*$-superrigid group with trivial amenable radical. Denote $G=A_0\wr_{\Gamma}(\Gamma\times\Gamma)$ the left-right wreath product group. Let $H$ be any countable group and let  $\Theta: C^*_r(G)\ra C_r^*(H)$ be a $\ast$-isomorphism. 

Then there exist a group isomorphism $\delta:G \ra H$,  a character $\eta:G\ra \mathbb T$, an automorphism  $ \phi\in {\rm Aut}( \cL(\delta(A_0)))$, and a unitary $w\in \cL(H)$ such that  $\Theta= {\rm ad} (w) \circ  \Phi_\phi \circ \Psi_{\eta,\delta}$.

In particular, $G$ is abstractly $C^*$-superrigid and not $C^*$-superrigid.
\end{mcor}

As a consequence of Corollary \ref{examplecoind}, the following coinduced groups are  C$^*$-superrigid.

\begin{mcor}\label{Cor:C*sup2}
Let $(\G_1\ast \G_2 \ast... \ast \G_n) \rtimes_{\rho} \G_0\in \mathscr A$. Note that $\Gamma_0$ can be seen as a subgroup of $\tilde\Gamma=\Gamma_0\times\Gamma_0$ via the diagonal embedding. 

Then the associated coinduced group $G=(\G_1\ast \G_2\ast ...\ast \G_n)^I\rtimes \tilde\Gamma$ of $\rho$ is C$^*$-superrigid.
\end{mcor}

To properly introduce our next main result, we present a generalized wreath product group construction which is inspired by a series of developments \cite{Po03, Po04, PV06, IPV10, BV12, KV15, CU18}. The concrete groups that result from this construction represent our last main examples of $C^*$-superrigid groups. A point of contrast between all the known  C$^*$-superrigidity results and Theorem \ref{Main:C*sup} below, is that our methods, while still rooted in Popa's deformation/rigidity theory and hence von Neumann algebraic in nature, use in an essential way both facts that the groups involved have unique trace and also satisfy the Baum-Connes conjecture. 

\noindent {\bf Class $\mathscr {GW}$.} Let $\Omega\lhd \Gamma$ be a normal inclusion of infinite groups such that $\Omega$ has property (T). Let $B_0\leqslant A_0$ be an inclusion of infinite groups where $B_0$ has property (T).  Also assume $\Gamma \curvearrowright I$ is a transitive action on a countable set $I$ with ${\rm Stab}_\Gamma(I)=1$. We denote by $\mathscr{GW}(A_0)$ the category of all generalized wreath product $A_0\wr_I \Gamma$ associated with the previous data. In this definition we emphasize $A_0$ as for the main results will involve specific choices for this subgroup.  Next, consider the following properties for the action $\Gamma \curvearrowright I$:  
\begin{enumerate}
\item [a)]  For each $i \in I$ we have $[\Gamma:{\rm Stab}_{\Gamma}(i)]=\infty$;
\item [b)] There exists $k\in \mathbb N$ such that for each $J\subseteq I$ satisfying  $|J|\geq k$ we have $|{\rm Stab}_{\Gamma}(J)|<\infty$;
\item [c)] For every $i\neq j$ we have that $|{\rm Stab}_{\Gamma}(i)\cdot j|=\infty$. 
\end{enumerate}
Let $\mathscr {GW}_0(A_0)$ be the class of all elements in $\mathscr {GW}(A_0)$ for which $\Gamma \curvearrowright I$ satisfies a) and b) and  let  $\mathscr {GW}_1(A_0)$  the class of all elements in $\mathscr {GW}_0(A_0)$ so that $\Gamma \curvearrowright I$ satisfies c). We obviously have $$\mathscr {GW}_1(A_0)\subset \mathscr {GW}_0(A_0)\subset\mathscr {GW}(A_0).$$

\begin{main}\label{Main:C*sup} Let $G=A_0\wr_I \Gamma\in \mathscr {GW}_1(A_0)$, where $A_0$ and $\Gamma$ are torsion free, satisfy the Baum-Connes conjecture and  $A_0$  is an abstractly $W^*$-superrigid group with trivial amenable radical. Let $H$ be any countable group and let  $\Theta: C^*_r(G)\ra C_r^*(H)$ be a $\ast$-isomorphism.

Then there exist a group isomorphism $\delta:G \ra H$,  a character $\eta:G\ra \mathbb T$, an automorphism  $ \phi\in {\rm Aut}( \cL(\delta(A_0)))$, and a unitary $w\in \cL(H)$ such that  $\Theta= {\rm ad} (w) \circ  \Phi_\phi \circ \Psi_{\eta,\delta}$.

In particular, $G$ is abstractly  $C^*$-superrigid and not $C^*$-superrigid.
\end{main}

Note that all the non-amenable abstractly C$^*$-superrigid groups that have been discovered so far, namely \cite{CI17, CD-AD20} and Corollaries \ref{Cor:C*sup1} and \ref{Cor:C*sup2}, are also abstractly W$^*$-superrigid, while we do not know if this is the case for the groups $G$ that appear in Theorem \ref{Main:C*sup}. The reason behind this is that at the C$^*$-algebraic level, an isomorphism between two reduced C$^*$-algebras $C^*_r(G)\cong C^*_r(H)$, where $G$ is a torsion free group satisfying the Baum-Connes conjecture, implies that $C^*_r(G)$ has no non-trivial projections, and hence, $H$ is torsion free. Next, by promoting the C$^*$-algebraic isomorphism $C^*_r(G)\cong C^*_r(H)$ to an isomorphism between their group von Neumann algebras, we can use the torsion freeness property of the groups (as in \cite{CU18}) combined with some von Neumann algebraic techniques \cite{IPV10,KV15} to finally derive an isomorphism between the groups $G$ and $H$.  

We continue by providing some concrete applications of Theorem \ref{Main:C*sup}. Since any hyperbolic group satisfies the Baum-Connes conjecture  \cite{La12} and the class of groups satisfying the Baum-Connes conjecture is closed under direct products \cite[Theorem 3.1]{O-O01b} and amalgamated free products \cite[Corollary 1.2]{O-O01a}, we derive as a corollary the following classes of abstract C$^*$-superrigid groups.

\begin{mcor} Let $G=A_0\wr_I \Gamma\in \mathscr {G W}_1(A_0)$ be any group such that $\Gamma$ is torsion free,  property (T), hyperbolic and $A_0$ belongs to one of the following classes: 

\begin{enumerate}
    \item  Let $A_0 = (\G_1\ast\G_2\ast...\ast\G_n ) \rtimes_\rho \G\in\mathscr A$.
    \item Let  $C$ be torsion free, amenable icc group  and let $B$ be a torsion free hyperbolic property (T) group. Let $A< B$ be an infinite cyclic subgroup that is hyperbolically embedded. Consider the canonical inclusion  $\Sigma := C^{B}\rtimes A <C\wr B=:\G$ and let ${\rm diag}(\Sigma) =\{(g,g)\,:\, g\in \Sigma \}<\G\times \G$ be the diagonal subgroup. Then consider the amalgamated free product $A_0= (\Gamma \times \Gamma)\ast_{{\rm diag}(\Sigma)} (\Gamma\times \Gamma) $. 
\end{enumerate} 

Then $G$ is abstractly  C$^*$-superrigid and not C$^*$-superrigid.
\end{mcor}

We conclude the introduction by providing the first computation of the {\it strictly outer automorphism group} of some reduced C$^*$-algebras. We refer the reader to Section \ref{Sec:Autom} for terminology and for more details. In von Neumann algebras, Popa's deformation/rigidity theory led to the discovery of many instances when complete calculation of automorphism groups were achieved for various families of von Neumann algebras arising from groups and their trace preserving actions. For example, in \cite{IPP05} it was shown that there exist many II$_1$ factors that possess only inner automorphisms, thus settling a long-standing open question of A. Connes. More generally, in \cite{PV06} it was shown in fact that  every finitely presented group can be realized as the outer automorphism group of a certain group factor $\mathcal L(G)$. See also \cite{FV08,Va08,PV21} for other very interesting results in this direction. However, these impressive results in the von Neumann algebraic context do not yield much insight towards the reduced group $C^*$-algebras. 

In order to state our result, we need the following terminology. One says that two groups $G$ and $H$ are {\it commensurable up to finite kernels} (abbrev.\ $G\cong_{\rm fk}H$) if there exist finite index subgroups $G_1\leqslant G$, $H_1\leqslant H$ and finite, normal subgroups $M_1\lhd G_1$,  $N_1\lhd H_1$ such that the quotients  are isomorphic, $G_1/M_1\cong H_1/N_1$. Note that commensurability up to finite kernels is an equivalence relation.

\begin{mcor}  
 For any icc, torsion free, hyperbolic, property (T) group $\Gamma$, there exists a countable icc group $G$ such that ${\rm sOut}(C^*_r(G))\cong_{\rm fk} \G$. 
\end{mcor}

{\bf Acknowledgment.} We are grateful to Stefaan Vaes for numerous comments and suggestions regarding the content of this paper. The authors also want to thank Adrian Ioana and Jesse Peterson for their feedback on this work.

\section{Preliminaries on countable groups and von Neumann algebras}

\subsection{Terminology} We fix notation regarding countable groups and tracial von Neumann algebras. 

In this paper we consider countable groups $G$. A group inclusion $H\leqslant G$ of finite index will be denoted by $H\leqslant_f G$. For any subgroup $H\leqslant G$ we denote by $C_G(H)=\{ g\in G\,:\, [g,H]=1\}$ its {\it centralizer} in $G$, by $vC_G(H)=\{ g\in G\,:\, |g^H|<\infty\}$ its {\it virtual centralizer}. Note that $vC_G(G)=1$ precisely when $G$ has infinite non-trivial conjugacy classes (icc). For group inclusion $H<G$,  the {\it one-sided quasi-normalizer} ${\rm Comm}^{(1)}_G(H)$ is the semigroup of all $g\in G$ for which there exists a finite set $F\subset G$ such that $Hg\subset FH$ \cite[Section 5]{FGS10}; equivalently, $g\in {\rm Comm}^{(1)}_G(H)$ if and only if $[H: gHg^{-1}\cap H]<\infty$. The {\it quasi-normalizer} ${\rm Comm}_G(H)$ is the group of all $g\in G$ for which exists a finite set $F\subset G$ such that $Hg\subset FH$ and $gH\subset HF$.
We canonically have $C_G(H)\leqslant {\rm vC}_G(H)\leqslant {\rm Comm}_G(H)\subseteq {\rm Comm}^{(1)}_G(H)$.  

For an action $G\car I$ and a subset $F\subset I$, we denote by ${\rm Stab}_G(F)=\{g\in G \,:\, g\cdot i=i, \text{ for all } i\in F\}$ and by ${\rm Norm}_G(F)=\{g\in G \,:\, g\cdot F=F\}$ the {\it stabilizer} and the {\it normalizer} of $F$, respectively. If $F$ is finite, note that ${\rm Stab}_G(F)\leqslant_f {\rm Norm}_G(F)$.

Let $J$ be a set. We denote by $G^J = \oplus_J G$ the direct sum of $G$ along $J$, which is the group of all finitely supported functions $f: J\ra G$. For a subset $F\subset J$, we denote its complement $J\setminus F$ by  $\widehat F$. 

Let $\mathcal F$ be a family of subgroups of $G$. A set $K \subset G$ is called {\it small relative to} $\mathcal F$ if there exist finite subsets $R,T\subset G$ and $\mathcal G\subseteq \mathcal F$ such that $K\subseteq \cup_{\Sigma \in \mathcal G} R \Sigma T$.  
For every subset $K\subseteq G$ we denote by $P_K$ the orthogonal projection from $\ell^2(G)$ onto the Hilbert subspace generated by the linear span of $\{\delta_g \,:\, g\in K \}$.

Throughout this paper all von Neumann algebras $\cM$ are {\it tracial}, i.e. endowed with a unital, faithful, normal linear functional $\tau:\cM\rightarrow \mathbb C$  satisfying $\tau(xy)=\tau(yx)$ for all $x,y\in \emm$. This induces a norm on $\emm$ by the formula $\|x\|_2=\tau(x^*x)^{1/2}$ for all $x\in \emm$. The $\|\cdot\|_2$-completion of $\emm$ will be denoted by $L^2(\emm)$.  Denote by $\mathscr U(\mathcal M)$ its unitary group, by $\mathcal Z(M)$ its center, by $\mathscr P(\mathcal M)$ the set of all its non-zero projections and by $(\mathcal M)_1$ its unit ball.

For a unital inclusion $\mathcal N\subseteq \mathcal M$ of von Neumann algebras, we denote by $\mathcal N'\cap \mathcal M =\{ x\in \emm \,:\, [x, \enn]=0\}$ the {\it relative commmutant} of $\enn$ inside $\emm$ and by $\mathscr N_\emm(\enn)=\{ u\in \mathscr U(\emm)\,:\, u\enn u^*=\enn\}$ the {\it normalizer} of $\enn$ inside $\emm$. We say that $\cN$ is {\it regular} in $\cM$ if $\mathscr N_{\cM}(\cN)''=\cM$.
We also denote by $E_{\mathcal N}:\mathcal M\rightarrow \mathcal N$ the $\tau$-preserving condition expectation onto $\mathcal N$. We denote the orthogonal projection from $L^2(\emm) \rightarrow L^2(\enn)$ by $e_{\enn}$. The Jones' basic construction \cite[Section 3]{Jo83} for $\enn \subseteq \emm$ will be denoted by $\langle \emm, e_{\enn} \rangle$.

 Finally, for any countable group $G$ we denote by $(u_g)_{g\in G} \subset \mathscr U(\ell^2 (G))$ its left regular representation, i.e.\ $u_g(\delta_h ) = \delta_{gh}$ where $\delta_h:G\rightarrow \mathbb C$ is the Dirac function at $\{h\}$. The weak operator closure of the linear span of $\{ u_g\,:\, g\in G \}$ in $\mathscr B(\ell^2 (G))$ is called the {\it group von Neumann algebra} of $G$ and will be denoted by $\El(G)$; this is a II$_1$ factor precisely when $G$ is icc. The norm closure of the linear span of $\{ u_g\,:\, g\in G \}$ in $\mathscr B(\ell^2 (G))$ is called the {\it reduced group C$^*$-algebra} of $G$ and will be denoted by $C^*_r(G)$.

\subsection{Wreath products and coinduced groups} 



In this subsection we present several structural results for generalized wreath products and coinduced groups.

\begin{proposition}\label{relcomm1} 
Let $G=A_0^I\rtimes_\sigma\Gamma$ be the coinduced group of $\Gamma_0\car A_0$ with $A_0$ icc and let $F\subset I$ be a subset. Then $\cL(A_0^F)'\cap \cL(G)\subset \cL(A_0^I)\rtimes {\rm Stab_\Gamma}(F)$.

Moreover, if we assume that $G=A_0\wr_I \Gamma$ is a generalized wreath product (i.e. $\Gamma_0\car A_0$ acts trivially), then $\cL(A_0^F)'\cap \cL(G)\subset \cL(A_0^{\widehat F})\rtimes {\rm Stab_\Gamma}(F)$.
\end{proposition}

{\it Proof.} 
First we notice that $\mathcal L(A_0^F)'  \cap \mathcal L(G)$ is contained in $\cL(vC_{G}(A_0^F))$.
Thus, we fix $g =aq\in vC_{G}(A_0^F)$, where $a\in A_0^I$ and $q\in \Gamma$ and notice the centralizer satisfies $C_{A_0^F}(g)\leqslant_f A_0^F$. Thus for every $i\in F$ there is a subgroup $B_i\leqslant_f A_0$ so that $\prod_{i\in F} B_i \leqslant_f C_{A_0^F}(g)$. 

It follows that $aq k =kaq$, for all $k\in \prod_{i\in F} B_i$. This further gives $a \sigma_q(k)= ka$ and hence  $\sigma_q(k)= a^{-1}ka$, for all $k\in \prod_{i\in F} B_i$. Decomposing $a= ts$ with $s\in A_0^{\widehat F}$ and $t\in A_0^F$ we further get 
\begin{equation}\label{innerhomo}\sigma_q(k)= t^{-1}kt,\text{ for all }k\in \prod_{i\in F} B_i.
\end{equation} In particular, \eqref{innerhomo} implies that $qF\subseteq F$. By symmetry we also have $F\subseteq qF$ and hence $qF=F$. 
Now, consider the decompositions $t=(t_i)_{i\in F}$ and $k=(k_i)_{i\in F}\in \prod_{i\in F} B_i$. If we let $c:\Gamma\times \Gamma/\Gamma_0\to \Gamma_0$ be the cocycle that appears in Definition \ref{def:coind}, then \eqref{innerhomo} shows that 
\begin{equation}\label{p1}
c(q,q^{-1}i)\cdot k_{q^{-1}i}=t^{-1}_i k_i t_i \text{ for all }k_i\in B_i \text{ and }k_{qi}\in B_{qi}.  
\end{equation}
If there is an $i\in F$ so that $qi\neq i$ then the previous equation is impossible as $B_i$ and $B_{qi}$ are infinite. Thus, $qi =i$ for all $i\in F$, and hence, $q\in {\rm Stab}_{\Gamma}(F)$.

For proving the moreover part, note first that the cocycle $c$ is trivial in this case. Hence, equation \eqref{p1} shows that $t_i\in C_{A_0}(B_i)$ and since $[A_0:B_i]<\infty$ and $A_0$ is icc we conclude that $t_i=1$. Altogether this shows that $t=1$ and hence $a\in A_0^{\widehat F}$. In particular, we have  $vC_{G}(A_0^F)\leqslant A_0^{\widehat F}\rtimes {\rm Stab}_{\Gamma}(F)$ and the proof of the proposition is completed. 
\hfill$\blacksquare$


\begin{lemma}\label{L:coindcomm}
Let $\Gamma_0<\Gamma$ be countable groups such that $[\Gamma:\Gamma_0]=\infty$. Let $\Gamma_0\car A_0$ be an action by group automorphisms and denote by $\Gamma\car A$ the associated coinduced action. Let $G_0=A_0\rtimes\Gamma_0$ and $G=A_0^I\rtimes\Gamma.$ Then the following hold:
\begin{enumerate}
    \item If $g\in G\setminus \Gamma$, then there exists a finite non-empty set $F\subset I$ such that $g \Gamma g^{-1}\cap\Gamma\subset {\rm Norm}(F)$.
    \item {\rm Comm}$^{(1)}_G(\Gamma)=\Gamma.$ 
    \item If Comm$^{(1)}_\Gamma(\Gamma_0)=\Gamma_0$, then Comm$^{(1)}_G(\Gamma_0)=$Comm$^{(1)}_{G_0}(\Gamma_0)$.
\end{enumerate}
\end{lemma}

{\it Proof.} (1) If $g=a\gamma$ with $1\neq a\in A_0^I$ and $\gamma\in\Gamma$, then $g \Gamma g^{-1}\cap\Gamma\subset {\rm Norm}(F)$, where $F$ is the support of $a$. 

(2) For proving the second part, if $g\in {\rm Comm}^{(1)}_G(\Gamma)$, then $[\Gamma:g\Gamma g^{-1}\cap\Gamma]<\infty$. From the first part, it follows that $[\Gamma:{\rm Norm}_{\Gamma}(F)]<\infty$. Since $F$ is finite, we deduce that $[\Gamma: {\rm Stab}_{\Gamma}(F)]<\infty$, which gives $[\Gamma:\Gamma_0]<\infty,$ contradiction.

(3) If $g=a\gamma\in {\rm Comm}^{(1)}_G(\Gamma_0)$ with $a\in A_0^I$ and $\gamma\in\Gamma$, then $g\Gamma_0 g^{-1}\cap\Gamma_0\subset \gamma\Gamma_0 \gamma^{-1}\cap\Gamma_0\subset\Gamma_0$ are finite index subgroups.
This shows that $\gamma\in {\rm Comm}^{(1)}_\Gamma(\Gamma_0)=\Gamma_0$ and therefore $[\Gamma_0: a\Gamma_0 a^{-1}\cap\Gamma_0]<\infty$. If we assume by contradiction that $a\notin A_0$, then $a=(a_i)_{i\in F}\in A_0^I$, is supported on a finite non-empty set $F\subset I$ with $F\neq \{\Gamma_0\}$. Note that $a\Gamma_0 a^{-1}\cap \Gamma_0\subset {\rm Norm}_\Gamma(F)$. Since $[{\rm Norm}_\Gamma(F):{\rm Stab}_\Gamma(F)]<\infty$ and $[\Gamma_0: a\Gamma_0 a^{-1}\cap \Gamma_0]<\infty$, we derive that $[\Gamma_0: {\rm Stab}(F)\cap\Gamma_0]<\infty$. This shows that there exists $j\in\Gamma\setminus\Gamma_0$ such that $[\Gamma_0: j\Gamma_0 j^{-1}\cap\Gamma_0]<\infty$, which implies that $j\in {\rm Comm}_\Gamma^{(1)}(\Gamma_0)=\Gamma_0$, contradiction. This shows that $g\in A_0\rtimes\Gamma_0$, which ends the proof.
\hfill$\blacksquare$

\begin{proposition}\label{trivialamenableradical}
Let $G=A_0^I\rtimes_\sigma\Gamma$ be the coinduced group of $\Gamma_0\car A_0$ with $A_0$ an icc group and ${\rm Stab}_{\Gamma}(I)=1$. If $A_0$ has trivial amenable radical, then $G$ has trivial amenable radical.

\end{proposition}

{\it Proof.} 
Let $1\neq D\lhd G$ be an amenable normal subgroup. Since $G$ is icc, $D$ must be infinite. Fix $d\in D$ and $k\in A_0^I$. As $A_0^I$ is normal in $G$  we have that $[d,k]=(dkd^{-1})k^{-1}\in A_0^I$. Similarly using the normality of $D$ we also have $[d,k]=d(kd^{-1}k^{-1})\in D$ and hence $[D,A_0^I]\in D\cap A_0^I$. Assume by contradiction that $D\cap A_0^I=1$. The previous relation implies that $D\leqslant C_G(A_0^I)$. However, by Proposition \ref{relcomm1} we have $C_G(A_0^I)=1$ and hence $D = 1$, a contradiction. Therefore we have that $1\neq D\cap A_0^I\lhd A_0^I$ and there exists a finite set $F\subset I$ such that $1\neq D\cap A_0^F\lhd A_0^F$. Since $A_0^F$ is icc we deduce that  $D\cap A_0^F$ is infinite amenable. Next, we fix $i\in F$ and we see $D\cap A_0^F\lhd A_0^{F\setminus \{i\}}\times A_0^i$. Using the same argument as in the beginning of the proof we see that the commutator $[D\cap A_0^F,A_0^i]\leqslant D\cap A_0^i$. Since $D\cap A_0^i$ is amenable and normal in $A_0^i$ we derive that $D\cap A_0^i=1$ and hence $D\cap A_0^F<C_{A_0^F}(A_0^i)= A_0^{F\setminus \{i\}}$. Since this holds for all $i\in F$ we conclude that $D\cap A_0^F<\cap_{i\in F} A_0^{F\setminus\{i\}}=1$, which is a contradiction.   
\hfill$\blacksquare$

\subsection{Characters and automorphisms} In this subsection we describe the automorphisms of some of the groups considered in this paper including the class $\mathcal A$ from \cite{CD-AD20} and various coinduced groups associated with them (Theorems \ref{computeauta}, \ref{autcoinduceda}). At the end we also record some computations of the automorphisms of the generalized wreath products groups in $\mathscr{GW}(A_0)$, in the same vain with classical results of Houghton and others \cite{Ho62,Ho72, Has78}.
\begin{theorem}\label{computeauta} Let $\G$ be a property (T) group. Let $\G_1,\G_2,..., \G_n$ be isomorphic copies of $\G$. Let $\G\ca^{\rho_i} \G_i$ be the action by conjugation and let $\G\ca^\rho \G_1\ast \G_2\ast...\ast\G_n=A$ be the action induced by the canonical free product automorphism $\rho_g =\rho_g^1\ast \rho_g^2\ast...\ast \rho_g^n $. Denote by $G= A \rtimes_\rho \G$ the corresponding semi-direct product. Then the following hold:
\begin{enumerate}
    \item $G$ has finite abelianization $G/[G,G]$; in particular, ${\rm Char}(G)$ is finite.
    \item Assume in addition that $\G$ is icc, torsion free and bi-exact. Then for every $\theta\in {\rm Aut}(G)$ one can find  $b\in G$, $\mu \in \mathfrak S_n$ and $\alpha\in {\rm Aut} (\G)$ such that for every $a\in A$ and $g\in \G$ we have that  $\theta (ag)= b (\alpha_\mu (a) \alpha(g) b^{-1}$. Here $\alpha_\mu \in {\rm Aut}(A)$ is the automorphism given by $
    \alpha_\mu (a_1a_2a_3\cdots a_k)= \alpha_{\mu(i_1)}(a_1)\alpha_{\mu(i_2)}(a_2)\cdots \alpha_{\mu(i_k)}(a_k)$ whenever  $a_s\in A_{i_s}$ with $i_1\neq i_2\neq \cdots \neq i_k$. Also $\alpha_{\mu(i)}$ is the automorphism $\alpha$ seen from $\G_i$ to $\G_{\mu(i)}$.
\end{enumerate}
\end{theorem}

\begin{proof} 1) Since $\G$ has property (T) then the abelianization $\G/[\G,\G]$ is finite and hence $[\G,\G]\leqslant \G$ has finite index. Letting  
$A=\G_1\ast \G_2 \ast \cdots \ast \G_n$ this further implies the abelianization $A/[A,A]\cong \G_1/[\G_1,\G_1] \oplus \cdots\oplus  \G_n/[\G_n,\G_n]$ is also finite and hence $[A,A]\leqslant A$ has finite index. Altogether these show that the inclusions $G=A\rtimes \G\geqslant [A,A]\rtimes \G \geqslant [A,A]\rtimes [\G,\G]$ have finite index and since $[A,A]\rtimes [\G,\G]\leqslant [G,G]$ it follows that $[G,G]\leqslant G$ has finite index; thus $G/[G,G]$ is finite.

2) Fix $\theta\in {\rm Aut}(G)$. Notice that $G$ admits a canonical amalgamated free product decomposition $G= {\ast^n_{i=1}}_\G (\G_i \rtimes \G)$. Thus since for every $1\leq i\leq n$ the subgroup  $\theta(\G_i\rtimes \G)<G$ has property (T) one can find $g_i\in G$ and $1\leq \sigma(i)\leq n$ such that $\theta (\G_i \rtimes \G)^{g_i}\leqslant \G_{\sigma(i)}\rtimes \G$. Similarly, one can find $h_i\in G$ and $1\leq \rho(\sigma(i))\leq n$ such that $\theta^{-1} (\G_{\sigma(i)} \rtimes \G)^{h_i}\leqslant \G_{\rho(\sigma(i))}\rtimes \G$. Altogether these show that 
$\theta (\G_i \rtimes \G)^{g_i}\leqslant \G_{\sigma(i)} \rtimes \G\leqslant \theta ((\G_{\rho(\sigma(i))}\rtimes \G)^{h_i^{-1}})$. In particular, we must have that $\rho(\sigma(i))=i$ and $g_i\theta(h_i)\in \theta(\G_i \rtimes \G )$. This combined with the above shows that $\sigma,\rho\in \mathfrak S_n$ and also \begin{equation}\label{conjfactors}\theta (\G_i \rtimes \G)^{g_i}= \G_{\sigma (i)} \rtimes \G,\text{  for all } 1\leq i\leq n.\end{equation}    

In particular, for every $i\neq j$ we have that $\theta(\G)\leqslant  (\G_{\sigma (i)} \rtimes \G)^{g_i^{-1}} \cap (\G_{\sigma (j)} \rtimes \G)^{g_j^{-1}}$. 
Using \cite[Lemma 2.10]{CD-AD20} we have that ${\rm Comm}^{(1)}_{\G_i\rtimes \G} (\G)=\G$ and hence by \cite[Lemma 2.11]{CI17} we have  ${\rm Comm}^{(1)}_{G} (\G)=\G$. Letting $g_i=a_i s_i$ with $a_i \in A$ and $s_i\in \G$ and using a basic analysis of the reduced words we derive that $(\G_{\sigma (i)} \rtimes \G)^{g_i^{-1}} \cap (\G_{\sigma (j)} \rtimes \G)^{g_j^{-1}}\leqslant  \G^{g_i^{-1}} \cap  \G^{g_j^{-1}}= \G^{a_i^{-1}} \cap  \G^{a_j^{-1}} $. Thus $\theta(\G)\leqslant\G^{a_i^{-1}} \cap  \G^{a_j^{-1}} $. Basic calculations show that  $\G^{a_i^{-1}} \cap  \G^{a_j^{-1}}= \{g\in \G \,|\, \rho_{g}(a_j^{-1}a_i)=a_j^{-1}a_j\}$. However, since $\rho$ is given by the free product of inner automorphisms and since $\G_i$ are bi-exact and torsion free one can see that if  $a_j^{-1}a_j\neq 1$ then the stabilizer is amenable, contradicting $\theta(\G)\leqslant\G^{a_i^{-1}} \cap  \G^{a_j^{-1}} $. In conclusion, $a_j=a_i$ for all $i,j$  and thus in relation \eqref{conjfactors} one can pick $g=g_1=g_2=\cdots= g_n\in G$ and also $\theta(\G)^g\leqslant\G$. Reversing the roles of $\G$ and $\theta (\G)$, a similar argument shows that $\theta(\G)^g\geqslant\G$ and hence \begin{equation}\label{actinggroup2}\theta(\G)^g=\G.\end{equation}
Working with ${\rm ad} (g) \circ \theta$ instead of $\theta$ we can assume without any loss of generality that in relations \eqref{conjfactors}-\eqref{actinggroup2} we have  $g=1$. Thus there is $\alpha \in {\rm Aut}(\G)$ such that  \begin{equation}\label{actinggroup3}\theta_{|\G}=\alpha.\end{equation}  

Moreover, since $\G_i \rtimes \G,\G_{\sigma(i)} \rtimes \G\cong \G\times \G$ and $\G$ is icc, bi-exact, nonamenable, then using \eqref{conjfactors} there are $\theta_1,\theta_2\in {\rm Aut} (\G)$ such that $\theta= \theta_1\oplus \theta_2$ or $\theta= (\theta_1\oplus \theta_2)\circ {\rm flip} $.

This follows directly from the main product rigidity result in\cite{CdSS15}. However, for convenience, we include a more general and direct group theoretic argument from which the assertion follows. Namely, if $A_1, A_2, B_1, B_2$ are product indecomposable groups and $\theta:A_1\times A_2 \ra B_1\times B_2$ is a group isomorphism then one can find a permutation $\tau\in \mathfrak S_2$ such that $\theta(A_i)=B_{\tau(i)}$ for all $i=\overline{1,2}$. To see this let $C_i=\theta(A_i)$ and note that $C_1\times C_2=B_1\times B_2$. Let $\pi:B_1\times B_2\ra B_2$ be the canonical projection and observe that $B_2= \pi(C_1)\times \pi(C_2)$. As $B_2$ is product indecomposable there is $j\in \overline{1,2}$ so that $\pi(C_j)=1$ and $\pi(C_{\hat j})=B_2.$ But this further entails that $C_j\leqslant B_1$ and thus $B_1 = C_j \times C'_{\hat j}$ where $C'_{\hat j}\leqslant C_{\hat j}$. Since $B_1$ is product indecomposable  we must have   $C'_{\hat j}=1$ and thus $B_1= C_j=\theta(A_j)$. This together with $\pi(C_{\hat j})=B_2$ also imply $B_2=C_{\hat j}= \theta (A_{\hat j})$, which give the desired conclusion.

Now in both scenarios above, \eqref{actinggroup3} implies that $\theta_1= \theta_2=\alpha$ and also  $\theta(\G_{i})=\G_{\sigma(i)}$ for all $i\in \overline{1,n}$. In particular, we have \begin{equation}\label{autcore5}\theta_{|A}= \alpha_{\sigma(1)}\ast \alpha_{\sigma(2)}\ast\cdots \ast \alpha_{\sigma(n)}\in {\rm Aut}(A),\end{equation}

where $\alpha_{\sigma(i)}$ is the automorphism $\alpha$ seen from $\G_i$ to $\G_{\sigma(i)}$.

Altogether,  \eqref{autcore5} and \eqref{actinggroup3} give the conclusion.\end{proof}

\noindent \emph{Remarks}. The semi-direct product $G=(\G_1\ast \G_2 \ast\cdots \ast \G_n)\rtimes \G$ covered by the prior theorem could be slightly modified so that the ``permutation'' automorphisms are eliminiated from its automorphism group. For example one can pick  group (strict) inclusions  $\G_n > \G_{n-1}> \cdots > \G_1 > \G$ of torsion free, hyperbolic property (T) groups  such that $[\G_n:\G]<\infty$. The proof is the same as before with the only exception that the permutation part of the automorphism in equation \eqref{actinggroup3} must be trivial because the $\G_i$ are co-Hopfian groups \cite{Se99}. We leave the details to the reader.

\begin{theorem}\label{autcoinduceda}Let $\G$ be an icc, bi-exact, weakly amenable, property (T) group. Let $\G_1,\G_2,\ldots, \G_n$ be isomorphic copies of $\G$. Let $\G\ca^{\rho_i} \G_i$ be the action by conjugation and let $\G\ca^\rho \G_1\ast \G_2\ast\cdots \ast\G_n=:A$ the action induced by the canonical free product automorphism $\rho_g =\rho_g^1\ast \rho_g^2\ast\cdots\ast \rho_g^n $ for all $g\in\G$. Consider $\tilde \G=\G\times \G$ and assume $\G_0=\G<\tilde \G$ is diagonally embedded. Also let $I = \tilde\G/\G_0$ and consider the coinduced group $G=A^I\rtimes \tilde \G$. Then we have the following description of its automorphism group. Writing $I = \{h \G_0\,|\, h\in \G\times 1\} $ for every $\theta \in {\rm Aut}(G)$ there exists $b\in G$, a permutation $\mu \in \mathfrak S_n$, and $\alpha \in {\rm Aut}(\G)$ such that for every $(x_{h\G_0})_{h\in \G\times 1}\in A^I$ and $g\in \tilde \G$ we have  \begin{equation}\label{descrautcoinduced}\theta((x_{h\G_0})_{h\in \G\times 1} g) = b (\alpha_\mu  (x_{\alpha(h^{-1})\G_0}))_{h\in \G\times 1} (\alpha \oplus \alpha)(g) b^{-1}.\end{equation}

Here, $\alpha_\mu \in {\rm Aut}(A)$ is the automorphism given by $
    \alpha_\mu (a_1a_2a_3\cdots a_k)= \alpha_{\mu(i_1)}(a_1)\alpha_{\mu(i_2)}(a_2)\cdots \alpha_{\mu(i_k)}(a_k)$ whenever  $a_s\in A_{i_s}$ with $i_1\neq i_2\neq \cdots \neq i_k$. Also $\alpha_{\mu(i)}$ is the automorphism $\alpha$ seen from $\G_i$ to $\G_{\mu(i)}$. \end{theorem} 


\begin{proof} Fix $\theta\in {\rm Aut}(G)$. First we prove that $\theta(A^I)=A^I$. Let $\pi: G \ra \tilde \G$ be the canonical surjection. Thus $\pi\circ \theta(A^I)$ is a normal subgroup of $\tilde\G$, generated be the commuting subgroups $ \pi\circ \theta(A^i)$ for $i\in I$. Since $\G$ is bi-exact and $I$ is infinite, there is $j\in I$ so that $ \pi\circ \theta(A^j)$ is amenable.  Let $B<\pi\circ \theta(A^I)$ be its amenable radical and notice that $ \pi\circ \theta(A^j)<B$. Since $\pi\circ \theta(A^I)$ is normal in $\tilde \G$, then so is $B$. However, since $\G$ is icc, bi-exact and weakly amenable, then by \cite[Corollary 0.2]{CSU11} we have that $\G$ does not have any non-trivial amenable normal subgroup and thus, $B =1$. Therefore,  $\pi\circ \theta(A^j)=1$, which implies that $\theta(A^j)< A^I$. Since $A^I$ is normal we have that $\theta(A^{gj})= \theta(g) \theta(A^j) \theta (g^{-1})< A^I$ and as $\tilde \G\ca I$ is transitive we conclude that $\theta(A^I)\leqslant A^I$. Similarly, we get $\theta(A^I)\geqslant A^I$ which proves the claim.

\vskip 0.03in
Notice that the semi-direct product feature of the group also implies the existence of two maps $\delta:\tilde \G \ra \tilde \G$ and $a: \tilde \G \ra A^I$ such that $\theta(g)= a_g \delta(g)$ for all $g\in \tilde \G$. Since $\theta$ is an automorphism one can easily see that $\delta$ is also an automorphism of $\tilde \G$. Moreover, the homomorphism property of $\theta$ implies that $a_g \sigma_{\delta(g)}(a_h)=a_{gh}$ for all $g,h\in \tilde \Gamma$. In other words, $a$ is a $1$-cocycle for the coinduced action $\tilde \G\ca^\sigma A^I$. Using the cocycle superigidity theorem for coinduced action of property (T) groups from \cite[Theorem 3.1]{Dr15}  one can find a unitary $v\in \mathscr U(\cL(A^I))$ and a character $\eta$ on $\tilde \G$ such that $u_{a_g}= \eta(g) v  \sigma_{\delta(g)}(v^*)$. Here, $\{u_h \,:\, h\in A^I\}$ are the canonical unitaries of $\cL(A^I)$. This implies that $u_{a_g \delta(g)} v u_{\delta(g)^{-1}} = \eta(g) v$ for all $g\in \tilde \G$. Let $v=\sum_h v_h u_h $ be its Fourier decomposition and notice that for all $h$ with $v_h \neq 0$ we have that $\mathcal O_h:=\{a_g \delta(g) h \delta(g)^{-1}, g\in \tilde \G\}$ is finite.  Thus, there is a finite index subgroup $\tilde \G_h<\tilde \G$  such that $a_g \delta(g) h \delta(g)^{-1}=h$ for all $g\in \tilde \G_h$. This further implies that $a_g= h \sigma_{\delta(g)} (h^{-1}) $ for all $g\in \tilde \G_h$. Fix $h_1$, $h_2$ such that $\mathcal O_{h_1}$ and $\mathcal O_{h_2}$ are finite. Thus, for every $g\in \tilde \G_{h_1}\cap \tilde \G_{h_2}$ we have that  $h_1 \sigma_{\delta(g)} (h^{-1}_1)= a_g= h_2 \sigma_{\delta(g)} (h^{-1}_2) $ and hence $h_2^{-1}h_1=\sigma_{\delta(g)}(h_2^{-1}h_1)$. Since $\tilde\G_{h_1}, \tilde\G_{h_2}\leqslant \tilde \G$ have finite index then so is $\tilde \G_{h_1}\cap \tilde \G_2\leqslant \tilde \G$. This implies that $h_2^{-1}h_1=1$, and hence, $h_1=h_2$. Altogether these show that there is $h\in A^I$ such that $v=u_h$, and hence, $a_g= h \sigma_{\delta(g)} (h^{-1})$ for all $g\in \tilde \G$. Thus, replacing $\theta$ by $ad(h^{-1})\circ \theta$ we can assume that \begin{equation}\label{autactgroup}\theta(g)=\delta(g)\text{ for all }g\in \tilde \G.\end{equation} 

Moreover, since $\G$ is icc and bi-exact, then using the same argument as in the proof of the previous result, one can find  $\delta_1,\delta_2\in {\rm Aut} (\G)$ such that     $\delta(g)=(\delta_1 (g_1), \delta_2(g_2))$ or $\delta(g)=(\delta_1 (g_2), \delta_2(g_1))$ for every $g=(g_1,g_2)\in \tilde \G$. 

\vskip 0.03in
Throughout this proof we denote by $i_o=\Gamma_0 \in I$ the trivial coset. Next, we show that there are automorphisms $\theta_i: A \ra A$ and a permutation $\tau:I \ra I$ with $\tau(i_0)=i_0$ such that 
\begin{equation}\label{autcore2}\theta((x_i)_{i\in I}) )=(\theta_{\tau^{-1}(i)}(x_{\tau^{-1}(i)}))_{i\in I}.\end{equation}

Fix $i\in I$. Since $\theta(A^i)$ is finitely generated, there is a finite subset $J\subset I$ that is minimal with respect to the set theoretic inclusion such that $C:=\theta(A^i)\leqslant A^J$. Assume by contradiction $|J|\geqslant 2$.  


Since  $C=\theta(A^i) \leqslant A^J \leqslant \theta (A^i) \oplus \theta (A^{I\setminus\{i\}})$ there is a subgroup $S \leqslant  A^J\cap \theta (A^{I\setminus\{i\}})$ such that $A^J=C\oplus S$. Since $|J|\geqslant 2$ and $C$ is non-amenable and bi-exact it follows that $S$ is non-amenable. Moreover, as $A^J$ is icc so is $S$. Since $A^J$ is bi-exact relative to the family $\{ A^{J\setminus \{j\}}, j\in J \}$ and $C$ is non-amenable then one can find $k\in J$ and a finite index subgroup $S_0\leqslant S$ so that $S_0\leqslant A^{J\setminus \{k\}}$. Using the icc conditions we get that $C= C_{A^J}(S) = C_{A^J}(S_0)\geqslant C_{A^J}(A^{J\setminus \{k\}})=A^k$. Hence there is $D\leqslant A^J\setminus \{k\}$ so that  $C= D\oplus A^k$ and since $C$ is icc bi-exact we conclude that $C=A^k$ which again contradicts the minimality  of $|J|$.

 In conclusion,  $J$ must be a singleton and we have proved that for every $i\in I$ there is $\tau(i)\in I$ such that $\theta(A^i)\leqslant A^{\tau(i)}$.  As $\theta$ is an automorphism, then one can see that $\tau$ must be a bijection and also $\theta(A^i)= A^{\tau(i)}$. Conjugating by an appropriate group element of $\tilde \G$ and using the transitivity of the action of $\tilde \Gamma$ on $I$, we can also assume that $\tau(i_0)=i_0$. Altogether, these clearly give \eqref{autcore2}.  

 \vskip 0.03in   

Next we analyze the cocycle $c: \tilde \G \times I \ra  \G_0$. Notice that $I = \{h \G_0\,|\, h\in \G\times 1\}$ and consider the section $\phi: I \ra \Gamma$ given by $\phi(h \G_0)=h$ for all $h\in \G\times 1$. Thus for every $g=(g_1,g_2)\in \tilde \G $ and  $i= h  \G_0\in I $ we see that  $c(g,i)= \phi( g i )^{-1} g \phi(i) = \phi((g_1h,g_2) \G_0 )^{-1}(g_1h,g_2)= \phi((g_1h g_2^{-1},1) \G_0 )^{-1}(g_1h,g_2)=((g_1h g_2^{-1})^{-1},1)(g_1h,g_2)=(g_2,g_2)$.

Using the identification  $I = \{h \G\,|\, h\in \G\times 1\}$ we also notice that there exists a permutation $\sigma: \G\ra \G$ with $\sigma (1)=1$ such that $\tau(h  \G_0)= \sigma(h) \G_0$ for all $h\in \G\times 1$.     

Using relations \eqref{autcore2} and \eqref{autactgroup} we have the equivariance condition $\theta \circ \sigma_g= \sigma_{\delta( g)} \circ\theta $ for all $g\in \tilde \G$. 
Using the definition of coinduced actions this implies that for all $(x_i)_{i\in I}\in A^I$ and $g\in \tilde \G$ we have \begin{equation}\label{equivariance2}\theta_{\tau^{-1}(i)}(\rho_{c(g,g \tau^{-1}(i))}(x_{g^{-1}\tau^{-1}(i)}))=\rho_{c(\delta(g), \delta(g)^{-1}i)}(\theta_{\tau^{-1}(\delta(g)^{-1}i)}(x_{\tau^{-1}(\delta(g)^{-1}i)})).\end{equation}

In particular, for all $i\in I$, $g\in \tilde \G$ we have 
$g^{-1}\tau^{-1}(i)=\tau^{-1}(\delta(g)^{-1}i)
$ and hence  \begin{equation}\label{baseset1}\tau(g^{-1}\tau^{-1}(i))=\delta(g)^{-1}i.\end{equation} Letting $i=i_0$ and $g\in \G_0$ we get that $\delta(\G_0)=\G_0$. Combining this with the above, we see that either \begin{equation}\label{autactinggroup}\begin{split}&\delta=\delta_1\oplus \delta_1\text{ for some }\delta_1\in {\rm Aut}(\G)\text{, or }\\&
\delta= (\delta_1\oplus \delta_1) \circ {\rm  flip}_{\tilde \G}.\end{split}
\end{equation}

Assume first $\delta=\delta_1\oplus \delta_1$ for some $\delta_1\in {\rm Aut}(\G)$. Then relation \eqref{baseset1} implies that for all  $g,h\in \G \times 1$, we have $\sigma(g^{-1}\sigma^{-1}(h)) \G_0=\tau(g^{-1}\tau^{-1}(h \G_0)))=\delta_1(g)^{-1}(h\G_0)=(\delta_1(g)^{-1}h)\G_0$ and hence $\sigma(g^{-1}\sigma^{-1}(h))=(\delta_1(g)^{-1}h)$. Letting $h=1$ this gives that \begin{equation}\label{basesetaction}\sigma(g)= \delta_1(g)\text{ for all }g\in \G\times 1.\end{equation}


To this end observe that using relations \eqref{baseset1} in \eqref{equivariance2} we get that for every $g=(g_1,g_2)\in \tilde \G$ and $i\in I$ we have  
\begin{equation}\label{equivariance4}\theta_{\tau^{-1}(i)}\circ \rho_{(g_2,g_2) }=\theta_{\tau^{-1}(i)}\circ \rho_{c(g,g \tau^{-1}(i))}=\rho_{c(\delta(g), \delta(g)^{-1}i)}\circ \theta_{\tau^{-1}(\delta(g)^{-1}i)}=\rho_{(\delta_1(g_2), \delta_1(g_2))}\circ \theta_{\tau^{-1}(\delta(g)^{-1}i)}.
\end{equation} 
Letting $g_2=1$ in this relation we see that $\theta_{\tau^{-1}(i)}= \theta_{\tau^{-1}(\delta_1(g_1)^{-1}i)}$, for all $g_1\in \G$. This however implies that there exists $\psi\in {\rm Aut}( A)$ such that \begin{equation}\label{coreaut2}\theta_i=\psi\text{ for all }i\in I,\end{equation} and therefore the previous relation entails that for all $g\in \G_0$ we have \begin{equation}\label{equivariance3}
   \psi \circ \rho_{g }=\rho_{\delta_1(g)}\circ \psi. 
\end{equation}

Next, we describe $\psi$. Since $A=\G_1\ast ...\ast \G_n$ and $\psi(\G_i)$ are icc and have property (T) then using Kurosh theorem one can find a map $\mu: \overline{1,n}\ra  \overline{1,n}$ and elements $a_1,...,a_n\in A$ such that $\psi(\G_i)^{a_i}\leqslant \G_{\mu(i)}$ for all $i=\overline{1,n}$. Proceeding in a similar manner, one can find map $\nu: \overline{1,n}\ra  \overline{1,n}$ and some elements $b_1,...,b_n\in A$ such that $\psi^{-1}(\G_{\mu(i)}^{b_i})\leqslant \G_{\nu(\mu(i))}$ for all $i\in \overline{1,n}$.  Altogether these show that  $\psi(\G_i)^{a_i b^{-1}_i}\leqslant \G_{\mu(i)}^{b_i}\leqslant \psi(\G_{\nu(\mu(i))})$. This however entails that $\nu(\mu(i))=i$ and $a_i b^{-1}_i\in \psi(\G_i)$ and  $\psi(\G_i) = \G_{\mu(i)}^{b_i}$ for all $i$; in particular $\mu\in \mathfrak S_n$.  Thus one can find $\psi_i \in {\rm Aut} (\G)$ with $i\in \overline{1,n}$ such that $\psi = ({\rm ad}(b_1)\circ \psi_{\mu(1)})\ast ({\rm ad}(b_2)\circ\psi_{\mu(2)})\ast\cdots \ast ({\rm ad}(b_n)\circ\psi_{\mu(n)})$. Now fix $k\in \overline{1,n}$ and  $a\in \G_k$ arbitrary. Then relation \eqref{equivariance3} implies that $b_k \psi_{\mu(k)}(gag^{-1})b^{-1}_k= \psi \circ \rho_{g }(a)=\rho_{\delta_1(g)}\circ \psi(a)=\delta_1(g) b_k \psi_{\mu(k)} (a)b^{-1}_k\delta_1(g)^{-1}$ for all $g\in \G_{k}$. Therefore $b^{-1}_k\delta_1(g)^{-1}b_k \psi_{\mu(k)}(g)\in C_A(\G_{\mu(k)})=1$ and hence 
$\psi_{\mu(k)}={\rm ad} (b^{-1}_k)\circ\delta_1$. Consequently, as $k$ was arbitrary, we get  \begin{equation}\label{basegroupfreeproduct}\psi = (\delta_1)_\mu.\end{equation}

Observe that \eqref{basegroupfreeproduct}, \eqref{coreaut2}, \eqref{basesetaction}, and \eqref{autactinggroup} combined with the prior relations give \eqref{descrautcoinduced} for $\alpha=\delta_1$.

If  $\delta=(\delta_1\oplus \delta_1) \circ {\rm  flip}_{\tilde \G}$ then  \eqref{equivariance4} implies $\theta_{\tau^{-1}(i)}\circ \rho_{(g_2,g_2) }=\rho_{(\delta_1(g_1), \delta_1(g_1))}\circ \theta_{\tau^{-1}(\delta_1(g_2^{-1}g_1)i)}$. 
Letting $g_2=1$ we see that $\theta_{\tau^{-1}(i)} =\rho_{(\delta_1(g_1), \delta_1(g_1))}\circ \theta_{\tau^{-1}(\delta_1(g_1) i)}$ for all $g_1\in \G\times 1$ and $i\in I$. Also, letting $g_1=1$ we get  $\theta_{\tau^{-1}(i)}\circ \rho_{(g_2,g_2)}=\theta_{\tau^{-1}(\delta_1(g_2^{-1})i)}$ for all $g_2\in \G\times 1$ and $i\in I$. Using these last three relations we see that 
$\theta_{\tau^{-1}(\delta_1(g_2^{-1})i)}=\theta_{\tau^{-1}(i)}\circ \rho_{(g_2,g_2)}=\rho_{(\delta_1(g_1), \delta_1(g_1))}\circ \theta_{\tau^{-1}(\delta_1(g_2^{-1}g_1)i)}= \theta_{\tau^{-1}( \delta_1(g_1^{-1}g_2^{-1}g_1)i)}$ for all $g_1,g_2\in \G\times 1$ and all $i \in I$. Therefore $\theta_{\tau^{-1}(i)}\circ \rho_{(g_2,g_2)}=\theta_{\tau^{-1}(\delta_1(g_2^{-1})i)}= \theta_{\tau^{-1}( \delta_1(g_1^{-1}g_2^{-1}g_1)i)}= \theta_{\tau^{-1}(i)}\circ \rho_{(g_1g_2g_1^{-1},g_1g_2 g_1^{-1})}$ for all $g_1,g_2 \in \Gamma \times 1$ and $i\in I$. This however entails that  $\rho_{g_2}= \rho_{g_1g_2g_1^{-1}}$ for all $g_1,g_2\in \Gamma \times 1$, which is a contradiction.  Hence $\delta\neq (\delta_1\oplus \delta_1) \circ {\rm  flip}_{\tilde \G}$.\end{proof}

We end the subsection describing the automorphisms of groups in class $\mathscr{GW}_1(A_0)$.  Notice that, in contrast to the general $KIB^*$-description of automorphisms of generalized wreath-products from \cite{Ho62,Ho72,Has78} in our situation there are no automorphisms of type $I$. In essence, this is due to the fact that the acting group has a normal infinite prop (T) subgroup.

\begin{theorem}\label{autgenbernoulli} Let $G= A_0\wr_I \G\in\mathscr{GW}_1(A_0)$ and let $\G_0={\rm  Stab}_\G(i)$.  Then the following hold:
\begin{enumerate}
    \item  For every $\theta \in {\rm Aut}(G)$, one can find $b\in G$, $\alpha\in {\rm Aut}(A_0)$ and $\beta\in {\rm Aut}_{\G_0} (\G)$ such that for every $a=(a_{h \G_0})_{h\Gamma_0\in \G/\G_0}\in A^I$ and $g\in \G$ we have that $\theta(a g)= b (\alpha(a_{\beta(h)\G_0 }))_{h\Gamma_0\in \G/\G_0}   \beta(g)b^{-1}$.   
    \item The abelianization of $G$ satisfies $G/[G,G]\leqslant A_0/[A_0,A_0]\times \G/[\G,\G]$.
\end{enumerate}
Above, ${\rm Aut}_{\G_0} (\G)$ denotes the subgroup of all automorphisms in ${\rm Aut} (\G)$ that preserve $\G_0$ as a set. 
\end{theorem}

\begin{proof} 1) The proof of the first part is very similar to the proof of Theorem \ref{autcoinduceda} or it can be obtained alternatively using the methods from \cite{Ho62,Ho72,Has78}. The details are left  to the reader.

2) We write $G=A_0^I\rtimes_{\sigma}\Gamma$. Notice that $[A^I_0,A^I_0]=[A_0,A_0]^I$. Also we have that the commutator is generated by $[G,G]=\langle [\G,\G], [A_0,A_0]^I, \sigma_g(a)a^{-1},\text{ for all }g\in \G,  a\in A_0^I\rangle$. Consider the subgroup $H= A_0^I\cap [G,G]$ and notice that  $[A_0,A_0]^I\lhd H\lhd A_0^I$. From the prior relations it follows that $H = \langle [A_0,A_0]^I, \sigma_g(a)a^{-1},\text{ for all }g\in \G,  a\in A_0^I\rangle$ and hence $[G,G]= H\rtimes [\G,\G]$. 

Next, we show that for any $i\in I$, we have
\begin{equation}\label{cosetscorecommut}A_0^i H= A^I.\end{equation} Towards this for every $a\in A_0$ and $j\neq k\in I$ consider the element $\delta^{j,k}(a)\in A^I$ defined as $\delta^{k,j}(a)(s)= a$ if $s=j$, $\delta^{k,j}(a)(s)= a^{-1}$ if $s=k$ and $\delta^{k,j}(a)(s)= 1$, otherwise. Also denote by $\delta^{j}(a)\in A^I$ the element defined as $\delta^{j}(a)(s)= a$ if $s=j$ and $\delta^{j}(a)(s)= 1$, otherwise. Since  $\sigma_g(a)a^{-1}\in H$ for all $g\in \G,  a\in A_0^I$ and the action $G\ca I$ is transitive then $\delta^{j,k}(a)\in H$ for all $j\neq k\in I$ and $a\in A_0$. Now fix $a=(a_j)_j\in A^I$ and let $J=\{j_1, \ldots ,j_s\}\subset I$ be its support. If $i\notin J$ then we have \begin{equation*}a =\delta^{j_1,j_2}(a_{j_1})\delta^{j_2,j_3}(a_{j_1}a_{j_2})\cdots \delta^{j_{s-1},j_s}(a_{j_1}a_{j_2}\cdots a_{j_{s-1}} ) \delta^{j_{s},i}(a_{j_1}a_{j_2}\cdots a_{j_{s}} ) \delta^i(a_{j_1}a_{j_2}\cdots a_{j_{s}}) \in H A_0^i. \end{equation*} 
 If $i\in J$ then after permuting the elements of $J$ we can assume without any loss of generality that $j_s=i$ and a similar computation shows that  
 \begin{equation*}a =\delta^{j_1,j_2}(a_{j_1})\delta^{j_2,j_3}(a_{j_1}a_{j_2})\cdots \delta^{j_{s-1},j_s}(a_{j_1}a_{j_2}\cdots a_{j_{s-1}} )  \delta^i(a_{j_1}a_{j_2}\cdots a_{j_{s}}) \in H A_0^i. \end{equation*} 
 Altogether, these show \eqref{cosetscorecommut}. Finally, since  $A^I/H= A_0^i H/H= A_0^i/A_0^i\cap H$ we can see that $G/[G,G] \cong (A^i_0/ A^i_0\cap H) \times (\G/[\G,\G])$ and since $[A^i_0,A_0^i]\leqslant A_0^i\cap H$ the conclusion follows.\end{proof}

\section{A Class of Automorphisms of Generalized Wreath Product von Neumann Algebras}\label{Section:HoughtonAut}

In the sixties, C. Houghton was able to completely describe the automorphisms of wreath product groups $G=A \wr B$, \cite{Ho62,Ho72}. Specifically, he was able to show that typically an automorphism of $G$ admits a so-called $KIB^*$ decomposition, i.e. it is a composition of three automorphisms arising from the base group $A$,  the unrestricted core $A^B$ and the acting group $B$. Similar results regarding generalized wreath products associated with transitive actions were obtained subsequently \cite{Has78}.

In this section we highlight a canonical class of $\ast$-automorphisms for wreath product von Neumann algebras which are natural analogues of the ones arising from the automorphisms of the base group in the core of the wreath product (i.e.\ type $K$). Let $\mathcal A_0$ be a finite von Neumann algebra and let $\Gamma \ca I$ be an action of a countable group $\Gamma$ on a countable set $I$. Form the infinite tensor product von Neumann algebra $\bar\otimes_I \mathcal A_0$ ``along $I$'' and consider the action $\sigma: \Gamma \ca {\rm Aut}(\bar\otimes_I \mathcal A_0)$ by Bernoulli shift $\sigma_g(\otimes_{i\in I} x_i)= \otimes_{i\in I} x_{g^{-1}i}$ for all $g\in \Gamma$ and $x_i\in \mathcal A_0$. Here, we denoted by $\otimes_{i\in I}x_i$ a finitely supported tensor from $\bar\otimes_I \mathcal A_0$. Then the corresponding crossed product von Neumann algebra $(\bar \otimes_I \mathcal A_0) \rtimes \Gamma$ is called a generalized wreath product algebra and it is denoted by $\mathcal A_0 \wr_I \Gamma$. Now  fix $\theta$ any $\ast$-automorphism of $\mathcal A_0$. This canonically gives rise to a tensor product automorphism $\otimes_I \theta$ of $\bar\otimes_I \mathcal A_0$ by letting $\otimes_I \theta (\otimes_{i\in I}  x_i)=\otimes_{i\in I} \theta(x_i)$ for all $\otimes_{i\in I} x_i\in \bar\otimes_I \mathcal A_0$. It is immediate from the definition that $(\otimes_I \theta)\circ\sigma_g=\sigma_g\circ (\otimes_I \theta)$ and hence $\otimes_I \theta$ extends to an automorphism $\Phi_\theta\in {\rm Aut}(\mathcal A_0 \wr_I \Gamma)$ that is identity on the subalgebra $\mathcal L(\Gamma)$. When $\theta={\rm ad}(v)$ for some unitary $v\in \mathcal A_0$, this will be denoted by $\Phi_v$ throughout this document.

Let $A_0\wr_I \Gamma$ be a generalized wreath product group. Note that the corresponding group von Neumann algebra $\mathcal L(A_0\wr_I \Gamma)$ can be canonically identified with the generalized wreath product algebra $\mathcal L(A_0) \wr_I \Gamma$ as introduced above. In this context, under composition, the automorphisms $\Phi_\theta$ for $\theta\in {\rm Aut}(\mathcal L(A_0))$ form a subgroup of outer automorphisms of $\mathcal L( A_0 \wr_I \Gamma)$.

\begin{proposition}\label{Proposition:notunitary} Assume $\Gamma\ca I$ satisfies $ {\rm Stab}_{\Gamma}(I)=1$ and $I$ is infinite. If $\theta \in {\rm Aut}( \cL(A_0 ))$ is non-trivial, then $\Phi_\theta\neq {\rm ad}(u)$ for any $u\in \mathcal L(A_0 \wr_I \Gamma)$.
\end{proposition}

{\it Proof.} Suppose there exists $u\in \mathscr U(\mathcal L(A_0\wr_I \Gamma))$ so that $\Phi_\theta={\rm ad}(u)$. Fix $i\in I$ and notice that for every $x\in \cL (A_0^i)$ we have $\Phi_\theta (x)u=ux$. Using the Fourier decomposition $u= \sum_{g\in Q} a_g u_g$ with $a_g \in \cL(A_0^I)$ we get that $\Phi_\theta(x)a_g =a_g \sigma_g(x)$ for all $x\in \cL(A_0^i)$. Now assume that $g\notin {\rm Stab}_{\Gamma}(i)$ and fix $\varepsilon >0$. By using Kaplanski's Density Theorem we can approximate $\|a_g -\sum_{k\in K}  c_k^g d^g_k\|_2<\varepsilon$ where $K$ is finite, $c_k \in \cL(A_0^i)$, $d_k^g\in \cL(A_0^{I\setminus {i}})$, and $\|\sum_{k\in K}  c_k^g d^g_k\|_\infty \leq 1$. Using the previous intertwining condition together with these estimates we see that for every $x\in \mathscr U(\cL(A_0^i))$ we have the following inequality 

\begin{equation}\begin{split}\label{noactinggroup}
\| E_{\cL(A_0^i)}(a_g a_g^*)\|_2&=\|\Phi_\theta (x) E_{\cL(A_0^i)}(a_g a_g^*)\|_2=\| E_{\cL(A_0^i)}(a_g \sigma_g(x)a_g^*) \|_2 \\
&\leq 2\varepsilon + \sum_{k,l\in K}\| c^g_k (c^g_l)^* E_{\cL(A_0^i)}(d^g_k \sigma_g(x)(d^g_l)^*)\|_2  \\
&\leq 2\varepsilon + \sum_{k,l\in K}\| c^g_k (c^g_l)^*\|_2 |\tau ( \sigma_g(x)(d^g_l)^*d^g_k)|.
\end{split}\end{equation} 

Now, since $\cL(A_0^i)$ is diffuse there is a sequence of unitaries $(x_n)_n\subseteq \cL(A_0^i)$ that converges to $0$ in $WOT$. In particular, $|\tau ( \sigma_g(x_n)(d^g_l)^*d^g_k)|\ra 0$ as $n\ra \infty$ for all $k,l\in K$ and using it the inequality \eqref{noactinggroup} together with the fact that $K$ is finite we get that $\| E_{\cL(A_0^i)}(a_g a_g^*)\|_2\leq 2\varepsilon$. Since  $\varepsilon>0$ was arbitrary we infer that  $E_{\cL(A_0^i)}(a_g a_g^*)=0$ and hence $a_g=0$. In conclusion, we obtained that for every $i\in I$ and every  $g\notin \rm Stab_{\Gamma}(i)$ we have that $a_g=0$. Thus, $a_g \neq 0$ only if $g\in \cap_{i\in I}{\rm Stab}_{\Gamma}(i)=1$ and hence $u\in \cL(A_0^I)$.

Next, we show that $\theta$ is an inner automorphism, i.e.\ there is a unitary $v\in \cL(A_0)$ such that $\theta ={\rm ad}(v)$. Fix $\varepsilon>0$. Since $u\in \cL(A_0^I)$ by Kaplanski's Density Theorem one can find a finite set $G\subset I$ and a unitary $u_G\in \cL(A_0^G)$ such that $\|u-u_G\|_2<\varepsilon$. Therefore for every unitary $x\in \cL(A_0^{I\setminus G})$ we have that \begin{equation}\begin{split}\|\Phi_\theta (x)-x\|_2&=\|(\Phi_\theta (x)-x) u_G\|_2=\|\Phi_\theta (x)u_G- u_Gx \|_2\\
&=\|\Phi_\theta (x)u_G -\Phi_\theta (x)u + u x- u_Gx  \|_2\leq 2\| u-u_G\|_2< 2\varepsilon. \end{split}\end{equation} 

Thus, for every $j\in I\setminus G$ and every unitary $x \in \cL(A_0^j)$ we have that $\|\theta(x) -x\|^2_2\leqslant 4\varepsilon ^2$. This implies that $2-4\varepsilon^2\leq 2 \mathfrak Re\tau(x^*\theta (x))$ for all unitaries $x \in \cL(A_0^j)$. Consider $\mathfrak K\subset \cL(A_0^j)$ the SOT-closure of the convex hull of the set $\{ x^*\theta (x)\,:\, x\in \mathscr U(\cL(A_0^j))\}$ and let $\xi\in \mathfrak K$ be the unique element of minimal $\|\cdot \|_2 $-length. This clearly satisfies \begin{equation}\label{itertwiningtheta}\xi\theta(x) =x\xi\text{ for all }x\in \cL(A_0^j).\end{equation} 
Moreover, one can see the previous relations imply that $\|\xi-1\|_2<\varepsilon$, so choosing conveniently small $\varepsilon> 0$, we have $\xi\neq 0$. Now, using the intertwining relation \eqref{itertwiningtheta} we get that $\xi\xi^*\in \cL(A_0^j)'\cap \cL(A_0^j)=\mathbb C 1$, and hence, $\xi$ is a non-trivial scalar multiple of a unitary $v\in \cL(A_0^j)$. So cancelling the scalar multiple we have that $\theta ={\rm  ad}(v)$ as claimed.

Thus, for every finite subset $F\subset I$ we have $v_F xv_F^*= uxu^*$, for all $x \in \cL(A_0^F)$; here, we denoted by $v_F= \otimes_F v$. The previous relation and Proposition \ref{relcomm1} imply that $u^* v_F \in \mathcal L(A_0^F)'  \cap \mathcal L(A_0\wr_I \Gamma)\subset \mathcal L ( A_0^{\widehat F}) \rtimes {\rm Norm}_{\Gamma}(F)$. Since $u,v_F\in \mathcal L(A_0^I)$, it follows that there is a unitary $w_F\in \mathcal L ( A_0^{\widehat F}) $ so that $u= v_F w_F$. Next, we show this entails that $v$ is a scalar, which implies that $\theta =1$, a contradiction. To this end, fix $\varepsilon>0$ and let $u_K\subset \mathcal L( A_0^K)$ be a unitary so that $\|u-u_K\|_2< \varepsilon.$
Also consider two arbitrary finite subsets $K \subset G_1 \subseteq G_2 \subset I$ and note that $u =v_{G_i} w_{G_i}$ for $i=1,2$. Hence, 
\begin{equation}\label{aprox1}
\| \tau(w_{G_i}) v_{G_i} - u_K\|_2= \| E_{\mathcal L(A_0^{G_i}) }( v_{G_i} w_{G_i} - u_K\|_2= \| E_{\mathcal L(A_0^{G_i}) }( u - u_K)\|_2\leq \|u - u_K\|_2\leq \varepsilon.
\end{equation} 

  This inequality further implies that $\| \tau(w_{G_1}) v_{G_1}  - \tau(w_{G_2})v_{G_2}\|_2\leq 2\varepsilon$. As $G_1\subseteq G_2$ this further shows that \begin{equation}\label{ineq1}\begin{split}
|\tau(w_{G_1})|^2+ |\tau(w_{G_2})|^2&\leq 4\varepsilon^2 +2 Re\{\tau (w_{G_1}) \overline{\tau (w_{G_2})}\tau (v)^{|G_2\setminus G_1|}\}\\
&\leq 4\varepsilon^2 +2 |\tau (w_{G_1})| |\tau (w_{G_2})| |\tau (v)|^{|G_2\setminus G_1|}.
\end{split}
\end{equation} 
However, from \eqref{aprox1} we have $1-\varepsilon \leq |\tau(w_{G_i})|\leq 1$ and combining with \eqref{ineq1} we see that $1-2 \varepsilon -\varepsilon^2\leq |\tau(v)|^{|G_2\setminus G_1|}$. As this holds for every $G_1\subseteq G_2$ then letting $\varepsilon\searrow 0$ and $|G_2\setminus G_1|\nearrow \infty$ we must have that $|\tau(v)|=1$, and hence, $v\in \mathbb T 1$, which proves our claim and finishes the proof.   
\hfill$\blacksquare$




Recall that for any countable group $G$, the maps $\Psi_{\omega,\delta}$ defined in introduction  realize a copy of the semi-direct product group ${\rm Char}(G)\rtimes {\rm Aut}(G)$ inside ${\rm Aut}(\cL(G))$. The following proposition shows that most of the automorphisms $\Phi_v$ introduced in this section are not contained in ${\rm Char}(G)\rtimes {\rm Aut}(G)$. The proof is standard and we leave it to the reader.

\begin{proposition} Let $G=A_0\wr_I \G$, where $A_0$ is an icc group. If $v\in \mathscr U(L(A_0))$ such that $\Phi_v=\Psi_{\eta,\delta}$, then $\eta =1$ and $v=u_k$ for some $k\in A_0$. 

\end{proposition}
\section{Popa's Intertwining Techniques}


More than fifteen years ago, Sorin Popa has introduced  in \cite[Theorem 2.1 and Corollary 2.3]{Po03} powerful analytic methods for identifying intertwiners between arbitrary subalgebras of tracial von Neumann algebras. Now this is termed in the literature  \emph{Popa's intertwining-by-bimodules technique} and is a highly instrumental tool in the classification of von Neumann algebras program via deformation/rigidity theory.

\begin{theorem}[\cite{Po03}] \label{corner} Let $(\mathcal M,\tau)$ be a separable tracial von Neumann algebra and let $\mathcal P, \mathcal Q\subseteq \mathcal M$ be (not necessarily unital) von Neumann subalgebras. 
Then the following are equivalent:
\begin{enumerate}
\item There exist $ p\in  \mathscr P(\mathcal P), q\in  \mathscr P(\mathcal Q)$, a $\ast$-homomorphism $\theta:p \mathcal P p\rightarrow q\mathcal Q q$  and a partial isometry $0\neq v\in q \mathcal M p$ such that $\theta(x)v=vx$, for all $x\in p \mathcal P p$.
\item For any group $\mathcal G\subset \mathscr U(\mathcal P)$ such that $\mathcal G''= \mathcal P$ there is no sequence $(u_n)_n\subset \mathcal G$ satisfying $\|E_{ \mathcal Q}(xu_ny)\|_2\rightarrow 0$, for all $x,y\in \mathcal  M$.
\end{enumerate}
\end{theorem} 
\vskip 0.02in
\noindent If one of the two equivalent conditions from Theorem \ref{corner} holds then we say that \emph{a corner of $\mathcal P$ embeds into $\mathcal Q$ inside $\mathcal M$}, and write $\mathcal P\prec_{\mathcal M}\mathcal Q$. If we moreover have that $\mathcal P p'\prec_{\mathcal M}\mathcal Q$, for any projection  $0\neq p'\in \mathcal P'\cap 1_{\mathcal P} \mathcal M 1_{\mathcal P}$ (equivalently, for any projection $0\neq p'\in\mathscr Z(\mathcal P'\cap 1_{\mathcal P}  \mathcal M 1_{P})$), then we write $\mathcal P\prec_{\mathcal M}^{s}\mathcal Q$.
\vskip 0.02in

\begin{lemma}\label{many} Let $A \leqslant H \leqslant G$ be countable groups such that ${\rm Comm}^{(1)}_G(A)=G$. Let $G \ca \mathcal N$ be a trace preserving action  and denote $\M = \mathcal N\rtimes G$. Let $\A_1,\dots, \A_k  \subseteq p\M p$ be von Neumann subalgebras such that  $\A_1 \prec_\M \mathcal N \rtimes H$ and $\mathcal A_i\prec^s_\M \mathcal N \rtimes A$ for all $2\leq i\leq k$.

Then there exist finitely many elements $x_i, y_i \in \cM$ and a constant $D>0$ such that for every $a_j\in \mathscr U(\cA_j)$ with $1\leq j\leq n$  we have 
$\sum_i\|E_{\mathcal N \rtimes H}(x_i a_1 a_2\cdots a_k y_i)\|_2\geq D$.

\end{lemma}

{\it Proof.}
Since  $\mathcal A_1\prec_\M \mathcal N\rtimes H$ we apply Popa's intertwining technique and deduce that there exist a scalar $0<C_1<1$ and a subset $T_1= K_1H K_2\subset G$ that is small relative to $\{H\}$ such that $\|P_{T_1}(a_1) - a_1 \|_2\leq C_1$, for all $a_1\in \mathscr U( \cA_1)$. Let $(1-C_1)/2>\varepsilon>0$. Since $\mathcal A_2 \prec^s_\M \mathcal N \rtimes A$, we can use \cite[Lemma 2.5]{Va10b} and obtain a subset $S\subset G$ that is small relative to $\{ A\}$ so that $\|P_S(a_2) - a_2 \|_2\leq \varepsilon$ for all $a_2\in \mathscr U( \cA_2)$. By using basic calculations, for every $a_1\in\mathscr U( A)$ and $a_2\in \mathscr U(\mathcal A_2)$ we have that \begin{equation}\label{belowbound}
\begin{split} \|P_S(P_{T_1}(a^*_1) a_1 a_2)\|_2 &\geq \|P_S(a_2)\|_2-\| P_S((P_{T_1}(a^*_1) -a^*_1) a_1a_2)\|_2 \\ & \geq 1-\varepsilon -\|P_{T_1}(a^*_1) -a^*_1 \|_2 \\
 & \geq 1-\varepsilon-C_1.
\end{split}
\end{equation}
Since ${\rm Comm}^{(1)}_G(A)=G$ there is a finite set $F\subset G$ so that $S\subseteq A F$. Thus, for every $y\in \M$ we have $\|P_S(y)\|_2^2\leq \sum_ {g\in F}\|E_{\mathcal N\rtimes A}(y u_{g^{-1}})\|_2^2$. Similarly, for every $g\in G$ there is a finite set $K_g\subset G$ such that  $g^{-1}A \subseteq A K_g$ and hence $\|E_{\mathcal N \rtimes A}(u_g y)\|_2^2=  \|u_{g^{-1}}E_{\mathcal N \rtimes A}(u_g y)\|_2^2= \|P_{g^{-1}A}(y)\|_2^2\leq \|P_{A K_g}(y)\|_2^2= \sum_{k\in K_g} \|E_{\mathcal N \rtimes A}(y u_{k^{-1}})\|_2^2$. Using these inequalities in combination with \eqref{belowbound} we derive that

\begin{equation*}
\begin{split} (1-\varepsilon-C_1)^2&\leq \|P_S(P_{T_1}(a^*_1) a_1 a_2)\|^2_2\leq \sum_{g\in F} \|E_{\mathcal N \rtimes A}(P_{T_1}(a^*_1) a_1a_2 u_{g^{-1}})\|_2^2 \\ &=  |K_1||K_2|\sum_{g\in F, s\in K_1, t\in K_2}\| E_{\mathcal N \rtimes A}( u_s E_{\cL(H)}( u_{s^{-1}}a^*_1 u_{t^{-1}}) u_ta_1 a_2 u_{g^{-1}})\|_2^2\\
& \leq |K_1||K_2|\sum_{g\in F, s\in K_1, t\in K_2, k\in K_s}\| E_{\mathcal N \rtimes A}( E_{\cL(H)}( u_{s^{-1}}a^*_1 u_{t^{-1}}) u_ta_1a_2 u_{{{kg}^{-1} }} )\|_2^2 \\
& \leq |K_1||K_2|\sum_{g\in F, s\in K_1, t\in K_2, k\in K_s}\|  E_{\mathcal N \rtimes H}(u_t a_1a_2 u_{{{kg}^{-1} }} )\|_2^2.\end{split}
\end{equation*}
This clearly implies that there exist $0<D_2<1$ and a subset $T_2<G$ that is small over $\{H\}$ such that $\|P_{T_2}(a_1a_2)\|_2\geq D_2$ for all $a_1\in \mathscr U(\cA_1), a_2\in \mathscr U(\mathcal A_2)$. Therefore letting $C_2=1-D_2$ we have $0<C_2<1$ and $\| P_{T_2}(a_1 a_2)-a_1 a_2\|_2<C_2$, for all $a_1 \in \mathscr U(\cA_1)$, $a_2 \in \mathscr U(\cA_2)$. Since $\cA_3 \prec^s \cN\rtimes A$ then repeating the same argument one can find a constant $0<C_3<1$ and a subset $T_3\subset G$  that is small relative to $\{H\}$ such that $\| P_{T_3}(a_1 a_2a_3) -a_1 a_2a_3\|_2<C_3$, for all $a_1 \in \mathscr U(\cA_1)$, $a_2 \in \mathscr U(\cA_2)$ and $a_3 \in \mathscr U(\cA_3)$. Proceeding by induction, after finitely many steps, we get the desired conclusion. 
\hfill$\blacksquare$

The following result is a generalization of \cite[Lemma 2.3]{BV12} and for the reader's convenience we include all the details of the proof.

\begin{proposition}\label{joinint} Let $A \leqslant H \leqslant G$ be countable groups such that ${\rm Comm}^{(1)}_G(A)=G$. Let $G \ca \mathcal N$ be a trace preserving action  and denote $\M = \mathcal N\rtimes G$. Let $\B \subseteq p\M p$ be a von Neumann subalgebra and let  $\mathcal G \leqslant \mathcal N_{p\M p}(\B)$ be a countable subgroup. 

If  $\B \prec^s_\M \mathcal N \rtimes A$ and $\mathcal G''\prec_\M \mathcal N \rtimes H$ then $(\mathcal B\mathcal G)''\prec_\M \mathcal N \rtimes H$.
 \end{proposition}

{\it Proof.} By applying Lemma \ref{many} for $\cA_1= \mathcal G''$,  $\cA_2= \mathcal B$, it follows that there exist finitely many elements $x_i, y_i \in \cM$ and a constant $D>0$ such that for every $b\in \mathscr U(\cB)$ and $g \in \mathcal G$ we have 
$\sum_i\|E_{\mathcal N \rtimes H}(x_i b g y_i)\|_2\geq D$. However, since $\mathscr U(B)\mathcal G$ is a group generating $(\mathcal B \mathcal G)''$ as a von Neumann algebra, we conclude using Theorem  \ref{corner}  that $(\mathcal B\mathcal G)''\prec_\M \mathcal N \rtimes H$.
\hfill$\blacksquare$

The following proposition has been obtained in \cite[Lemma 2.4]{Dr19a} under the additional assumption that $\cB$ is abelian. Since we realized that this condition is not necessarily by making a mild modification of the proof, we include a complete proof of this result for the convenience of the reader.

\begin{proposition}\label{L:joint}
Let $\cM$ be a tracial von Neumann algebra and let $\cQ\subset q\cM q$ be a regular von Neumann subalgebra. Let $\cA,\cB\subset p\cM$p be commuting subalgebras such that $\cA\prec_{\cM} \cQ$ and $\cB\prec^s_\cM \cQ.$ 

Then $\cA\vee\cB\prec_\cM \cQ.$
\end{proposition}

{\it Proof.} The proof is inspired by the proof of \cite[Lemma 2.6]{Dr19a}. Since $\cA\prec_{\cM} \cQ$ there exist projections $a\in \cA, q_0\in Q$ a non-zero partial isometry $w\in q_0\cM a$ and a $*$-homomorphism $\Psi:a\cA a\to q_0\cQ q_0$ such that 
\begin{equation}\label{ee1}
    \Psi(x)w=wx, \text{ for any }x\in a\cA a.
\end{equation}

Assume by contradiction that $a\cA a\vee \cB a\nprec_\cM \cQ$. Then there exist two sequences $(a_n)_n\subset \mathscr U (a\cA a)$ and $(b_n)_n\subset \mathscr U(\cB)$ such that 
\begin{equation}\label{ee2}
    \lim_{n\to\infty} \|E_{\cQ}(xa_nb_ny)\|_2=0, \text{ for all } x,y\in\cM.
\end{equation}
Denote $f=w^*w\leq a $ and $e=\vee_{u\in\mathscr U(\cB)} ufu^*\in \cB'\cap p\cM p$.
By using relations \eqref{ee1}, \eqref{ee2} and the fact that $\cA$ and $\cB$ commute, we derive that for any $u\in\mathscr U(\cB)$ we have
\[
    \lim_{n\to \infty}\|E_{\cQ}(wu^*b_n y) \|_2=\lim_{n\to \infty}\|E_{\cQ}(\Psi(a_n)wu^*b_n y) \|_2=\lim_{n\to \infty}\|E_{\cQ}(wa_nu^*b_n y) \|_2=0, \text{ for any }y\in\cM.
\]

Next, since $\cQ\subset q\cM q$ is regular, we further deduce that
\begin{equation}\label{ee3}
    \lim_{n\to \infty}\|E_{\cQ}(xufu^*b_n y) \|_2=0, \text{ for all }x, y\in \cM \text{ and }u\in\mathscr U(\cB).
\end{equation}
Let $u_1,u_2\in\mathscr U(\cB)$. Using basic spectral theory we obtain from \eqref{ee3} that 
$$    \lim_{n\to \infty}\|E_{\cQ}(x\chi_{(c,\infty)}(u_1fu_1^*+u_2fu_2^*)b_n y) \|_2=0, \text{ for all }x, y\in \cM \text{ and }c>0.$$
We denote by $s(b)$ the support projection of a positive element $b\in\cM$. As $\chi_{(c,\infty)}(u_1fu_1^*+u_2fu_2^*)\to s(u_1fu_1^*+u_2fu_2^*)$ in SOT-topology, we deduce that $$   \lim_{n\to \infty}\|E_{\cQ}(s(u_1fu_1^*+u_2fu_2^*)b_n y)) \|_2=0, \text{for any }  y\in \cM.$$  As $s(u_1fu_1^*+u_2fu_2^*)=u_1fu_1^*\vee u_2fu_2^*$, we further derive that $    \lim_{n\to \infty}\|E_{\cQ}((u_1fu_1^*\vee u_2fu_2^*)b_n y)) \|_2=0$, for all $ y\in \cM.$
Proceeding by induction, we obtain that $\lim_{n\to \infty}\|E_{\cQ}((\vee_{u\in\mathcal F}ufu^*)b_n y)) \|_2=0$, for all $ y\in \cM$ and every finite subset $\mathcal F\subset\mathscr U(\cB)$. Since the family of projections $\vee_{u\in\mathcal F}ufu^*$ indexed by finite subsets $\mathcal F\subset\mathscr U(\cB)$ converges to $e$ in the SOT-topology, a basic calculation shows that $    \lim_{n\to \infty}\|E_{\cQ}(eb_n y)) \|_2=0$, for all $ y\in \cM.$ Since $\cQ\subset q\cM q$ is regular, it follows that $\cB e\nprec_{\cM} \cQ$, contradiction.
\hfill$\blacksquare$

\section{An augmentation technique for intertwining and relative amenability}

In this subsection we prove Theorem \ref{Th:augmentationtechniques} which plays an important role in the proofs of Theorem \ref{Main:leftright} and Theorem \ref{Main:coinduced}. The result is an immediate consequence of the "augmentation technique" developed in \cite{CD-AD20}, being essentially contained in \cite[Section 3]{CD-AD20}. Nevertheless, for completeness, we will include all the details. 

First, we recall the relative amenability notion introduced by 
Ozawa and Popa \cite[Definition 2.2]{OP07}. 

\begin{definition}[\!\!\cite{OP07}]
Let $(\cM,\tau)$ be a von Neumann algebra and let $\cP\subset p\cM p,\cQ\subset \cM$ be von Neumann subalgebras. $\cP$ is {\it amenable relative to $\cQ$ inside $\cM$} if there exists a positive linear functional $\Phi:p\langle \cM,e_{\cQ}\rangle p\to\mathbb C$ such that $\Phi_{|p\cM p}=\tau$ and $\Phi$ is $\cP$-central. Moreover, if $\cP p'$ is non-amenable relative to $\cQ$ for any non-zero projection $p'\in \cP'\cap p\cM p$, we say that $\cP$ is {\it strongly non-amenable relative to} $\cQ$ 
\end{definition}

\begin{theorem}\label{Th:augmentationtechniques}
Let $\Gamma$ be a countable group and denote $\cM=\cL(\Gamma)$. Let $\cQ_1, \cQ_2\subset \cM $ be von Neumann subalgebras which form a commuting square, i.e. $E_{\cQ_1}\circ E_{\cQ_2}= E_{\cQ_2}\circ E_{\cQ_1}$, with $\cQ_1$ regular in $\cM$.

Let $\cP\subset p\cM p$ be a diffuse von Neumann subalgebra 
and let $\Sigma<\Gamma$ be a subgroup such that $\cP\prec^s_{\cM} \cL(\Sigma)$. Assume in addition that there exist some non-zero projections $f_1,f_2\in \cL(\Sigma)'\cap \cM $ such that $\cL(\Sigma)f_i$ is amenable relative to $\cQ_i$ inside $\cM$ for any $1\leq i\leq 2.$ Then the following hold:

\begin{enumerate}
    \item If $\cP$ has no amenable direct summand, then $\cQ_1\cap \cQ_2$ is non-amenable.
    
    \item If $\cP$ has property (T), then $\cQ_1\cap \cQ_2$ does not have the Haagerup property.
    
\end{enumerate}

\end{theorem}

Note that if $f_1$ and $f_2$ have full support, then Theorem \ref{Th:augmentationtechniques} can easily be proven in the following way. First, by using \cite[Proposition 2.7]{PV11} we derive that $\cL(\Sigma)$ is amenable relative to $\cQ_1\cap \cQ_2$. Next, by \cite[Lemma 2.6(2)]{DHI16}, we have that $P$ is amenable relative to $\cL(\Sigma)$. Finally, by applying  \cite[Proposition 2.4(3)]{OP07} we derive that $\cP$ is amenable relative to $\cQ_1\cap \cQ_2$, which implies the conclusion of the theorem. However, our theorem deals with {\it arbitrary} non-zero projections $f_1$ and $f_2$. In order to overcome this difficulty, we developed in \cite{CD-AD20} an augmentation technique that exploits the group von Neumann algebra structure of $\cL(\Sigma)\subset \cL(\Gamma)$ by considering a mixing action of $\Gamma.$

{\bf Notation.} Throughout the proof we use the following notation. Let $\cM$ be a tracial von Neumann algebra and $\cP\subset p\cM p$ and $\cQ\subset q\cM q$ some subalgebras. We write $\cP\prec_{\cM}^{s'}\cQ$ if $\cP\prec_{\cM}\cQ q'$, for any non-zero projection $q'\in \cQ'\cap q\cM q$. We also write $\cP\prec_{\cM}^{s,s'}\cQ$ if $\cP\prec_{\cM}^{s}\cQ$ and $\cP\prec_{\cM}^{s'}\cQ$.

{\it Proof.} We consider a free, mixing action $\Gamma\car \cD$ with abelian base and denote $\tilde \cM=\cD\rtimes \Gamma$. Let $\Delta:\cM\to \cM\bar\otimes \cM$ be the $*$-homomorphism given by $\Delta(gu_g)=du_g\otimes u_g,$ for all $d\in\cD,g\in\Gamma$. Since $\cP\prec_{\cM} \cL(\Sigma),$ it follows that $\Sigma$ is an infinite group, and hence, $(\cD\rtimes\Sigma)'\cap \cM=\mathbb C$. 

By using \cite[Remark 2.2]{DHI16} we further get $P\prec^{s}_{\tilde\cM} \cD\rtimes\Sigma$. Using \cite[Lemma 2.4(4)]{DHI16} we derive that $P\prec^{s,s'}_{\tilde\cM} \cD\rtimes\Sigma$ and by applying \cite[Lemma 2.3]{Dr19b} we further deduce that $\Delta(\cP )\prec_{\tilde\cM\bar\otimes \cM}^{s,s'} \tilde\cM\bar\otimes \cL(\Sigma)$. Using \cite[Lemma 2.6(2)]{DHI16} we obtain that $\Delta(\cP )$ is amenable relative to $\tilde\cM\bar\otimes (\cL(\Sigma)f_i\oplus \mathbb C(1-f_i))$ inside $\tilde\cM\bar\otimes \cM$, for any $1\leq i\leq 2$. Together with the assumption and \cite[Proposition 2.4(3)]{OP07} we have that for any $1\leq i\leq 2$
\begin{equation}\label{xs}
    \Delta(\cP ) \text{ is amenable relative to }\tilde\cM\bar\otimes \cQ_i \text{ inside }\tilde\cM\bar\otimes \cM.
\end{equation}

Since $\cQ_1, \cQ_2\subset \cM $ form a commuting square and $\cQ_1$ regular in $\cM$, we obtain from \cite[Proposition 2.7]{PV11} and \eqref{xs} that $\Delta(\cP ) \text{ is amenable relative to }\tilde\cM\bar\otimes (\cQ_1\cap\cQ_2) \text{ inside }\tilde\cM\bar\otimes \cM$.

Part (1) of the theorem follows directly from \cite[Lemma 10.2]{IPV10}. For proving part (2) we assume by contradiction that $\cQ_1\cap \cQ_2$ has the Haagerup property. Hence, \cite[Lemma 1]{HPV11} implies that 
$\Delta(\cP )\prec_{\tilde\cM\bar\otimes\cM} \tilde \cM \bar\otimes 1$ since $\cP$ has property (T). Using \cite[Lemma 10.2]{IPV10} we further obtain that $\cP\prec_{\tilde \cM} \cD$, which implies that $\cP$ is not diffuse, contradiction.
\hfill$\blacksquare$

\section{Some intertwining results in von Neumann algebras associated to wreath products and coinduced groups}

The goal of this section is to prove the following theorem which is an important ingredient for the proofs of our main results. 

\begin{theorem}\label{commutationcontrolincommultiplication}\label{Th:coind1} 
Let $G$ be a countable group that has one of the following forms:

\begin{enumerate}
    \item [(A)] $G= A_0\wr_I \Gamma\in \mathscr {GW}_0(A_0)$.
    
    \item [(B)] $G=A_0^I\rtimes \Gamma$ is the coinduced group of a group action $\Gamma_0\car A_0$ by automorphisms with $[\Gamma:\Gamma_0]=\infty$ and $\Gamma$ is a product of $n\ge 1$ icc non-amenable, weakly amenable, bi-exact and property (T) groups.
\end{enumerate}

Let $\cM=\cL(G)=\cL(H)$ for a countable group $H$ and denote by $\Delta: \cM\to\cM\bar\otimes\cM$ the $*$-homomorphism given by $\Delta(v_h)=v_{h}\otimes v_{h}$, for any $h\in H$. Then the following hold:
\begin{enumerate}
\item  $\Delta(\cL(A_0^{I}))\prec_{M\bar\otimes \M} \cL(A_0^{I})\bar\otimes \cL(A_0^{I})$;
\item  There exists a unitary $u\in \mathcal U(\M\bten \M)$ such that $u\Delta(\cL(\Gamma))u^*\subseteq \cL(\Gamma)\bten \cL(\Gamma)$.    
\end{enumerate}
\end{theorem}

Note that part (A) was proven in \cite[Proposition 4.2]{CU18} under the stronger assumption that $A_0$ has property (T). Although part (B) is a coinduced version of the first part, note that it applies to generalized wreath products with {\it arbitrary} base as long as the acting group is a product of hyperbolic, property (T) groups.

\subsection{Proof of Theorem \ref{commutationcontrolincommultiplication}, case (A)}

Before proceeding to the proof of our theorem, we prove the following lemma which is a direct consequence of \cite[Lemma 4.1.3]{IPV10}. For reader convenience we include some details in the proof on how it follows from these results.
\begin{lemma}\label{intquasinormalizer} Let $\cR= \cA_0\wr_I \Gamma$ be a generalized wreath product with $
\cA_0$ tracial. Let $\M$ be any tracial von Neumann algebra and fix a non-zero projection $p\in \M\bten \R$.  Assume that $\B\subseteq p (\M\bten \R)p$ is a von Neumann subalgebra for which there exists a finite set $\mathcal G\subset I$ satisfying  $\B\prec \M \bten (\cA_0^I\rtimes {\rm Stab}_{\Gamma}(\mathcal G))$ and  $\B\nprec \M \bten (\cA_0^I\rtimes {\rm Stab}_{\Gamma}(\mathcal G_o))$ for all finite suprasets $\mathcal G_o\supset \mathcal G$. 

Then $\mathcal {QN}_{p(\M \bten \R) p}(\B)'' \prec_{\M \bten \R} \M \bten (\cA_0^I\rtimes {\rm Stab}_{\Gamma}(\mathcal G))$.
\end{lemma}

{\it Proof.} The assumption implies that there exist non-zero projections $b\in \B, c\in \M \bten (\cA_0^I\rtimes {\rm Stab}_{\Gamma}(\mathcal G))$, a non-zero partial isometry $v\in \M \bten \R$ and a unital $\ast$-isomorphism onto its image $\Phi: b\B b \ra  \C:= \Phi(bBb ) \subseteq \M \bten (\cA_0^I\rtimes {\rm Stab}_{\Gamma}(\mathcal G))$ such that 
\begin{equation}\label{partint}
\Phi(x)v=vx \text{ for all }x\in b\B b.
\end{equation} 
This entails that $vv^*\in c Cc'\cap c (\M \bten \R)c$ and $v^*v\in b\B b'\cap b (\M\bten \R)b$. Moreover, we can assume without any loss of generality that the support projection of $E_{\M \bten (\cA_0^I\rtimes {\rm Stab}_{\Gamma}(\mathcal G))}(vv^*)$ equals $c$. Let $u$ be a unitary such that $u v^*v=v$. Thus equation \eqref{partint} implies that 
$u(\B v^*v )u^*= v\B v^*= \C vv^*$. Also note that since   $\B\nprec \M \bten (\cA_0^I\rtimes {\rm Stab}_{\Gamma}(\mathcal G_o))$ for all finite sets $\mathcal G_o\supset \mathcal G$ then we must have that $\C\nprec_{\M \bten (\cA_0^I\rtimes {\rm Stab}_{\Gamma}(\mathcal G))} \M \bten (\cA_0^I\rtimes {\rm Stab}_{\Gamma}(\mathcal G_o))$ for all finite suprasets $\mathcal G_o\supset \mathcal G$. Thus, by \cite[Lemma 4.1.3]{IPV10} we get that $vv^*\in \mathcal {QN}_ {c (\M\bten \R )c}(\C )''\subseteq \M \bten (\cA_0^I\rtimes {\rm Stab}_{\Gamma}(\mathcal G)) )$. Thus, by Popa's quasi-normalizer compression formula we get  
\begin{equation}\label{dd1}
\begin{split}u v^*v \mathcal {QN}_ {p(\M\bten \R )p }( \B )'' v^*v u^* &=\mathcal {QN}_ { u v^*v (\M\bten \R ) v^*v u^*}( u\B v^*v u^*)'' \\ &=\mathcal {QN}_ {vv^* (\M\bten \R ) vv^*}(\C vv^* )''\\ &= vv^* \mathcal {QN}_ { \M\bten \R }(\C )''vv^*\subseteq \M \bten (\cA_0^I\rtimes {\rm Stab}_{\Gamma}(\mathcal G)). 
\end{split}
\end{equation}  
Since ${\rm Stab}_{\Gamma}(\mathcal G)\leqslant {\rm Norm}(\mathcal G)$ has finite index, we derive the proof of the result from \eqref{dd1}.
\hfill$\blacksquare$

{\it Proof of Theorem \ref{commutationcontrolincommultiplication}, case (A).}
(1) For ease of notation, we will simply write $A$ instead of $A_0$. Fix a projection $p\in \Delta (\cL(A^I))'\cap \M \bten \M$.  Also, fix a base point $o\in I$ and denote by $\B= \cL(A^o)$ a diffuse, property (T) von Neumann subalgebra given by $G\in \mathscr {GW}_0(A)$. Let $\P\subseteq \M \bten \M$ be the quasi-normalizer of $\Delta(\B)p$ in $p(\M \bten \M)p$ and notice that $\P\supseteq \Delta(\cL (A^{\widehat o}))p$. Here, we denoted by $\widehat o=I\setminus \{o\}$. Since $\Delta (\B)p$ has property (T) then using \cite[Theorem 4.2]{IPV10} one of the following must hold:
\begin{enumerate}
\item [(i)]\label{int}$\Delta(\B)p\prec \M\otimes 1$;
\item [(ii)]\label{int2}$\P\prec \M\bten \cL(A^{(I)}\rtimes {\rm Stab}_{\Gamma}(j))$ for some $j\in I$;
\item [(iii)]\label{int3}there is a unitary $u\in p(\M\bten \M)p$ such that $u \P u^*\subseteq \M\bten \cL(\Gamma)$.
\end{enumerate} 
Since $\B$ is diffuse, we derive from \cite[Proposition 7.2]{IPV10} that the case (i) is impossible. Suppose that (iii) holds. By the remark above we get that $u\Delta(\cL(A^{\widehat o})) u^*\subseteq \M\bten \cL(\Gamma)$. There are two possibilities: either $\Delta(\cL(A^{\widehat o}))p\prec_{\M \bten \M } \M\bten \cL({\rm Stab}_{\Gamma}(j))$ for some $j\in I$ or $\Delta(\cL(A^{\widehat o}))p\nprec_{\M \bten \M } \M \bten \cL({\rm Stab}_{\Gamma}(j))$ for all $j\in I$. 
\vskip 0.08in 
In the first case we again have two possibilities: either there is a maximal finite set $j\in \mathcal G \subset I$ such that $\Delta(\cL(A^{\widehat o}))p\prec_{\M \bten \M } \M \bten \cL({\rm Stab}_{\Gamma}(\mathcal G))$ or there is no such subset. Assume the first subcase. Notice that $u^*\Delta(\cL(A^{\widehat o}))p u\prec_{\M \bten \M } \M \bten \cL(A^I\rtimes {\rm Stab}_{\Gamma}(\mathcal G))$. Also if there is a finite subset $\mathcal G\subsetneq \mathcal G_o$ so that  $ u^*\Delta(\cL(A^{\widehat o}))p u\prec_{\M \bten \M } \M \bten \cL(A^I\rtimes {\rm Stab}_{\Gamma}(\mathcal G_o))= \cL(G\times (A^I\rtimes {\rm Stab}_{\Gamma}(\mathcal G_o)))$ and since  $u\Delta(\cL(A^{\widehat o})) u^*\subseteq \M\bten \cL(\Gamma)=\cL(G\times \Gamma)$ one can find $g\in G\times G$ such that $u^*\Delta(\cL(A^{\widehat o}))p u\prec_{\M \bten \M } \cL(G \times (A^I\rtimes {\rm Stab}_{\Gamma}(\mathcal G_o))\cap g (G\times \Gamma)g)$. However, one can see that there is an element $n\in A^I$ such that $G \times (A^I\rtimes {\rm Stab}_{\Gamma}(\mathcal G_o))\cap g (G\times \Gamma)g= G\times (A^I \rtimes {\rm Stab}_{\Gamma}(\mathcal G_o) \cap n\Gamma n^{-1}))= G\times (n {\rm Stab}_{\Gamma}(\mathcal G_o) n^{-1}) = n (G\times {\rm Stab}_{\Gamma}(\mathcal G_o))n^{-1}$. Altogether, these relations show that $u^*\Delta(\cL(A^{\widehat o}))p u\prec_{\M \bten \M } \cL (G\times n \Gamma n^{-1}) = u_n (\M\bar\otimes \cL({\rm Stab}_{\Gamma}(\mathcal G_o))) u_n^* $ and hence $u^*\Delta(\cL(A^{\widehat o}))p u\prec_{\M \bten \M } \M\bar\otimes \cL({\rm Stab}_{\Gamma}(\mathcal G_o))$, which contradicts the hypothesis. Thus, $\mathcal G$ is maximal with the property that $\Delta(\cL(A^{\widehat o}))p u\prec_{\M \bten \M } \M \bten \cL(A^I\rtimes {\rm Stab}_{\Gamma}(\mathcal G))$. Using Lemma \ref{intquasinormalizer} we get that 
$$
\Delta(\cL(A^I))p \subseteq \mathcal {QN}_{p(\M\bten \M)p}(\Delta (\cL(A^{\widehat o}))p)''\prec_{\M\bten \M } \M \bten \cL(A^I\rtimes {\rm Stab}_{\Gamma}(\mathcal G)).
$$ 
Again two possibilities arise: either there is a maximal finite subset $j\in \mathcal G\subseteq \mathcal K$ such that $\Delta(\cL(A^{I}))p\prec_{\M \bten \M } \M\bten \cL(A^I\rtimes {\rm Stab}_{\Gamma}(\mathcal K))$ or there is no such subset. In the first sub-subcase, by Lemma \ref{intquasinormalizer} and Popa's quasi-normalizer compression formula we have 
$$p\mathscr {QN}_{\M\bten \M}(\Delta (\cL(A^{I})))'' p\prec_{\M\bten \M } \M \bten \cL(A^I\rtimes {\rm Stab}_{\Gamma}(\mathcal K)),$$
and hence, $\mathscr {QN}_{\M\bten \M}(\Delta (\cL(A^{I})))'' \prec_{\M\bten \M } \M \bten \cL(A^I\rtimes {\rm Stab}_{\Gamma}(\mathcal K))$ as $p\in \Delta (\cL(A^I))'\cap \M\bten \M \subseteq \mathscr {QN}_{\M\bten \M}(\Delta (\cL(A^{I})))''$. Therefore   $\Delta(\M)\prec_{\M\bten \M} \M\bten \cL(A^I\rtimes {\rm Stab}_{\Gamma}(\mathcal K))\subseteq \M\bten \cL(A^I\rtimes {\rm Stab}_{\Gamma}(j))$ which by \cite[Lemma 7.2.7]{IPV10} gives that $[\Gamma: {\rm Stab}_{\Gamma}(j)]<\infty$, contradiction. In the second sub-subcase, by taking $\mathcal G\subseteq \mathcal F$ with $|\mathcal F|\geq k$ we get that $\Delta(\cL(A^I))p\prec_{\M\bten \M} \M\otimes 1$, which again contradicts \cite[Proposition 7.2]{IPV10}. In the second subcase, again by taking $|\mathcal G|\geq k$ we get $\Delta(\cL(A^{\widehat o}))p\prec_{\M\bten \M} \M\otimes 1$, which is a contradiction as before.
\vskip 0.08in

Thus, we are left with the situation when $\Delta(\cL(A^{\widehat o}))p\nprec_{\M \bten \M } \M \bten \cL({\rm Stab}_\Gamma(j))$ for all $j\in I$ or, equivalently, $u^*\Delta(\cL(A^{\widehat o}))p u\nprec_{\M \bten \M } \M \bten \cL({\rm Stab}_\Gamma(j))$ for all $j\in I$. By \cite[Lemma 4.1.1]{IPV10} this further implies that $u^* \Delta (\cL(A^I)) pu \subseteq \mathscr {QN}_{ p(\M\bten \M)p}(u^* \Delta (\cL(A^{\widehat o}))p u)''\subseteq \M \bten \cL(\Gamma)$. As we also have that $u^*\Delta(\cL(A^{I}))p u\nprec_{\M \bten \M } \M \bten {\rm Stab}_\Gamma(j)$ for all $j\in I$, the same argument shows that  $u^* p\mathscr {QN}_{\M\bten \M}(\Delta (\cL(A^I)))''p u= \mathscr {QN}_{p(\M\bten \M)p}(u^* \Delta (\cL(A^I)) u)'' \subseteq \M \bten \cL(\Gamma)$. Since $\Delta (\mathcal M)\subseteq \mathscr {QN}_{\M\bten \M}(\Delta (\cL(A^I)))''$, we then get $\Delta (\M)\prec \M\bten \cL(\Gamma)$, which further entails that the inclusion $\cL(\Gamma)\subseteq \M$ has finite index, contradiction.
\vskip 0.08in
Next, assume (ii). Again we have two possibilities: either there exists a finite maximal set $j\in \mathcal G \subset I$ such that $\Delta (\cL(A^{\widehat o}))p\prec M\bten \cL(A^{(I)}\rtimes {\rm Stab}_\Gamma(\mathcal G))$ or there is no such subset. In the first subcase, by Lemma \ref{intquasinormalizer} we have  $ \Delta(\cL(A^I))p\subseteq \mathscr {QN}_{p(\M\bten \M)p}(\Delta (\cL(A^{\widehat o}))p )''\prec_{\M\bten \M } \M \bten \cL(A^I\rtimes {\rm Stab}_\Gamma(\mathcal G))$. Thus, using Lemma \ref{intquasinormalizer} we get $ p\mathscr {QN}_{\M\bten \M}(\Delta (\cL(A^{I})))''p= \mathscr {QN}_{p(\M\bten \M)p}(\Delta (\cL(A^{I}))p)''\prec_{\M\bten \M } \M \bten \cL(A^I\rtimes {\rm Stab}_\Gamma(\mathcal G))$ and hence $\mathscr {QN}_{\M\bten \M}(\Delta (\cL(A^{I})))''\prec \M \bten \cL(A^I\rtimes {\rm Stab}_\Gamma(\mathcal G))$. Using again that $\Delta(\M)\subseteq \mathscr {QN}_{\M\bten \M}(\Delta (\cL(A^{I})))''$ we have $\Delta (\M)\prec \M \bten \cL(A^I\rtimes {\rm Stab}_\Gamma(\mathcal G))$, which further entails that $[\M : \cL(A^I \rtimes {\rm Stab}_\Gamma(\mathcal G))]<\infty$, a contradiction. In the second subcase, by taking $\mathcal G$ with  $|\mathcal G|\geq k$ we get $ \Delta(\cL(A^{\widehat o}))p\prec \M \bten \cL(A^I)$. Since this can be done for every base point $o\in I$ we have obtained that 
$$\Delta(\cL(A^{\widehat o}))p\prec \M \bten \cL(A^I) \text{ for all }o\in I.$$
By using \cite[Lemma 2.4(2)]{DHI16}, the previous relation implies that $\Delta(\cL(A^{\widehat o}))\prec^s \M \bten \cL(A^I)$  for all $o\in I$. By varying the point $o\in I$ we also have that $\Delta(\cL(A^{o}))\prec^s \M \bten \cL(A^I)$ and therefore using Proposition \ref{L:joint} we conclude that 
 $\Delta(\cL(A^I))\prec^s \M\bten \cL(A^I)$.

In a similar manner, working on the other side of the tensor we get  $\Delta(\cL(A^I))\prec^s  \cL(A^I)\bten \M $. Altogether, by using \cite[Lemma 2.8]{DHI16} we get  $\Delta(\cL(A^I))\prec^s  \cL(A^I)\bten \cL(A^I) $ which finishes the proof of i). 
\vskip 0.12in
(2) Let $\cP$ be the quasi-normalizer of $\Delta(\cL(\Omega))$ inside $\cM\bar{\otimes}\cM$ and notice that $\Delta(\cL(\Gamma))\subseteq \P$. Since the inclusion $\Delta(\cL(\Omega))\subset \M\bar{\otimes}(\cL(A^I)\rtimes \Gamma)$ is rigid by \cite[Theorem 4.2]{IPV10} one of the following has to hold:
\begin{enumerate}
\item [(i)] $\Delta(\cL(\Omega))\prec_{\M\bar\otimes \M} \M\bar{\otimes}1$;
\item [(ii)] $\P \prec_{\M\bar\otimes \M} \M\bar{\otimes}(\cL(A^I) \rtimes {\rm Stab}_\Gamma(i))$, for some $i\in I$;
\item [(iii)] There is $v\in \mathcal U(\M\bar{\otimes}\M)$ so that $v\P v^*\subset \M\bar{\otimes}\cL(\Gamma)$.
\end{enumerate}
As $\Delta(\cL(\Omega))$ is diffuse one cannot have possibility (i) by \cite[Proposition 7.2]{IPV10}. Suppose we are in case (ii). This implies in particular that $\Delta(\cL(\Gamma)) \prec_{\M\bar\otimes \M} \M\bar{\otimes}(\cL(A^I)\rtimes {\rm Stab}_\Gamma(i))$. But since $\Delta(\cL(A^I))\prec^s_{\M\bar\otimes \M} \cL(A^I)\bar{\otimes} \cL(A^I)$, by the same argument as in the beginning of the proof of \cite[Theorem 8.2]{Io10} or the Theorem \ref{joinint} above, we would get $\Delta( \cL(A^I)\rtimes \Gamma)=\Delta(\M)\prec_{\M\bar\otimes \M} \M\bar{\otimes}(\cL(A^I) \rtimes {\rm Stab}_\Gamma(i))$, which by  \cite[Lemma 7.2.2]{IPV10} implies that $\cL(A^I)\rtimes {\rm Stab}_\Gamma(i) \subset \M$ has finite index, a contradiction. So (iii) must be true, hence a fortiori $v\Delta(\cL(\Gamma))v^*\subset \M\bar{\otimes}\cL(\Gamma)$. Repeating the argument for the inclusion $v\Delta(\cL(\Gamma))v^*\subset \M\bar{\otimes}\cL(\Gamma)=(\cL(A^I)\rtimes \Gamma)\bar{\otimes}\cL(\Gamma)$, we obtain an unitary $u\in \M\bar{\otimes}\M$ such that $u\Delta(\cL(\Gamma))u^*\subset \cL(\Gamma)\bar{\otimes}\cL(\Gamma)$, as desired.
\hfill$\blacksquare$

\subsection{Proof of Theorem \ref{Th:coind1}, case (B)}

(1) Denote $A=A_0^I$ and let $\Gamma=\Gamma_1\times\dots\times\Gamma_n$, where the groups $\Gamma_i$'s are icc non-amenable, weakly amenable, bi-exact and property (T) groups. 

Fix $k\in\overline{1,n}.$ We first show that there exists a non-empty set $S\subset I$ such that $\Delta(\cL(A_0^S))$ is amenable relative to $\cM\bar\otimes \cL(A\rtimes\Gamma_{\hat k})$. Let $i=\Gamma_0\in I$. Using \cite[Lemma 5.2]{KV15}  we have that $\Delta (\cL(A_0))$ is amenable relative to $\cM\bar\otimes \cL(A\rtimes\Gamma_{\hat k})$ or $\Delta (\cL(A_0^{\hat i}))\prec_{\cM\bar\otimes\cM}\cM\bar\otimes \cL(A\rtimes\Gamma_{\hat k})$. If the first possibility holds, then we are done. By assuming instead that the second one holds and by applying \cite[Lemma 2.4]{DHI16} we can find a non-zero projection $z\in \mathcal Z(\Delta (\cL(A_0^{\hat i}))'\cap \cM\bar\otimes \cM)\subset \Delta (\cL(A))'\cap \cM\bar\otimes \cM$ such that 
\begin{equation}\label{x1}
\Delta (\cL(A_0^{\hat i}))z\prec^s_{\cM\bar\otimes \cM}\cM\bar\otimes \cL(A\rtimes\Gamma_{\hat k}).    
\end{equation}
Next, we notice that there exists $g\in\Gamma\setminus\Gamma_0$ such that $\Delta(u_g)z\Delta(u_g)^*z\neq 0.$ If this does not hold, we can let $J$ be a system of representatives for $\Gamma/\Gamma_0$ and derive that the projections $\{\Delta(u_g)z\Delta(u_g)^*\}_{g\in J}$ are mutually orthogonal. This would imply that $[\Gamma:\Gamma_0]<\infty$, contradiction. Note that $z_g:=\Delta(u_g)z\Delta(u_g)^* \in \Delta (\cL(A))'\cap \cM\bar\otimes \cM$ and $z_gz\neq 0$ for some $g\in\Gamma\setminus\Gamma_0$. Since $g\in\Gamma\setminus\Gamma_0$ and $i=\Gamma_0\in I$, it follows that $i\in g\cdot\hat i$, and hence,
\begin{equation}\label{x2}
 \Delta(\cL(A_0))z_g\prec^s_{\cM\bar\otimes\cM} \cM\bar\otimes \cL(A\rtimes\Gamma_{\hat k}).   
\end{equation}
Note that $z$ and $z_g$ commute since $z_g\in \Delta (\cL(A_0^{\hat i}))'\cap \cM\bar\otimes \cM$ and $z\in \mathcal Z(\Delta (\cL(A_0^{\hat i}))'\cap \cM\bar\otimes \cM)$.
Now, since $0\neq z_gz\in \Delta (\cL(A))'\cap \cM\bar\otimes \cM$ is a non-zero projection, we can apply Proposition \ref{L:joint} and obtain from \eqref{x1} and \eqref{x2} that $\Delta(\cL(A))z_gz\prec_{\cM\bar\otimes\cM} \cM\bar\otimes \cL(A\rtimes\Gamma_{\hat k})$. By using \cite[Lemma 2.4(2)]{DHI16} we would get that $\Delta(\cL(A))\prec_{\cM\bar\otimes\cM}^s \cM\bar\otimes \cL(A\rtimes\Gamma_{\hat k})$. We used here the fact that $\mathscr{N}_{\cM\bar\otimes\cM}(\Delta(\cL(A)))'\cap \cM\bar\otimes\cM\subset \Delta(\cM)'\cap \cM\bar\otimes \cM=\mathbb C 1$ which follows $\cM$ being a II$_1$ factor. Next, by using \cite[Lemma 2.6(3)]{DHI16} we derive that $\Delta(\cL(A))$ is amenable relative to $\cM\bar\otimes \cL(A\rtimes\Gamma_{\widehat k})$, and therefore, we can take $S=I$. This shows that there exists a non-empty set $S\subset I$ such that $\Delta(\cL(A_0^S))$ is amenable relative to $\cM\bar\otimes \cL(A\rtimes\Gamma_{\hat k})$.

Next, since the relative amenability property is closed for inductive limits by \cite[Lemma 2.7]{DHI16}, we can assume that $S\subset I$ is a maximal non-empty subset with the property that $\Delta(\cL(A_0^S))$ is amenable relative to $\cM\bar\otimes \cL(A\rtimes\Gamma_{\hat k})$. Assume by contradiction that $S\neq I$.
 Since $\Gamma_k$ is non-amenable, weakly amenable and bi-exact we can apply \cite[Theorem 1.4]{PV12} and derive that for any non-zero projection $e\in \Delta(\cL(A))'\cap \cM\bar\otimes \cM$ we have $\Delta(\cL(A))e$ is amenable relative to $\cM\bar\otimes \cL(A\rtimes\Gamma_{\hat k})$ or $\Delta(\cL(A_0^S))e\prec_{\cM\bar\otimes\cM} \cL(A\rtimes\Gamma_{\hat k})$. The first possibility implies by \cite[Lemma 2.6(2)]{DHI16} that $\Delta(\cL(A))$ is amenable relative to $\cM\bar\otimes \cL(A\rtimes\Gamma_{\hat k})$, since $\Delta(\cM)'\cap \cM\bar\otimes \cM=\mathbb C 1$. Hence, $S=I$, contradiction. 

Therefore, by \cite[Lemma 2.4(2)]{DHI16} we deduce that $\Delta(\cL(A_0^S))\prec^s_{\cM\bar\otimes\cM} \cM\bar\otimes \cL(A\rtimes\Gamma_{\hat k})$. 
Let $i\in I\setminus S$. It is clear that we also have $\Delta(\cL(A_0^i))\prec^s_{\cM\bar\otimes\cM} \cM\bar\otimes \cL(A\rtimes\Gamma_{\hat k})$. By using Proposition \ref{L:joint} we get that $\Delta(\cL(A_0^{S\cup \{i\}}))\prec^s_{\cM\bar\otimes\cM} \cM\bar\otimes \cL(A\rtimes\Gamma_{\hat k})$. This implies by \cite[Lemma 2.6(3)]{DHI16} that $\Delta(\cL(A_0^{S\cup \{i\}}))$ is amenable relative to $ \cM\bar\otimes \cL(A\rtimes\Gamma_{\hat k})$, which contradicts the maximality of $S$.

Hence, $S=I$. Since $\Gamma_{\hat k}$ is non-amenable, weakly amenable, bi-exact, we obtain by \cite[Theorem 1.4]{PV12} that $\Delta(\cM)$ is amenable relative to $\cM\bar\otimes \cL(A\rtimes\Gamma_{\hat k})$ or $\Delta(\cL(A)) \prec_{\cM\bar\otimes\cM} \cM\bar\otimes \cL(A\rtimes\Gamma_{\hat k})$. The first option implies by \cite[Proposition 4.1(2)]{BV12} that $\cM$ is amenable relative to $\cL(A\rtimes\Gamma_{\hat k})$, which implies by \cite[Proposition 2.4(1)]{OP07} that
$\Gamma_2$ is amenable, contradiction. The second option gives $\Delta(\cL(A)) \prec^s_{\cM\bar\otimes\cM} \cM\bar\otimes \cL(A\rtimes\Gamma_{\hat k})$ by \cite[Lemma 2.4(2)]{DHI16} since $\mathcal N_{\cM\bar\otimes\cM}(\Delta(\cL(A)))'\cap \cM\bar\otimes\cM=\mathbb C 1$. 

Finally, we have $\Delta(\cL(A)) \prec^s_{\cM\bar\otimes\cM} \cM\bar\otimes \cL(A\rtimes\Gamma_{\hat k})$, for any $k\in\overline{1,n}$. By applying \cite[Lemma 2.8(2)]{DHI16}, we derive that $\Delta(\cL(A)) \prec^s_{\cM\bar\otimes\cM} \cM\bar\otimes \cL(A)$. In a similar way, we obtain that $\Delta(\cL(A)) \prec^s_{\cM\bar\otimes\cM} \cL(A) \bar\otimes \cM.$
Thus, the claim follows by using once again \cite[Lemma 2.8(2)]{DHI16}.

(2) Since $\Gamma$ has property (T) we use \cite[Theorem 3.3]{Dr20a} (see also \cite[Theorem 2.6]{KV15}) to deduce that $\Delta(\cL(\Gamma))$ $\prec_{\cM\bar\otimes \cM} \cM\bar\otimes \cL(A\rtimes\Gamma_0)$ or $\Delta(\cL(\Gamma))  \prec_{\cM\bar\otimes\cM} \cM\bar\otimes \cL(\Gamma)$. 
Assume the first option holds. In combination with part (1) and Proposition \ref{joinint}, we derive that $\Delta(\cM)\prec_{\cM\bar\otimes\cM} \cM\bar\otimes \cL(A\rtimes\Gamma_0)$. Using  \cite[Proposition 7.2(2)]{IPV10} we further deduce that $\cM\prec_{\cM} \cL(A\rtimes\Gamma_0)$, and hence, $[\Gamma:\Gamma_0]<\infty$, contradiction.
Therefore, the second option holds. Since $\cL(\Gamma)$ is a II$_1$ factor, we note that the proof of \cite[Theorem 4.2]{IPV10} (or the proof of the moreover part of \cite[Theorem 3.3]{Dr20a}) shows that there exists a unitary $u\in \cM\bar\otimes \cM$ such that $u\Delta(\cL(\Gamma)) u^* \subset \cM\bar\otimes \cL(\Gamma)$.
By repeating the same arguments for the inclusion $u\Delta(\cL(\Gamma)) u^* \subset \cM\bar\otimes \cL(\Gamma)$, we finish the proof.
\hfill$\blacksquare$

\section{Superrigidity for group von Neumann algebras}\label{Section:coinduced}

\subsection{W$^*$-rigidity for generalized wreath produdcts}

We first show the following strong rigidity result for groups that belong to $\mathscr{GW}_1(A_0)$.


\begin{theorem}\label{discretizationactingpart} Let $G=A_0\wr_I \Gamma\in \mathscr{ GW}_1(A_0)$ with $A_0$ and $\Gamma$ torsion free and denote by ${\rm Stab}_\G(i)=\G_0$. Let $H$ be any torsion free group and assume that $\Theta: \cL(G)\ra \cL(H)$ is a $\ast$-isomorphism. Then $H$ admits a wreath product decomposition $H = B_0 \wr_J \Lambda \in \mathscr{GW}_1(B_0)$.

Moreover, if we denote by $\La_0= {\rm Stab}_\La(j)$, one can find a group isomorphism $\delta:\Gamma \ra \Lambda$ with $\delta(\G_0)=\La_0$, a character $\eta:\Gamma\ra \mathbb T$, a $\ast$-isomorphism $\theta: \cL(A_0)\ra \cL(B_0)$, and a unitary $w\in \cL(H)$ such that for every $x=\otimes_{h\in \G/\G_0} x_{h\G_0} \in \cL(A_0^{(I)})$ and $g\in \Gamma$ we have 
\begin{equation*}
\Theta(x u_g) =\eta(g)w^* (\otimes_{h\in \G/\G_0}\theta (x_{\delta(h)\La_0}) v_{\delta(g)})w.\end{equation*} 
Here, $\{u_g \,|\, g\in \Gamma\}$ and $\{v_h \,|\, h\in \Lambda\}$ are the canonical unitaries of $\cL(\Gamma)$ and $\cL(\Lambda)$, respectively. 
\end{theorem}

{\it Proof.} Note first that Theorem \ref{Th:coind1} gives that $\Delta(\cL(A_0^{I}))\prec_{\M\bar\otimes \M} \cL(A_0^{I})\bar\otimes \cL(A_0^{I})$ and there exists a unitary $u\in \mathcal U(\M\bten \M)$ such that $u\Delta(\cL(\Gamma))u^*\subseteq  \cL(\Gamma)\bar\otimes \cL(\Gamma)$. Moreover, it follows almost directly by using the techniques introduced in \cite[Theorem 3.1 and Proposition 7.2.3]{IPV10} and \cite[Theorem 4.1]{KV15} (see the proof of \cite[Proposition 4.3]{CU18} for precise explanations) that  
\begin{equation}\label{a1a}
\text{ there exists a unitary }w\in \mathscr U(\M\bten \M) \text{ such that } 
w\Delta(\Gamma)w^*\subseteq \mathbb T (\Gamma\times \Gamma).    
\end{equation}
Next, the conclusion of the theorem follows from \eqref{a1a} by using the techniques introduced in \cite[Proposition 7.2 and Theorem 8.3]{IPV10}. The reader is referred to the explanations provided in the proof of \cite[Theorem 4.4]{CU18} on how these results can be used to get the statement.
\hfill$\blacksquare$

\begin{corollary}\label{Cor:GW} Let $G=A_0\wr_I \Gamma\in \mathscr {GW}_1(A_0)$ with $A_0$ and $\Gamma$ torsion free and $A_0$ is abstractly $W^*$-superrigid. Let $H$ be a torsion free group and let $\Theta: \cL(G)\ra \cL(H)$ be a $\ast$-isomorphism. 

Then there exist a group isomorphism $\delta:G \ra H$,  a character $\eta:G\ra \mathbb T$, an automorphism  $ \phi\in {\rm Aut}( \cL(\delta(A_0)))$, and unitary $w\in \cL(H)$ such that  $\Theta= {\rm ad} (w) \circ  \Phi_\phi \circ \Psi_{\omega,\delta}$. 

Moreover, if $A_0$ is $W^*$-superrigid, then $\Phi_\phi=\Phi_t$, for a unitary $t \in \cL(\delta(A_0))$.

\end{corollary}



{\it Proof.} Applying the previous theorem, we obtain that $H$ admits a wreath product decomposition $H = B_0 \wr_J \Lambda$ satisfying the following properties: there exist a group isomorphism $\rho_1:\Gamma \ra \Lambda$ with $\rho_1(\G_0)=\La_0$, a character $\eta_1:\Gamma\ra \mathbb T$, a $\ast$-isomorphism $\theta: \cL(A_0)\ra \cL(B_0)$ and a unitary $w\in \cL(H)$ such that for every $x=\otimes_{h\in \G/\G_0} x_{h\G_0}\in \cL(A_0^{(I)})$ and $g\in \Gamma$ we have \begin{equation}\label{halfisom}
\Theta(x u_g) =\eta_1(g)w^*( \otimes\theta(x_{\rho_1(g)\La_0}) v_{\delta(g)})w.\end{equation} 
Since $A_0$ is abstractly $W^*$-superrigid one can find a group isomorphism $\rho_2: A_0\ra B_0$, a multiplicative character $\eta_2: A_0 \ra \mathbb T$ and a $\ast$-isomorphism $\phi\in {\rm Aut}(\cL(B_0))$ such that $\theta(u_a) =\eta_2(a) \phi( v_{\rho_2(a)}) $ for all $a\in A_0$. Combining this with formula \eqref{halfisom} we see that the conclusion of the theorem holds. 
\hfill$\blacksquare$

\subsection{W$^*$-superrigidity for coinduced groups}

The goal of this subsection is to provide examples of W$^*$-superrigid coinduced groups, see Corollary \ref{Cor:coind}. This adds new examples of different structure to the prior ones found in \cite{IPV10,BV12,Be14,CI17,CD-AD20,Dr20b}. Moreover, all known W$^*$-superrigidity generalized wreath product groups are with abelian base \cite{IPV10,BV12,Be14}. In contrast, Corollary \ref{Cor:coind} provides large classes of W$^*$-generalized wreath product groups with non-amenable base.

We first introduce the following class of coinduced groups and prove the main technical result of this subsection, which is Theorem \ref{Th:coinduced}.

\begin{assumption}\label{A:assumption}
Let $G=A_0^I\rtimes \tilde\Gamma$ be the associated coinduced group of $\Gamma_0\car A_0$, where
\begin{enumerate}
    \item $\tilde\Gamma=\Gamma\times\Gamma$, where $\Gamma$ is an icc, torsion free, non-amenable, weakly amenable, bi-exact, property (T) group.
    \item $\Gamma_0<\tilde\Gamma$ is the diagonal subgroup defined by $\Gamma_0=\{(g,g)|g\in\Gamma\}.$
\end{enumerate}

\end{assumption}

\begin{theorem}\label{Th:coinduced}
Let $G=A_0^I\rtimes_\sigma \tilde\Gamma$ be as in Assumption \ref{A:assumption} and suppose that either $\Gamma_0\car A_0$ acts trivially, or $A_0$ is icc, torsion free, bi-exact and it contains an infinite property (T) subgroup $K_0$ that is $\Gamma_0$-invariant and has trivial virtual centralizer.
Let $H$ be a countable group and let $\Theta: \cL(G)\ra \cL(H)$ be a $\ast$-isomorphism. 

Then there exist subgroups $\Lambda_0<\tilde\Lambda<H$ and $\Sigma_0<H$ such that $\Sigma_0$ is normalized by $\Lambda_0$ and $H=\Sigma_0^I\rtimes_\rho\tilde\Lambda$ is the coinduced group of $\Lambda_0\car \Sigma_0.$ Moreover, there exist a group isomorphism $\delta:\tilde\Gamma\to\tilde\Lambda$ with $\delta(\Gamma_0)=\Lambda_0$, a $*$-isomorphism $\theta:\cL(A_0)\to\cL(\Sigma_0)$ that satisfies $\theta(\sigma_g(a))=\rho_{\delta(g)}(\theta(a))$ for all $g\in \Gamma_0$ and $a\in \cL(A_0)$, a character $\eta: \tilde\Gamma\to\mathbb T$ and a unitary $w\in \cL(H)$ such that $\text{for all }x\in \cL(A_0^I) \text{ and }g\in \tilde\Gamma,$ we have
\[
\Theta(x u_g) =\eta(g)w^* \theta^{\bar\otimes I}(x) v_{\delta(g)}w.
\]
Here, we denoted by $\{u_g\}_{g\in G}$ and $\{v_h\}_{h\in H}$ the canonical unitaries of $\cL(G)$ and $\cL(H)$, respectively. 
\end{theorem}

We continue with a series of preliminary results that are needed to derive the proof of Theorem \ref{Th:coinduced}. We actually prove Theorem \ref{Th:coinduced} through the following two steps.

{\bf Step I: The reconstruction of the acting group $\tilde\Gamma$.}

\begin{theorem}\label{Th:coind2}
Assume the same conditions as in Theorem \ref{Th:coinduced} and denote $\cM=\Theta(\cL(G))=\cL(H)$.

Then there exist a subgroup $\tilde\Lambda<H$ and a unitary $u\in\mathcal U(\cM)$ such that $u\Theta(\cL(\tilde\Gamma))u^*=\cL(\tilde\Lambda)$.
\end{theorem}

{\bf Notation.} Throughout the proof we use the following notation. Let $\cM$ be a tracial von Neumann algebra and $\cP\subset p\cM p$ and $\cQ\subset q\cM q$ some subalgebras. We write $\cP\prec_{\cM}^{s'}\cQ$ if $\cP\prec_{\cM}\cQ q'$, for any non-zero projection $q'\in \cQ'\cap q\cM q$. We also write $\cP\prec_{\cM}^{s,s'}\cQ$ if $\cP\prec_{\cM}^{s}\cQ$ and $\cP\prec_{\cM}^{s'}\cQ$.

{\it Proof of Theorem \ref{Th:coind2}.}
Denote by $\Delta: \cM\to\cM\bar\otimes\cM$ the $*$-homomorphism given by $\Delta(v_\lambda)=v_{\lambda}\otimes v_{\lambda}$, for any $\lambda\in H$.
For ease of notation, we identify $\Theta(G)$ with $G$ and denote $A=A_0^ I$, $\Gamma_1=\Gamma\times 1$ and $\Gamma_2=1\times\Gamma$.

By using \cite[Theorem 15.1.15 and Lemma 15.3.3]{BO08} and \cite[Lemma 3.3]{CD-AD20}, we derive that there exists $i$ such that $\Delta(\cL(\Gamma_1))\prec_M \cL(\tilde\Gamma)\bar\otimes\cL(\Gamma_i)$. We continue by applying \cite[Theorem 4.1]{DHI16} (see also \cite[Theorem 3.1]{Io11} and \cite[Theorem 3,3]{CdSS15}). Since $\Gamma_{\hat i}$ has property (T), we therefore obtain that there exists a subgroup $\Sigma<H$ with non-amenable centralizer $C_H(\Sigma)$ such that  $\cL(\Gamma_1)\prec_{\cM} \cL(\Sigma) \text{ and }\cL(\Gamma_{\hat i})\prec_{\cM} \cL(C_H(\Sigma)).$
Since $\tilde\Gamma$ is icc, we have $\cL(\tilde\Gamma)'\cap \cM=\mathbb C1$. Hence, by using \cite[Lemma 2.4(2)]{DHI16} we deduce that
\begin{equation}\label{aa10}
\cL(\Gamma_1)\prec^s_{\cM} \cL(\Sigma) \text{ and }\cL(\Gamma_{\hat i})\prec^s_{\cM} \cL(C_H(\Sigma)).    
\end{equation}

The rest of the proof is divided between the following two claims.

{\bf Claim 1.}
$    \cL(C_H(\Sigma)) \text{ is strongly non-amenable relative to }\cL(A\rtimes\Gamma_1) \text{ inside }\cM.$

{\it Proof of Claim 1.}
Assume by contradiction that $\cL(C_H(\Sigma))e \text{ is amenable relative to }\cL(A\rtimes\Gamma_1)$ for a non-zero projection $e\in \cL(C_H(\Sigma))'\cap \cM$. By passing to relative commutants in the first intertwining relation of \eqref{aa10}, we obtain by \cite[Lemma 3.5]{Va08} that $\cL(C_H(\Sigma))\prec_\cM \cL(\Gamma_2)$. By applying \cite[Lemma 2.4(3) and Lemma 2.6(2)]{DHI16} we further derive that there exists a non-zero projection $f\in \cL(C_H(\Sigma))'\cap \cM$ such that $\cL(C_H(\Sigma))f \text{ is amenable relative to }\cL(\Gamma_2)$. 

Next, note that $\cL(A\rtimes\Gamma_1)$ and $\cL(\Gamma_2)$ have amenable intersection and form a commuting square. 
Since $\cL(\Gamma_{\hat i})\prec^s_{\cM} \cL(C_H(\Sigma))$ we are in the setup of applying Theorem \ref{Th:augmentationtechniques} and we therefore obtain that $\Gamma_{\hat i}$ is amenable, contradiction. 
\hfill$\square$

{\bf Claim 2.} There exists a subgroup $H_0<H$ such that $\cL(H_0)\prec_{\cM}^s \cL(\tilde\Gamma)$ and $\cL(\tilde\Gamma)\prec_{\cM}\cL(H)$.

{\it Proof of Claim 2.}
Next, we note that Claim 1 implies that $ \cL(C_H(\Sigma)) $  is strongly non-amenable relative to  $\cL(A)$ inside $\cM.$ Hence, we can apply \cite[Theorem 2.6]{KV15} and derive that (i) $\cL(\Sigma)\prec_{\cM} \cL(A\rtimes\Gamma_0)$ or (ii) there exists a unitary $u\in \mathscr U(\cM)$ such that $u\cL(\Sigma)u^*\subset \cL(\tilde\Gamma)$. If (i) holds, note that together with \eqref{aa10} we have 
\[
\cL(\Sigma)\prec_{\cM} \cL(A\rtimes\Gamma_0) \text{ and } \cL(\Sigma)\prec_{\cM} \cL(\Gamma_i).
\]
We can use the augmentation technique as in the proof of Claim 1 and derive from $\cL(\Gamma_1)\prec_{\cM}\cL(\Sigma)$ that
\[
\Psi(\cL(\Gamma_1))\prec_{\tilde\cM\bar\otimes \cM}^s \tilde\cM\bar\otimes \cL(A\rtimes\Gamma_0) \text{ and } \Psi(\cL(\Gamma_1))\prec_{\tilde\cM\bar\otimes \cM}^s \tilde\cM\bar\otimes \cL(\Gamma_i).
\]
Since $\Gamma_0\cap\Gamma_i=1$ , it is easy to see that \cite[Lemma 2.6(3)]{DHI16} together with \cite[Proposition 2.7]{PV11} implies that $\Psi(\cL(\Gamma_1)) $ is amenable relative to $ \tilde\cM\bar\otimes 1 $. This proves that $\Gamma_1$ is amenable, contradiction. Hence, $\cL(\Sigma)\nprec_{\cM} \cL(A\rtimes\Gamma_0)$ and (ii) holds, meaning that there exists a unitary $u\in \mathscr U(\cM)$ such that $u\cL(\Sigma)u^*\subset \cL(\tilde\Gamma)$. Note that $\Omega:=\{g\in H| \mathcal O_{\Sigma}(g) \text{ is finite}\}$ is normalized by $\Sigma$ and $C_H(\Sigma)\subset\Omega\Sigma\subset {\rm Comm}^{(1)}_H(\Sigma)$.
Since $u\cL(\Sigma) u^*\nprec_{\cL(\tilde\Gamma)} \cL(\Gamma_0)$ we can use [Theorem 2.6, CD-AD20] and obtain that $u\cL(\Omega\Sigma)u^*\subset \cL(\tilde\Gamma)$.

Next, note that Claim 1 implies that $ u\cL(C_H(\Sigma))u^* \text{ is strongly non-amenable relative to }\cL(\Gamma_1) $ inside $\cL(\tilde\Gamma)$.
By using \cite[Lemma 5.2]{KV15} we derive that $u\cL(\Sigma)u^*\prec_{\cL(\tilde\Gamma)}^s \cL(\Gamma_1).$ Using this we deduce that
\begin{equation}\label{aa14}
\cL(\Gamma_2)\prec_{\cL(\tilde\Gamma)}^{s'} u\cL(\Omega\Sigma)u^*.    
\end{equation}
Indeed, if $z\in (u\cL(\Omega\Sigma)u^*)'\cap \cL(\tilde\Gamma)$ is a non-zero projection we have that $u\cL(\Sigma)u^*z\prec_{\cM} \cL(\Gamma_1).$ By passing to relative commutants, we deduce from \cite[Lemma 3.5]{Va08} that $\cL(\Gamma_2)\prec_{\cM} u\cL(\Omega)u^*z$ since $\cL(\Sigma)'\cap\cM\subset \cL(\Omega)$. Note that $\cL(\Gamma_2)\nprec_{\cM} \cL(g\tilde\Gamma g^{-1}\cap \tilde\Gamma)$, for any $g\in G\setminus\tilde\Gamma$. Otherwise, we can use Lemma \ref{L:coindcomm}(1) and get that $\cL(\Gamma_2)\prec_{\cM} \cL(\Gamma_0)$, which contradicts the diagonal embedding of $\Gamma_0<\Gamma_1\times\Gamma_2$.
Thus, we can apply \cite[Lemma 2.7]{CI17} and derive that $\cL(\Gamma_2)\prec_{\cL(\tilde\Gamma)} u\cL(\Omega)u^*z$, which proves \eqref{aa14}.

Next, by using \cite[Lemma 2.7]{CI17} as before, we obtain from \eqref{aa10} that  $\cL(\Gamma_1)\prec_{\cL(\tilde\Gamma)} u\cL(\Omega\Sigma)u^*$. By using \cite[Lemma 2.4(4)]{DHI16} we obtain that there exists a non-zero projection $z_1\in \mathcal Z( (u\cL(\Omega\Sigma)u^*)'\cap \cL(\tilde\Gamma))$ such that
\begin{equation}\label{aa15}
    \cL(\Gamma_1)\prec^{s'}_{\cL(\tilde\Gamma)} u\cL(\Omega\Sigma)u^*z_1.
\end{equation}
From \eqref{aa14} and \eqref{aa15} we derive that $((u\cL(\Omega\Sigma)u^*)'\cap \cL(\tilde\Gamma))z_1\prec_{\cL(\tilde\Gamma)}^s \cL(\Gamma_i)$ for any $i\in\overline{1,2}$. By using \cite[Lemma 2.8(2)]{DHI16} we obtain that $((u\cL(\Omega\Sigma)u^* )'\cap \cL(\tilde\Gamma))z_1\prec^s_{\cL(\tilde\Gamma)} \mathbb C 1$ and hence there exists a non-zero projection $z_2\in \mathcal Z( (u\cL(\Omega\Sigma)u^*)'\cap \cL(\tilde\Gamma))$ with $z_2\leq z_1$ such that
$((u \cL(\Omega\Sigma)u^*) '\cap \cL(\tilde\Gamma))z_2=\mathbb C z_2$. Now, we can apply \cite[Lemma 2.6]{Dr19b} and obtain that $\cL(\tilde\Gamma)\prec_{\cL(\tilde\Gamma)} u \cL(\Omega\Sigma)u^* z_2$.
In particular, we have that
    $u \cL(\Omega\Sigma)u^*\prec_{\cM}^s \cL(\tilde\Gamma) \text{ and } \cL(\tilde\Gamma)\prec_{\cM} u \cL(\Omega\Sigma)u^*.$ This ends the proof of the claim. \hfill$\square$.

Finally, we recall that $\tilde\Gamma<G$ are icc groups that satisfy ${\rm Comm}^{(1)}_G(\tilde\Gamma)=\tilde\Gamma$ by Lemma \ref{L:coindcomm}. Using Claim 2 we can apply \cite[Theorem 2.11]{CD-AD20} to show that there
 exist a subgroup $\Lambda<H$ and a unitary $u_1\in\mathscr U(\cM)$ such that $u_1 \cL(\tilde\Gamma)u_1^*= \cL(\Lambda)$.
\hfill$\blacksquare$

\begin{theorem}\label{Th:coind3}
Assume the same conditions as in Theorem \ref{Th:coinduced} and denote $\cM=\Theta(\cL(G))=\cL(H)$.

Then there exist a subgroup $\tilde\Lambda<H$ and a unitary $u\in\mathcal U(\cM)$ such that $\mathbb T u\Theta(\tilde\Gamma)u^*=\mathbb T \tilde\Lambda$.
\end{theorem}

Before proceeding to the proof of Theorem \ref{Th:coind3} we need to recall the notion of a height of an element. Following \cite[Section 4]{Io10}, the height of an element $x\in \cL(G)$ is defined as 
$$h_G(x)=\underset{g\in G}{\rm max}|\tau(xu_g^*)|,$$
where $\{u_g\}_{g\in G}$ are the canonical unitaries that generate $\cL(G)$.

This notion has been of crucial use in several other classification results in von Neumann algebras \cite{Io10,IPV10,KV15,CI17,CU18,CDK19,CDHK20,CD-AD20}.
The main ingredient of the proof of Theorem \ref{Th:coind3} is \cite[Theorem 3.1]{IPV10} which asserts that if $H$ is any countable group such that $\cL(G)=\cL(H)$ and $h_G(H):=\underset{h\in H}{\rm inf} h_G(v_h)>0$, then there exists a unitary $u\in\cL(G)$ such that $\mathbb T G=u \mathbb T H u^*$. 
Here, we denoted by $\{v_h\}_{h\in H}$ the canonical unitaries that generate $\cL(H)$.

{\it Proof of Theorem \ref{Th:coind3}.} 
To simplify the writing, we identify $\Theta(G)$ with $G$ and
denote $\Gamma_1=\Gamma\times 1$ and $\Gamma_2=1\times\Gamma$.
By Theorem \ref{Th:coind2} there is a unitary $u\in\mathcal U(\cM)$ such that $  u\cL(\tilde\Gamma) u^*=\mathbb \cL(\Lambda)$ and therefore the proof of Theorem \ref{Th:coind3} will follow from \cite[Theorem 3.1]{IPV10} by showing that
\begin{equation}\label{height1}
h_\Lambda(u\tilde\Gamma u^*)>0.     
\end{equation}

Next, by using \cite[Theorem A]{CdSS15} we obtain a decomposition $\Lambda=\Lambda_1\times \Lambda_2$, a unitary $x\in\ \mathcal U(\cM)$ and a decomposition $\cL(\tilde\Gamma)=\cL(\Gamma_1)^{t_1}\bar\otimes \cL(\Gamma_2)^{t_2}$ for some $t_1,t_2>0$ with $t_1t_2=1$ such that 
\begin{equation}\label{split1}
\cL(\Lambda_1)=x\cL(\Gamma_1)^{t_1}x^* \text{ and } \cL(\Lambda_2)=x \cL(\Gamma_2)^{t_2}x^*.    
\end{equation}

Using \cite[Lemma 2.4]{CU18} we derive that \eqref{height1} is equivalent to $h_\Lambda(x\tilde\Gamma x^*)>0$. Since equation \eqref{split1} shows that the heights of $x\Gamma_1 x^*$ and $x\Gamma_2 x^*$ do not interact, it follows that \eqref{height1} is equivalent to $h_\Lambda(x\Gamma_i x^*)>0$, for any $1\leq i\leq 2$. For ease of notation, we will prove this only for $i=1$.

Denote by $\Delta: \cM\to\cM\bar\otimes\cM$ the $*$-homomorphism given by $\Delta(v_\lambda)=v_{\lambda}\otimes v_{\lambda}$, for any $\lambda\in H$. By letting $w=(x^*\otimes x^*)\Delta(x)$, we obtain from \eqref{split1} that $w\Delta(\cL(\tilde\Gamma))w^*\subset \cL(\tilde\Gamma)\bar\otimes \cL(\tilde\Gamma)$ and
\begin{equation}\label{split2}
    w(\Delta(\cL(\Gamma_1)^{t_1}))w^*\subset \cL(\Gamma_1)^{t_1}\bar\otimes \cL(\Gamma_1)^{t_1} \text{ and } w(\Delta(\cL(\Gamma_2)^{t_2}))w^*\subset \cL(\Gamma_2)^{t_2}\bar\otimes \cL(\Gamma_2)^{t_2}.
\end{equation}

By using \cite[Lemma 2.4]{CU18} and \cite[Lemma 2.5]{CU18} it follows that $h_{\tilde\Gamma\times\tilde\Gamma}(w\Delta(\Gamma_1)w^*)>0$ implies $h_\Lambda(x\Gamma_1 x^*)>0$. Hence, it remains to show that $h_{\tilde\Gamma\times\tilde\Gamma}(w\Delta(\Gamma_1)w^*)>0$.
We proceed now by following the strategy of the proof of \cite[Lemma 5.10]{KV15} (see also the proof of \cite[Theorem 4.1]{Io10}). For the convenience of the reader we will include all the details. Assume by contradiction that $h_{\tilde\Gamma\times\tilde\Gamma}(w\Delta(\Gamma_1)w^*)>0$. This shows that there exists a sequence of elements $(x_n)_n\subset  w\Delta(\Gamma_1)w^*$ such that $h_{\tilde\Gamma\times\tilde\Gamma}(x_n)\to 0$.
We continue by proving that for every sequence of unitaries $(w_n)_n\subset \cL(\tilde \Gamma\times\tilde\Gamma)$ and for every $a\in \cM\bar\otimes \cM$ with $E_{\cM\bar\otimes \cL(\tilde\Gamma)}(a)=E_{\cL(\tilde\Gamma)\bar\otimes \cM}(a)=0$, we have that
\begin{equation}\label{all}
    \lim_{n\to 0} \|E_{\cL(A_0)^I\bar\otimes \cL(A_0)^I}(x_naw_n)\|_2=0.
\end{equation}
For proving \eqref{all}, note that it is enough to assume that $a\in \cL(A_0)^I\bar\otimes \cL(A_0)^I$ with $E_{\cL(A_0)^I\otimes 1}(a)=E_{1\otimes \cL(A_0)^I}(a)=0$. Moreover, since $a$ can be approximated by elementary tensors, we may further assume that $a\in (\cL(A_0)\ominus\mathbb C 1)^{F_1}\otimes (\cL(A_0)\ominus\mathbb C 1)^{F_2}$, for some finite subsets $F_1,F_2\subset I$. In this case, $(\sigma_{g_1}\otimes\sigma_{g_2})(a)\perp (\sigma_{h_1}\otimes\sigma_{h_2})(a)$, for all $(g_1,g_2),(h_1,h_2)\in \tilde\Gamma\times \tilde\Gamma$ with the property that $(g_1 F_1, g_2 F_2)\neq (h_1 F_1, h_2 F_2).$ Here and after, we denote by $\sigma$ the action $\tilde\Gamma\car A_0^I$.

For a subset $S\subset \tilde\Gamma\times\tilde\Gamma$, we let $P_S$ be the orthogonal projection of $L^2(\cM\bar\otimes \cM)$ onto the closed linear span of $L^2(\cL(A_0^I)\bar\otimes \cL(A_0^I))u_g$ with $g\in S$.
Let $\epsilon>0$. By \eqref{split2} there exists a finite subset $F_0\subset \Gamma_2$ with the property that $\|v-P_{F}(v)\|_2\leq\epsilon$, for any unitary $v\in w(\Delta(\cL(\Gamma_1)))w^*$, where $F=\Gamma_1\times F_0\times\Gamma_1\times F_0$.

Next, we note that for all $(h_1,h_2)\in \tilde\Gamma\times\tilde\Gamma$ the set $\{(g_1,g_2)\in F| (g_1 F_1, g_2 F_2)=(h_1 F_1, h_2 F_2)  \}$ contains at most $k:=|F_0|^2|F_1| ^2 |F_2|^2$ elements. This follows from the fact that the action $\tilde\Gamma\car I$ is isomorphic to the left-right action $\Gamma\times\Gamma\car\Gamma$. Using the decomposition $\cM\bar\otimes \cM=(\cL(A_0)^I\bar\otimes \cL(A_0)^I)\rtimes (\tilde\Gamma\times\tilde\Gamma)$, it follows that $E_{\cL(A_0)^I\bar\otimes \cL(A_0)^I}(P_F(x_n)aw_n)=\sum_{g\in F}\tau(x_nu_g^*)\tau(w_nu_g)(\sigma\otimes\sigma)_g(a)$.

Therefore, for all $h\in \tilde\Gamma\times\tilde\Gamma$ we have that
\[
|\langle E_{\cL(A_0)^I\bar\otimes \cL(A_0)^I}(P_F(x_n)aw_n),(\sigma\otimes\sigma)_{h}(a)\rangle|\leq k h_{\tilde\Gamma\times\tilde\Gamma}(x_n).
\]
We further deduce that for all $h\in \tilde\Gamma\times\tilde\Gamma$ we have that $\alpha:=\|E_{\cL(A_0)^I\bar\otimes \cL(A_0)^I}(P_F(x_n)aw_n)\|_2^2$ is majorated by
\begin{equation*}
\begin{split}
\alpha &\leq \sum_{h\in F} |\langle E_{\cL(A_0)^I\bar\otimes \cL(A_0)^I}(P_F(x_n)aw_n),\tau(x_nu_h^*)\tau(w_nu_h)(\sigma\otimes\sigma)_{h}(a)\rangle|\\
&\leq k h_{\tilde \Gamma\times \tilde\Gamma}(x_n)\to 0.
\end{split}
\end{equation*}
Since $\|v-P_{F}(v)\|_2\leq\epsilon$, for any unitary $v\in w(\Delta(\cL(\Gamma_1)))w^*$, the equation \eqref{all} follows.

Next, we note that Theorem \ref{Th:coind1}(1) gives that $\Delta(\cL(A_0)^I)\prec_{\cM\bar\otimes \cM}^s \cL(A_0)^I\bar\otimes \cL(A_0)^I$. Note also that $\Delta(\cL(A_0)^I)\nprec_{\cM\bar\otimes \cM}\cM\bar\otimes \cL(\tilde\Gamma)$ and $\Delta(\cL(A_0)^I)\nprec_{\cM\bar\otimes \cM} \cL(\tilde\Gamma)\bar\otimes \cM$.
Otherwise, by using Theorem \ref{Th:coind1}(2) and \cite[Lemma 2.3]{BV12} we would get that $\Delta(\cM)\prec_{\cM\bar\otimes \cM}\cM\bar\otimes \cL(\tilde\Gamma)$ or $\Delta(\cM)\prec_{\cM\bar\otimes \cM} \cL(\tilde\Gamma)\bar\otimes \cM$. Using \cite[Proposition 7.2]{IPV10} it would follow that $\cL(\tilde \Gamma)\subset \cM$ has finite index, which is not possible.
Hence, we can find a finite subset $S\subset \tilde\Gamma\times\tilde\Gamma$ such that $\|b-P_S(b)\|_2\leq 1/8$, for any unitary $b\in\mathscr U(w\Delta(\cL(A_0)^I)w^*)$ and that there exists $d\in\mathcal U(w\Delta(\cL(A_0)^I)w^*)$ such that $\|E_{\cM\bar\otimes \cL(\tilde\Gamma)}(d)\|_2\leq 1/16$ and $\|E_{\cL(\tilde\Gamma)\bar\otimes \cM}(d)\|_2\leq 1/16$. Hence, $\|x_ndx_n^*-P_S(x_ndx_n^*)\|_2\leq 1/8,$ for any $n\ge 1$.
Since $E_{\cM\bar\otimes \cL(\tilde\Gamma)}$ and $E_{\cL(\tilde\Gamma)\bar\otimes \cM}$ commute, it follows that $d_0:=(1- E_{\cM\bar\otimes \cL(\tilde\Gamma)})\circ (1-E_{\cL(\tilde\Gamma)\bar\otimes \cM})(d)$ verifies $E_{\cM\bar\otimes \cL(\tilde\Gamma)}(d_0)=E_{\cL(\tilde\Gamma)\bar\otimes \cM}(d_0)=0$ and $\|d_0\|_2\ge 7/8$. We therefore obtain that
\begin{equation}\label{f1}
    \|x_nd_0x_n^*-P_S(x_nd_0x_n^*)\|_2\leq 1/2.
\end{equation}

By using equation \eqref{all} we obtain that $\|P_S(x_nd_0x_n^*)\|_2^2=\sum_{g\in S}\|E_{\cL(A_0)^I\bar\otimes \cL(A_0)^I}(x_nd_0x_n^* u_g^*)\|_2^2\to 0.$ In combination with \eqref{f1} we obtain that $7/8\leq\|d_0\|_2\leq 1/2$, contradiction.
\hfill$\blacksquare$

{\bf Step II: The reconstruction of the coinduced group $A_0^I\rtimes\tilde\Gamma$.}

In this step we complete the proof of Theorem \ref{Th:coinduced}. Before doing that, we first prove the following structural result for group von Neumann algebras of coinduced groups that satisfy Assumption \ref{A:assumption}.

\begin{theorem}\label{Th:coind4}
Let $G=A_0^I\rtimes \tilde\Gamma$ be as in Assumption \ref{A:assumption}.
Let $H$ be a countable group and let $\Theta: \cL(G)\ra \cL(H)$ be a $\ast$-isomorphism. 

Then $H$ admits a semi-direct product decomposition $H=\Sigma\rtimes\tilde\Lambda$ and there exist a group isomorphism $\delta:\tilde\Gamma\to\tilde\Lambda$, a $*$-isomorphism $\Theta_0:\cL(A_0^I)\to\cL(\Sigma)$, a character $\eta: \tilde\Gamma\to\mathbb T$ and a unitary $w\in \cL(H)$ such that $\text{for all }x\in \cL(A_0^I) \text{ and }g\in \tilde\Gamma,$ we have
\[
\Theta(x u_g) =\eta(g)w^* \Theta_0(x) v_{\delta(g)}w.
\]
Here, we denoted by $\{u_g\}_{g\in G}$ and $\{v_h\}_{h\in H}$ the canonical unitaries of $\cL(G)$ and $\cL(H)$, respectively. 
\end{theorem}

{\it Proof.} Denote $\cM=\Theta(\cL(G))=\cL(H)$.
Using Theorem \ref{Th:coind3} we can find a subgroup $\tilde\Lambda<H$, a group isomorphism $\delta:\tilde\Gamma\to\tilde\Lambda$, a character $\omega:\tilde\Gamma\to \mathbb T$ and a unitary $u_0\in\mathscr U(\cM)$ such that 
    $\Theta(u_g)=\omega(g)u_0^* v_{\delta(g)} u_0, \text{ for all }g\in\tilde\Gamma.$ By replacing 
   the $*$-isomorphism $\Theta$ by the $*$-isomorphism $\Theta_{u_0,\omega}$ given by $\Theta_{u_0,\omega}(u_g)=\omega(g)u_0\Theta(u_g)u_0^*$ for $g\in G$, we have that
 \begin{equation}\label{tt1}
    \Theta(u_g)= v_{\delta(g)}, \text{ for all }g\in\tilde\Gamma.
\end{equation}
Denote by $\Delta: \cM\to\cM\bar\otimes\cM$ the $*$-homomorphism given by $\Delta(v_\lambda)=v_{\lambda}\otimes v_{\lambda}$, for any $\lambda\in H$. From \eqref{tt1} we derive that $\Delta(\Theta(u_g))=\Theta(u_g)\otimes \Theta(u_g)$, for any $g\in\tilde\Gamma$ and from Theorem \ref{Th:coind1} we get that $\Delta(\Theta(\cL(A_0^I)))\prec_{\cM\bar\otimes \cM}^s \Theta(\cL(A_0^I))\bar\otimes \Theta(\cL(A_0^I))$. We can proceed as in Claim 4 of the proof of \cite[Theorem 4.4]{CU18} and deduce that $\Delta(\Theta(\cL(A_0^I)))\subset \Theta(\cL(A_0^I)) \bar\otimes \Theta(\cL(A_0^I)).$ Using \cite[Lemma 7.1.2]{IPV10} we further deduce that there exists a subgroup $\Sigma<H $ such that $\Theta(\cL(A_0^I))=\cL(\Sigma)$ and the conclusion follows.
\hfill$\blacksquare$

{\bf Proof of Theorem \ref{Th:coinduced}.}
Using Theorem \ref{Th:coind4} we can assume without loss of generality that $H$ admits a semi-direct product decomposition $H=\Sigma\rtimes\tilde\Lambda$ and there exists a group isomorphism $\delta:\tilde\Gamma\to\tilde\Lambda$ such that $\Theta(\cL(A_0^I))=\cL(\Sigma)$ and $\Theta(u_g)=v_{\delta(g)}$, for all $g\in\tilde\Gamma$. For ease of notation, we are identifying $\Theta(x)$ with $x$, for any $x\in\cL(G)$, and we therefore have $\cM:=\cL(G)=\cL(H)$.

Assume first that $A_0$ is icc, torsion free, bi-exact and it contains an infinite property (T) subgroup $K_0$ that is $\Gamma_0$-invariant and has trivial virtual centralizer.
Throughout the proof we will be using the following notation: for a subgroup $K<A_0$, we denote by $K^i<A_0^i$ the canonical copy of $K$ on the $i$th component of $A_0^I$.
The proof of our theorem follows now from the following three claims.

{\bf Claim 1.} For any $i\in I$ there exists a subgroup $\Sigma^i<\Sigma$ such that $\cL(K_0^i)\subset \cL(\Sigma^i)\subset \cL(A_0^i)$.

{\it Proof of Claim 1.} Let $i\in I$. Since $K_0^i$ has property (T), then for any $\epsilon>0$ there exists a finite subset $F\subset I$ such that 
\begin{equation}\label{o1}
    \|P_{A_0^F\times A_0^F}(\Delta(k))-\Delta(k)\|_2\leq \epsilon, \text{ for all }k\in (\cL(K_0^i))_1.
\end{equation}
The following argument is inspired by the proof of Claim 4.5 in \cite[Theorem 4.4]{CU18}.
Let $g\in {\rm Stab}_{\tilde\Gamma}(i)$. From \eqref{o1} we get that $\|P_{A_0^F\times A_0^F}(\Delta(u_gku_g^*)-\Delta(u_gku_g^*))\|_2\leq \epsilon,$  for all $k\in (\cL(K_0^i))_1$. As $\Delta(u_g)=v_{\delta(g)}\otimes v_{\delta(g)}=u_g\otimes u_g$, we further derive that
\begin{equation}\label{o3}
    \|P_{A_0^{gF}\times A_0^{gF}}(\Delta(k))-\Delta(k)\|_2\leq \epsilon, \text{ for all }k\in (\cL(K_0^i))_1 \text{ and }g\in{\rm Stab}_{\tilde \Gamma}(i).
\end{equation}
Combining \eqref{o1} and \eqref{o3} and the fact that $P_{S_1}\circ P_{S_2}=P_{S_1\cap S_2},$ for all subsets $S_1,S_2\subset A_0^I\times A_0^I$, we obtain that 
\begin{equation}\label{o4}
    \|P_{(A_0^F\times A_0^F)\cap(A_0^{gF}\times A_0^{gF})}(\Delta(k))-\Delta(k)\|_2\leq 2\epsilon, \text{ for all }k\in (\cL(K_0^i))_1 \text{ and }g\in{\rm Stab}_{\tilde \Gamma}(i).
\end{equation}
Since for all $j\neq i$ we have that ${\rm Stab}_{\tilde\Gamma}(i)\cdot j$ is an infinite set, we see that there exists $g\in {\rm Stab}_{\tilde\Gamma}(i)$ such that $gF\cap F\subset \{i\}$ and hence  $(A_0^F\times A_0^F)\cap(A_0^{gF}\times A_0^{gF})\subset A_0^i\times A_0^i.$ Using this in \eqref{o4}, we derive that $ \|P_{A_0^i\times A_0^i}(\Delta(k))-\Delta(k)\|_2\leq 2\epsilon, \text{ for all }k\in (\cL(K_0^i))_1$. As $\epsilon>0$ is arbitrary, that $\Delta(\cL(K_0^i))\subset \cL(\Sigma^i)\bar\otimes \cL(\Sigma^i)$. Using \cite[Lemma 2.8]{CDK19} we derive that there exists a subgroup $\Sigma^i<\Sigma$ that satisfies $\cL(K_0^i)\subset \cL(\Sigma^i)\subset \cL(A_0^i)$.
\hfill$\square$

{\bf Claim 2.} For any $i\in I$ there exists a subgroup $\Omega_i<\Sigma$ such that $\cL(A_0^{I\setminus\{i\}})\subset \cL (\Omega_i)\subset \cL(A_0^I)$ and $[\cL(\Omega_i):\cL(A_0^{I\setminus\{i\}})]<\infty$.

{\it Proof of Claim 2.} Passing to relative commutants inside $\cL(A_0^i)$ in Claim 1, we obtain that $$\cL(\Sigma^i)'\cap \cL(A_0^i)\subset \cL(K_0^i)'\cap \cL(A_0^i)\subset \cL(vC_{A_0^i}(\Sigma^i))=\mathbb C 1,$$ since $K_0^i$ has trivial virtual centralizer in $A_0^i.$ This implies that $\cL(\Sigma^i)'\cap \cL(A_0^i)=\mathbb C1$. Passing to relative commutants inside $\cL(A_0^I)=\cL(\Sigma)$ in Claim 1, we obtain that $\cL(\Sigma^i)'\cap \cL(\Sigma)=\cL(A_0^{I\setminus\{i\}})$. Let $\Omega_i=vC_\Sigma(\Sigma^i)$ and note that $$\cL(A_0^{I\setminus\{i\}})=\cL(\Sigma^i)'\cap \cL(\Sigma)\subset \cL(\Omega_i)\subset \cL(A_0^I).$$ Since $\cL(A_0)$ is a II$_1$ factor, we can apply \cite[Theorem A]{Ge95} and derive that there exists a subalgebra $B_i\subset \cL(A_0^i)$ such that 
\begin{equation}\label{oo1}
    \cL(\Omega_i)=\cL(A_0^{I\setminus\{i\}})\bar\otimes B_i.
\end{equation}
Note that $\cL(\Sigma^i\Omega_i)'\cap \cL(A_0^I)=\mathbb C 1$, which implies that $\Sigma^i\Omega_i$ is icc. Since $vC_{\Omega_i}(\Omega_i)\subset vC_{\Sigma^i\Omega_i}(\Sigma^i\Omega_i)$, we derive that $\Omega_i$ is icc as well.
Since $\Omega_i=vC_\Sigma(\Sigma^i)$, one can find an increasing sequence of groups $\Omega_i^s<\Omega_i^{s+1}<\dots<\Omega_i$ with $\cup_{s\ge 1}\Omega_i^s=\Omega_i$ such that $C_{\Sigma^i}(\Omega_i^{s+1})< C_{\Sigma^i}(\Omega_i^s)<\dots< \Sigma^i$ and $[\Sigma^i: C_{\Sigma^i}(\Omega_i^s)]<\infty$, for all $s\ge 1$. Since
$$
\cL(A_0^{I\setminus\{i\}})=\cL(\Sigma^i)'\cap \cL(\Sigma)\subset \cL(C_{\Sigma^i}(\Omega_i^s))'\cap \cL(\Sigma)\subset \cL(A_0^I),
$$
we can use \cite[Theorem A]{Ge95} once again to deduce that there exists a subalgebra $B_i^s\subset \cL(A_0^i)$ such that 
\begin{equation}\label{oo2}
\cL(C_{\Sigma^i}(\Omega_i^s))'\cap \cL(\Sigma)=\cL(A_0^{I\setminus\{i\}})\bar\otimes B^s_i.    
\end{equation}
Next, notice that $B_i^s,\cL(C_{\Sigma^i}(\Omega_i^s))\subset \cL(A_0^i)$ are commuting subalgebras and $\cL(C_{\Sigma^i}(\Omega_i^s))$ has no amenable direct summand since $[\Sigma^i: C_{\Sigma^i}(\Omega_i^s)]<\infty$ and $\Sigma^i$ is non-amenable. The assumption implies by \cite{Oz03} that $\cL(A_0^i)$ is solid, and hence, $B_i^s$ is a completely atomic von Neumann algebra for any $s\ge 1.$ By \cite[Proposition 3.2]{OP07} it follows that the action by conjugation $\Sigma^i\car B_i^s$ is weakly compact in the sense of \cite{OP07}.

From \eqref{oo2} we have  $\cL(\Omega_i^s)\subset \cL(A_0^{I\setminus\{i\}})\bar\otimes B^s_i$, and hence,
$$
\cL(\Omega_i)=\overline{\cup_{s\ge 1}\cL(\Omega_i^s)}^{\rm SOT}\subset \cL(A_0^{I\setminus\{i\}}) \bar\otimes \overline{\cup_{s\ge 1}B_i^s}^{\rm SOT}.
$$
Using \eqref{oo1} we further derive that $B_i\subset \overline{\cup_{s\ge 1}B_i^s}^{\rm SOT}$. Moreover, since $B_i$ is normalized by $\Sigma^i$, we obtain from \cite[Lemma 2.7]{DHI16} and \cite[Proposition 3.2]{OP07} that the action $\Sigma^i\car B_i$ by conjugation is weakly compact. Using that $A_0$ is bi-exact, we obtain from the proof of \cite[Theorem 6.1]{CSU11} that $B_i$ is not diffuse, and hence, it is finite dimensional.  
\hfill$\square$

Since $\Omega_i$ is icc, we can therefore apply \cite[Theorem 4.6]{CdSS17} and derive that there exists a subgroup $\Psi_i<\Sigma$ such that 
\begin{equation}\label{oo3}
\cL(A_0^i)\prec^s_{\cL(\Sigma)} \cL(\Psi_i) \text{ and }\cL(\Psi_i)\prec_{\cL(\Sigma)}^s\cL(A_0^i).    
\end{equation}

{\bf Claim 3.} For any $i\in I$ 
there exists a subgroup $\Sigma_0^i<\Sigma$ such that $\cL(\Sigma_0^i)=\cL(A_0^i)$.

{\it Proof of Claim 3.} Let $i\in I$. From relation \eqref{oo3} and \cite[Remark 2.2]{DHI16} we deduce that $\Delta(\cL(A_0^i)) \prec^s_{\cL(A_0^I)\bar\otimes \cL(A_0^I)} \Delta(\cL(\Psi_i)).$ Since $\Delta(\cL(\Psi_i))\subset \cL(\Psi_i)\bar\otimes \cL(\Psi_i)$, we conclude from \eqref{oo3} and \cite[Lemma 3.7]{Va08}  that $\Delta(\cL(A_0^i)) \prec^s_{\cL(A_0^I)\bar\otimes \cL(A_0^I)} \cL(A_0^i)\bar\otimes \cL(A_0^i).$ It follows that for any $\epsilon>0$ there exists a finite subset $G\subset I$ with $i\in G$ such that 
\begin{equation*}
    \|P_{A_0^G\times A_0^G}(\Delta(x))-\Delta(x)\|_2\leq \epsilon, \text{ for all }x\in (\cL(A_0^i))_1.
\end{equation*}
By proceeding exactly as in the proof of Claim 1, we derive that $\Delta(\cL(A_0^i)) \subset \cL(A_0^i)\bar\otimes \cL(A_0^i)$, and hence, the claim follows from \cite[Lemma 7.1]{IPV10}.
\hfill$\square$

Let $i_0=\Gamma_0\in I$ and denote $\Sigma_0=\Sigma_0^{i_0}$.
Since $\tilde\Gamma\car I$ is transitive and $\delta:\tilde\Gamma\to\tilde\Lambda$ is a group isomorphism such that $\Theta(u_g)=v_{\delta(g)}$, we get that $\Sigma_0^i\cong \Sigma_0$, for any $i\in I$, and hence, we can assume $\Sigma =\Sigma_0^I$.  Denote $\delta(\Gamma_0)=\Lambda_0<\tilde \Lambda$. Since $\Gamma_0\car A_0$ acts by group automorphisms, it follows that $\Lambda_0\car \Sigma_0$ acts by group automorphisms as well. Moreover, we have that $H=\Sigma_0^I\rtimes\tilde\Lambda$, where
$\tilde\Lambda\car \Sigma_0^I$ is the coinduced action of $\Lambda_0\car \Sigma_0$.

We continue by explicitly writing how the coinduced actions $\tilde\Gamma\car \cL(A_0)^I$ and $\tilde\Lambda\car \cL(\Sigma_0)^I$ are defined. Let $\pi_i:\cL(A_0)\to \cL(A_0^I)$ be the embedding of $\cL(A_0)$ as the $i$'th tensor factor and let $c:\tilde\Gamma\times I\to \Gamma_0$ be a cocycle such that
\begin{equation}\label{f1}
u_g \pi_i(a) u_g^*=\pi_{gi}(u_{c(g,i)}au_{c(g,i)}^*), \text{ for all }g\in\tilde\Gamma, i\in I \text{ and }a\in \cL(A_0).     
\end{equation}
We still denote by $\pi_i$ the embedding $\pi_i:\cL(\Sigma_0)\to \cL(\Sigma_0^I)$ of $\cL(\Sigma_0)$ as the $i$'th tensor factor and let $d:\tilde\Lambda\times I\to \Lambda_0$ be a cocycle such that
\begin{equation}\label{f2}
v_h \pi_i(b) v_h^*=\pi_{hi}(v_{d(h,i)}bv_{d(h,i)}^*), \text{ for all }h\in\tilde\Gamma, i\in I \text{ and } b\in\cL(\Sigma_0).     
\end{equation}
Without loss of generality we may assume that $\delta(c(g,i))=d(\delta(g),i)$, for all $g\in\tilde\Gamma,i\in I.$

Note that Claim 3 gives a family of automorphisms  $\{\theta_i:\cL(A_0)\to \cL(\Sigma_0)\}_{i\in I}$ such that $\Theta_{|\cL(A_0^I)}=\bar\otimes_{i\in I} \theta_i$, and hence, $\Theta\circ \pi_i=\pi_i\circ \theta_i$, for any $i\in I$. From \eqref{f1} and \eqref{f2} we derive that for all $g\in\tilde\Gamma,i\in I$ and $a\in \cL(A_0)$ we have $v_{\delta(g)}(\pi_i\circ \theta_i)(a)v_{\delta(g)}^*=(\pi_{gi}\circ\theta_{gi})(u_{c(g,i)}au_{c(g,i)}^*)$. Thus,
\begin{equation}\label{equivariant}
\theta_{gi_0}(c(g,i_0)\cdot a)=d(\delta(g),i_0) \cdot \theta_{i_0}(a), \text{ for all } g\in\tilde\Gamma \text{ and }a\in \cL(A_0).     
\end{equation}

Hence, we obtain that $\theta_{i_0}(g\cdot a)=\delta(g)\cdot \theta_{i_0}(a)$, for all $g\in\Gamma_0$ and $a\in \cL(A_0)$.  
Since $\tilde\Gamma\car I$ is transitive, \eqref{equivariant} implies that we can assume that $\theta_i=\theta_j$, for all $i\neq j$, and hence, the conclusion follows in this case.

Assume now that $\Gamma_0\car A_0$ acts trivially. By proceeding as in Claims 4.5-4.6 from the proof of \cite[Theorem 4.4]{CU18} we derive the conclusion of Claim 3, and hence, the theorem follows.
\hfill$\blacksquare$

As a consequence of Theorem \ref{Th:coinduced}, we obtain the following W$^*$-superrigid coinduced groups.

\begin{corollary}\label{Cor:coind}
Let $G=A_0^I\rtimes \tilde\Gamma$ be as in Assumption \ref{A:assumption}. Then the following hold:
\begin{enumerate}
    \item If $A_0$ is abstractly $W^*$-superrigid and either $\Gamma_0\car A_0$ acts trivially (i.e. $\tilde \Gamma\car A_0^I$ is a generalized Bernoulli action), or $A_0$ is icc, torsion free, bi-exact and it contains an infinite property (T) subgroup that is $\Gamma_0$-invariant and has trivial virtual centralizer, then $G$ is $W_{\rm aut}^*$-superrigid.
    \item If $A_0\rtimes\Gamma_0$ is $W^*$-superrigid such that ${\rm Comm}^{(1)}_{A_0\rtimes\Gamma_0}(\Gamma_0)=\Gamma_0$ and $A_0$ contains a property (T) subgroup that is $\Gamma_0$-invariant and has trivial virtual centralizer, then $G$ is $W^*$-superrigid.
\end{enumerate}
\end{corollary}

{\it Proof.} (1) This part follows directly from Theorem \ref{Th:coinduced}.

(2) Let $H$ be any countable group, let $\Theta:\cL(G)\to\cL(H)$ be a $*$-isomorphism and denote $\cM=\Theta(\cL(G))=\cL(H)$. From Theorem \ref{Th:coinduced} we derive that there exist a unitary $w_1\in \cL(H)$ and some subgroups $\Lambda_1,\Lambda_2<H$ such that 
\begin{equation}\label{xx1}
\mathbb T\Theta(\tilde\Gamma)=\mathbb Tw_1^* \Lambda_1 w_1    
\end{equation}
and $\Theta(\cL(A_0\rtimes\Gamma_0))=w_1^* \cL(\Lambda_2)w_1$. Since $A_0\rtimes\Gamma_0$ is $W^*$-superrigid, there exists a unitary $w_2\in \cL(H)$ such that \begin{equation}\label{xx2}
 \mathbb T \Theta(A_0\rtimes\Gamma_0)=\mathbb T w_2^* \Lambda_2w_2.   
\end{equation}
Note that \cite[Lemma 2.10]{CD-AD20} shows that ${\rm Comm}_{\tilde\Gamma}^{(1)}(\Gamma_0)=\Gamma_0$.
Using Lemma \ref{L:coindcomm}(3) and our assumption we get that ${\rm Comm}_{G}^{(1)}(\Gamma_0)=\Gamma_0.$
Since $\Gamma_0=\tilde\Gamma\cap (A_0\rtimes\Gamma_0)<G$ are icc groups, equations \eqref{xx1},\eqref{xx2} allows us to apply
\cite[Theorem 8.5]{CD-AD20} and obtain a unitary $w\in\cL(H)$ such that $\mathbb T\Theta(G)=\mathbb T w^*(\Lambda_1\vee\Lambda_2)w$. This shows that $H=\Lambda_1\vee\Lambda_2$ and finishes the proof of the corollary.
\hfill$\blacksquare$

\section{Superrigidity for Group $\mathbb C^*$-algebras}

In \cite[Corollary B]{CI17} were introduced the first family of examples of non-amenable groups that can be completely recovered from their reduced $C^*$-algebras. These family consists of uncountable many groups that appear as specific amalgamated free products and could contain torsion. In this section we introduce new examples of $C^*$-superrigid groups from the realm of generalized wreath products as introduced in \cite{IPV10}. Notice that other examples of amalgams, HNN extensions and other semi-direct products have been found very recently in \cite{CD-AD20}. A point of contrast between the aforementioned and our results is that our methods, while still von Neumann algebraic in nature, use in an essential way both facts that the groups involved have unique trace and also satisfy the Baum-Connes conjecture. Moreover, it is the first instance when the $C^*$-superrigidity statement is slightly more complex than the ones presented in \cite{CI17,CD-AD20} as it involves a family of non-trivial automorphisms of Houghton type. 



\begin{corollary}\label{C^*} Let $G=A_0\wr_I \Gamma\in \mathscr {GW}_1(A_0)$, where  $A_0$  is $W^*$-superrigid with trivial amenable radical and C$_r^*(G)$ has no non-trivial projections. 

Then for any group $H$ and any $\ast$-isomorphism  $\Theta: C^*_r(G)\ra C_r^*(H)$ one can find a group isomorphism $\delta:G\ra H$,  a character $\eta:G\ra \mathbb T$ and  unitaries $ t\in \cL(H)$, $w\in \cL( \delta(A_0))$ such that $
\Theta= {\rm ad} (t) \circ   \Phi_w \circ \Psi_{\eta, \delta}$. In particular, $G$ is abstractly $C^*$-superrigid and not $C^*$-superrigid.
\end{corollary}

{\it Proof.} Fix $\Theta: C^*_r(G)\ra C_r^*(H)$ a $\ast$-isomorphism. Since C$_r^*(G)$ has no non-trivial projections then so is C$_r^*(H)$; in particular, $H$ is torsion free. Since $A_0$ has trivial amenable radical then  Proposition \ref{trivialamenableradical} implies that $G$ has trivial amenable radical as well. So using \cite[Theorem 1.3]{BKKO14}  it follows that  C$_r^*(G)$ has unique trace. This further implies that $\Theta$ lifts to a $\ast$-isomorphism of the corresponding von Neumann algebras, $\Theta: \cL(G)\ra \cL(H)$.  As $H$ is torsion free, the desired conclusion follows directly from Corollary \ref{Cor:GW}. 
\hfill$\blacksquare$

\begin{proposition} Let $G=A_0\wr_I \Gamma\in \mathscr {GW}_1(A_0)$. If  $A_0$ and $\Gamma$ satisfy the Baum-Connes conjecture then so does $G$. In particular, $C^*_r(G)$ has no non-trivial projections.  
\end{proposition}
\begin{proof} The class of groups satisfying the Baum-Connes conjecture is closed under direct products and also direct unions \cite[Theorem 8.7 (2)(d)]{Lu18} and therefore the core $A_0^{I}$ satisfies the Baum-Connes conjecture as well. Since Q satisfies the Baum-Connes conjecture then using \cite[Theorem 3.1]{O-O01b} we conclude that $G= A_0^{I}\rtimes \Gamma=A_0\wr_I \Gamma $ satisfies the Baum-Connes conjecture. \end{proof}

\begin{proposition}\label{BCgroups} Assume that $A_0$ belongs to any of the following classes of groups:
\begin{enumerate}
    \item  Let $\G$ be a torsion free hyperbolic group and let $\G_1,\G_2,..., \G_n$ be isomorphic copies of $\G$. Let $\G\ca^{\rho_i} \G_i$ be the action by conjugation and let $\G\ca^\rho \G_1\ast \G_2\ast...\ast\G_n$ the action induced by the canonical free product automorphism $\rho_g =\rho_g^1\ast \rho_g^2\ast...\ast \rho_g^n $. Denote by $A_0 = (\G_1\ast\G_2\ast...\ast\G_n ) \rtimes_\rho \G$ the corresponding semi-direct product.
    \item Let  $C$ be a torsion free Haagerup group and let $B$ be a torsion free hyperbolic group. Let $A< B$ be an infinite cyclic subgroup that is hyperbolically embedded. Consider the canonical inclusion  $\Sigma :=  C^{B}\rtimes A<C\wr B=:\G$ and let ${\rm diag}(\Sigma) =\{(g,g)\,:\, g\in \Sigma \}<\G\times \G$ be the diagonal subgroup. Then consider the amalgamated free product $A_0= (\Gamma \times \Gamma)\ast_{{\rm diag}(\Sigma)} (\Gamma\times \Gamma) $. 
\end{enumerate} 
Then $A_0$ satisfies the Baum-Connes conjecture. 
\end{proposition}

{\it Proof.} 1. Since $\G_i$ are hyperbolic it follows that the free product $\G_1\ast\G_2\ast...\ast\G_n$ is hyperbolic as well. By \cite{La12}  this implies that  $\G_1\ast\G_2\ast...\ast\G_n$ satisfies the Baum-Connes conjecture. Since $\G$ satisfies the Baum-Connes conjecture it follows from \cite{O-O01b} that $N= (\G_1\ast \G_2\ast...\ast \G_n)\rtimes_\rho \G$ satisfies the Baum-Connes conjecture as well. 

2. The result follows using a series of stability results for the class of groups satisfying this conjecture. Recall that $\Gamma=  C \wr B$ where $B$ is a hyperbolic group and $\Sigma=  C^B \rtimes \mathbb Z$ with $A< B$ is a hyperbolically embedded group isomorphic to the integers. Since a hyperbolic group satisfies the Baum-Connes conjecture \cite{La12} and $C^B$ has Haagerup property it follows from \cite{HK1,HK2,O-O01b} that $\Gamma= C \wr B$ satisfies the Baum-Connes conjecture as well. As the class of groups satisfying the Baum-Connes conjecture is closed under direct products \cite[Theorem 3.1]{O-O01b} and amalgamated free product \cite[Corollary 1.2]{O-O01a} we conclude that $A_0=(\Gamma \times \Gamma)\ast_{diag(\Sigma)} (\Gamma\times \Gamma)$ satisfies the Baum-Connes conjecture, as desired.
\hfill$\blacksquare$

Therefore combining  Corollary \ref{C^*} with Proposition \ref{BCgroups}, Proposition \ref{trivialamenableradical} and \cite{CD-AD20}, we get several new families of examples of non-amenable torsion free abstractly C$^*$-superrigid groups. 

\begin{corollary}\label{C^*-superrigid5} Let $G=A_0\wr_I \Gamma\in \mathscr {G W}_1(A_0)$ be any group such that $\Gamma$ is torsion free,  property (T), hyperbolic and $A_0$ belongs to one of the following classes: 

\begin{enumerate}
    \item  Let $\G_1,\G_2,..., \G_n$ be isomorphic copies of $\G$. Let $\G\ca^{\rho_i} \G_i$ be the action by conjugation and let $\G\ca^\rho \G_1\ast \G_2\ast...\ast\G_n$ the action induced by the canonical free product automorphism $\rho_g =\rho_g^1\ast \rho_g^2\ast...\ast \rho_g^n $. Denote by $A_0 = (\G_1\ast\G_2\ast...\ast\G_n ) \rtimes_\rho \G$ the corresponding semi-direct product.
    \item Let  $C$ be torsion free, amenable icc group  and let $B$ be a torsion free hyperbolic property (T) group. Let $A< B$ be an infinite cyclic subgroup that is hyperbolically embedded. Consider the canonical inclusion  $\Sigma := C^{B}\rtimes A <C\wr B=:\G$ and let ${\rm diag}(\Sigma) =\{(g,g)\,:\, g\in \Sigma \}<\G\times \G$ be the diagonal subgroup. Then consider the amalgamated free product $A_0= (\Gamma \times \Gamma)\ast_{{\rm diag}(\Sigma)} (\Gamma\times \Gamma) $. 
\end{enumerate}

Then  for every group $H$ and any $\ast$-isomorphism  $\Theta: \cL(G)\ra \cL(H)$ there exist a group isomorphism $\rho:G \ra H$,  a character $\eta:G\ra \mathbb T$, and  unitaries $ t\in \cL(\delta(A_0))$, $w\in \cL(H)$ such that  $\Theta= {\rm ad} (w) \circ  \Phi_t \circ \Psi_{\omega,\delta}$. In particular, $G$ is abstractly C$^*$-superrigid and not C$^*$-superrigid.
\end{corollary}

We conclude this section with the following family of C$^*$-superrigid coinduced and generalized wreath product groups.

\begin{corollary}\label{superrigidC^*2}
Let $\G$ be an icc, torsion free, bi-exact, property (T) group and let  $\G_0,\G_1$,\dots,$\G_n$ be isomorphic copies of $\G$. For every $1\leq i\leq n$ consider the action $\G_0 \curvearrowright^{\rho^i} \G_i$ by conjugation, i.e.\ $\rho^i_\g (\la)=\g\la \g^{-1}$, for all $\g\in \G, \la\in \G_i$ and let $\G_0\ca^{\rho} \G_1\ast \G_2\ast ...\ast \G_n$ be the action induced by the canonical free product automorphism $\rho_\g = \rho^1_\g \ast \rho^2_\g \ast ...\ast \rho^n_\g$, for all $\g\in \G$. Note that $\Gamma_0$ can be seen as a subgroup of $\tilde\Gamma=\Gamma\times\Gamma$ by letting $\Gamma_0=\{(g,g)|g\in\Gamma\}<\tilde\Gamma$. 
 
Then the associated coinduced group $G=(\G_1\ast \G_2\ast ...\ast \G_n)^I\rtimes \tilde\Gamma$ of $\rho$ is C$^*$-superrigid.
\end{corollary}

{\it Proof.} The group $G$ is W$^*$-superrigid by Corollary \ref{Cor:coind}. Next, Proposition \ref{trivialamenableradical} implies that $G$ has trivial amenable radical, and hence, by using \cite[Theorem 1.3]{BKKO14},  it follows that  C$_r^*(G)$ has unique trace. This further implies that any $\ast$-isomorphism betweem C$_r^*(G)$ and C$_r^*(H)$, where $H$ is a countable group, lifts to a $\ast$-isomorphism of the corresponding von Neumann algebras. This concludes the proof.
\hfill$\blacksquare$

\begin{corollary} Let $\Gamma$ be an icc, torsion free, weakly amenable, bi-exact, property (T) group. Let $A_0$ be an icc abstractly W$^*$-superrigid with trivial amenable radical.

Then the left-right wreath product group $G=A_0\wr_{\Gamma}(\Gamma\times\Gamma)$ is abstractly C$^*$-superrigid.

\end{corollary}

\section{Automorphisms of reduced group $C^*$-algebras and group von Neumann algebras}\label{Sec:Autom}

A topic of central importance in the theory of $C^*$-algebras is to understand the structure of the symmetries of a $C^*$-algebra $\mathcal N $, i.e., the automorphisms group ${\rm Aut}(\mathcal N)$. This group is usually endowed with the topology induced by the $\|\cdot\|_\infty$-pointwise convergence. Inside this topological group there are two canonical subgroups, namely the group of inner automorphisms $\rm Inn(\mathcal N)$ and its closure $\overline{\rm Inn (\mathcal N)}$. Closely related to these groups there is another class of groups that we briefly describe below. Assume that $\pi: \mathcal N \rightarrow B(\mathcal H)$ is a faithful $\ast$-representation of the $C^*$-algebra $\mathcal N$. An automorphism $\theta \in \rm Aut(\mathcal N)$ is called \emph{$\pi$-weakly inner} if there is a unitary $u$ belonging to $\pi(\mathcal N)''$, the von Neumann algebra generated by $\pi(\mathcal N)$ in $B(\mathcal H)$, such that $\theta={\rm ad}(u)$. The class of all these automorphisms form a subgroup of ${\rm Aut}(\mathcal N)$ that will be denoted by ${\rm wInn}_\pi(\mathcal N)$. When $\mathcal N= C^*_r(G)$ and $\pi$ is induced by the left regular representation of $G$ this subgroup will simply be denoted by ${\rm wInn}(\mathcal N)$ this being the only case under consideration throughout our paper. By analogy, the elements of ${\rm wInn}(\mathcal N)$ will be called \emph{weakly inner automorphisms}. Moreover the set of all unitaries $u\in \mathscr U(\mathcal N)$ implementing an element of ${\rm wInn}(\mathcal N)$ form a subgroup which will be denoted by $\mathscr U_{\rm wInn}(\mathcal N)$ throughout the section.

Also if $\tau$ is a fixed tracial state on $\cN$ then we denote by ${\rm Aut}_\tau(\cN)$ the subgroup of all  $\tau$-preserving automorphisms in ${\rm Aut}(\mathcal N)$.

Next we recall some elementary properties of these automorphism subgroups and for convenience we also include some proofs.

\begin{theorem} Let $G$ be a countable group and let $\mathcal N=C^*_r(G)$. Then the following hold:
\begin{enumerate}
\item We have the following group inclusions ${\rm Inn}(\mathcal N)\leqslant \overline{{\rm Inn}(\mathcal N)}, {\rm wInn}(\mathcal N)\leqslant {\rm Aut}_\tau(\mathcal N)\leqslant {\rm Aut}(\mathcal N)$, with ${\rm Inn}(\mathcal N)$ and $\overline{{\rm Inn}(\mathcal N)}$ normal in ${\rm Aut}(\mathcal N)$. Also ${\rm Aut}_\tau(\mathcal N)$ is closed in ${\rm Aut}(\mathcal N)$. 

\item The subgroup ${\rm wInn}(\mathcal N)$ is normal in ${\rm Aut}_\tau(\mathcal N)$. In particular, when $G$ has trivial amenable radical then ${\rm wInn}(\mathcal N)$ is normal in ${\rm Aut}(\mathcal N)$. 

\item If $\mathcal L(G)$ is a full II$_1$ factor, then ${\rm wInn}(\mathcal N)$ is a closed subgroup of ${\rm Aut}(\mathcal N)$, and hence ${\rm Inn}(\mathcal N)\leqslant \overline{{\rm Inn}(\mathcal N)}\leqslant {\rm wInn}(\mathcal N)\leqslant {\rm Aut}(\mathcal N)$.  
\end{enumerate} 
\end{theorem}

{\it Proof.} 1) We only check normality of $\overline{{\rm Inn}(\mathcal N)}$. Fix $\theta\in \overline{{\rm Inn}(\mathcal N)}$  and $\sigma \in \rm Aut(\mathcal N )$. There exist a sequence $u_n\in \mathscr U(\mathcal N)$ so that $\lim_n\|\theta(x)-u_nxu_n^*\|_\infty= 0$ for all $x\in \mathcal N$. As $\sigma$ preserves $\|\cdot \|_\infty$ we get  $\lim_n\|\sigma \circ \theta(x)-\sigma(u_nxu_n^*) \|_\infty=\lim_n\|\sigma(\theta(x)-u_nxu_n^*)\|_\infty= 0$ for all $x\in \mathcal N$. Letting $x=\sigma^{-1}(y)$ with $y\in \mathcal N$ we get $\lim_n\|\sigma \circ \theta\circ \sigma^{-1}(y)-\sigma(u_n)y \sigma(u_n^*) \|_\infty= 0$ for all $y\in \mathcal N$. As $\sigma(u_n)\in \mathscr U(\mathcal N)$ we conclude that $\sigma\circ \theta\circ \sigma^{-1}\in \overline{\rm Inn(\mathcal N)}$. Since this holds for all $\sigma\in {\rm Aut}(\mathcal N)$, the statement follows.

2) To check normality fix $\theta \in {\rm wInn} (\mathcal N)$ and $\sigma \in {\rm Aut}_\tau (\mathcal N)$. As $\sigma$ is $\tau$-invariant then it extends to $\tilde \sigma \in {\rm Aut}(\mathcal L(G))$. Picking $u\in \mathscr U(\mathcal L(G))$ so that $\theta={\rm ad} (u)$ we have $\sigma \circ \theta\circ\sigma^{-1}(x)= \sigma( u \sigma^{-1}(x)u^* )= \tilde \sigma( u \sigma^{-1}(x)u^* )= \tilde \sigma(u) x \tilde\sigma (u^*)$, for all $x\in \mathcal N$. Hence $ \sigma \circ \theta \circ \sigma^{-1}\in {\rm wInn}(\mathcal N)$ and thus ${\rm wInn} (\mathcal N)$ is normal in ${\rm Aut}_\tau (\mathcal N)$. When $G$ has trivial amenable radical then by \cite[Theorem 1.3]{BKKO14} $\mathcal N$ has unique trace and hence ${\rm Aut}_\tau (\mathcal N)={\rm Aut} (\mathcal N)$ and the conclusion follows from above. 
\vskip 0.03in

3) Fix $\theta \in \overline{{\rm wInn} (\mathcal N)}$. Let $u_n\in \mathscr U(\mathcal L(G))$ be a sequence so that $\lim_n\|\theta(x)- u_nxu_n^*\|_\infty=0$,  for all $x\in \mathcal N$. Notice this implies that $\theta$ is $\tau$-preserving and thus it extends to an automorphism  $\tilde \theta\in {\rm Aut}(\mathcal L(G))$. Also we canonically have $\lim_n\|\theta(x)- u_nxu_n^*\|_2= 0$ and hence $\lim_n\|\tilde \theta(x)- u_nxu_n^*\|_2= 0$, for all $x\in \mathcal N$. Now fix $\varepsilon >0$ and $x\in \mathcal L(G)$. By Kaplansky Density theorem there is $x_\varepsilon\in \mathcal N$ so that $\|x_\varepsilon\|_\infty\leq\|x\|_\infty$ and $\|x-x_\varepsilon\|_2\leq \varepsilon$. From above there is  $n_\varepsilon \in \mathbb N$ so that $\|\tilde \theta(x_\varepsilon)- u_nx_\varepsilon u_n^*\|_2\leq \varepsilon$ for all $n\geq n_\varepsilon$. These estimates and the triangle inequality show that for all $n\geq n_\varepsilon$ we have $\|\tilde \theta(x)-u_n xu_n^*\|_2\leq \| \tilde \theta (x)-\tilde \theta (x_\varepsilon)\|_2 +\| \theta (x_\varepsilon)-u_n x_\varepsilon u_n^*\|_2+ \|u_n xu_n^*-u_n x_\varepsilon u_n^*\|_2=2\|x-x_\varepsilon\|_2+ \| \theta (x_\varepsilon)-u_n x_\varepsilon u_n^*\|_2\leq 3\varepsilon$. In particular, $\tilde \theta$ belongs to the closure ${\rm Inn }(\mathcal L(G))$ with respect to the $\|\cdot\|_2$-pointwise topology. As $\mathcal L(G)$ is full then by \cite{Co74} the group ${\rm Inn }(\mathcal L(G))$ is closed in this topology and hence $\tilde \theta\in \rm Inn (\mathcal L(G))$. Thus $\tilde\theta={\rm ad}(w)$ for some $w\in \mathscr U(\mathcal L(G))$ and hence $\theta \in {\rm wInn}(\mathcal N)$ showing that ${\rm wInn}(\mathcal N)$ is closed.
\hfill$\blacksquare$

Consequently, to every reduced group $C^*$-algebra $\mathcal N=C^*_r(G)$ we can associate the following automorphism  quotient groups: the \emph{outer automorphism group} ${\rm Out}(\mathcal N )={\rm Aut}(\mathcal N )/{\rm Inn}(\mathcal N ) $ and the \emph{strictly outer $\tau$-preserving automorphism group} ${\rm sOut}_\tau(\mathcal N)= {\rm Aut}_\tau(\mathcal N)/{\rm wInn}  (\mathcal N) $. We also have the \emph{strictly weakly inner automorphism subgroup} ${\rm swInn}(\cN)={\rm wInn}(\cN)/{\rm Inn}(\cN)$ that is normal in ${\rm Aut}_\tau(\cN)/{\rm Inn}(\cN)$.   When $\mathcal N$ has unique trace then ${\rm sOut}_\tau(\mathcal N)$ will be denoted by ${\rm sOut}(\mathcal N)$. In this situation we have the following canonical short exact squence 
\begin{equation}
   1 \ra {\rm swInn}(\cN)\ra {\rm Out}(\cN)\ra {\rm sOut}(\cN)\ra 1. 
\end{equation}

\vskip 0.05in

This shows that understanding the structure of the outer automorphisms of $\cN$  amounts to the structural study of the strictly weakly-inner autormorphism and the strictly outer automorphisms of $\cN$.  Even from the early studies in the subject it emerged that in general $C^*$-algebras tend to have an abundance of strictly weakly inner automorphisms \cite{KR67,AP79,Ki81}. In fact, these works culminated in the mid eighties with a remarkably general result by J. Phillips which asserts that for all non-inner amenable icc groups $G$ the $C^*$-algebra $\cN=C^*_r(G)$ possesses \emph{uncountably} many strictly weakly-inner automorphisms \cite{Ph87}. The precise statement is the following.

\begin{theorem}\label{phillipsresult}\cite[Theorem 3.1 and Corollary 3.2]{Ph87} Let $G$ be an icc group such that $\mathcal L(G)$ is a full factor. Then the quotient $\overline{{\rm Inn} (C^*_r(G))}/ {\rm Inn}(C^*_r(G))$ (and hence ${\rm swInn} (C^*_r(G))$)  is uncountable.
\end{theorem}

This theorem highlights that in general the strictly weakly-inner automorphism group is pretty wild. In fact, little is known about its structure and future investigations in this direction  would be very interesting. For example, we do not even know whether in general all these automorphisms are approximately inner or not.

Using Popa's deformation/rigidity theory, over the last fifteen years or so, many instances have been discovered when complete calculation of automorphism groups were achieved for various families of von Neumann algebras arising from groups and their trace preserving actions. For example, in \cite{IPP05} it was shown that there exist many II$_1$ factors that possess only inner automorphisms, thus settling a long-standing open question of A. Connes. More generally, in \cite{PV06} it was shown in fact that  every finitely presented group can be realized as the outer automorphism group of a certain group factor $\mathcal L(G)$. See also \cite{FV08,Va08,PV21} for other very interesting results in this direction. However, these impressive results in the von Neumann algebraic context do not yield much insight towards the reduced group $C^*$-algebras. In fact, at this time there are very few results in the literature regarding the possible structure of ${\rm Aut}(C^*_r(G))$ for $G$ non-amenable. Except for the main results in \cite{CI17,CD-AD20}, the other existent von Neumann algebraic results do not seem to canonically extend to the $C^*$-algebraic framework. 
\vskip 0.06in
In the remaining part, we use the rigidity results from the prior sections to describe the automorphisms of some reduced group $C^*$-algebras, especially its strictly outer automorphisms.  

We start by observing that all families of groups 
$G$ covered by \cite[Corollary C]{CI17}, \cite[Theorem 11.2]{CD-AD20} and Theorem \ref{superrigidC^*2} in this paper satisfy in particular the following formula \begin{equation}\label{soutform}{\rm sOut} (C^*_r(G))= {\rm Char}(G)\rtimes {\rm Out}(G).\end{equation} Pairing this formula with Theorems \ref{computeauta} and \ref{autcoinduceda}, one obtains concrete descriptions of these strictly outer automorphism groups. However, to simplify the writing, we highlight only a few applications of these results.   
\vskip 0.04in 
Recall that two groups $G$ and $H$ are commensurable up to finite kernels (abbrev.\ $G\cong_{\rm fk}H$) if and only if there exist finite index subgroups $G_1\leqslant G$, $H_1\leqslant H$ and finite, normal subgroups $M_1\lhd G_1$,  $N_1\lhd H_1$ such that the quotients  are isomorphic, $G_1/M_1\cong H_1/N_1$. It is easy to see that commensurability up to finite kernels is an equivalence relation.

\begin{corollary} \begin{enumerate}
\item For every icc, torsion free, bi-exact, property (T) group $\Gamma$, there is a countable icc group $G$ such that ${\rm sOut}(C^*_r(G))\cong_{\rm fk} {\rm Aut}(\G)$. 
\item For any icc torsion free, hyperbolic, property (T) group $\Gamma$, there exists a countable icc group $G$ such that ${\rm sOut}(C^*_r(G))\cong_{\rm fk} \G$. \end{enumerate}
\end{corollary}
\begin{proof} For each $\G$  we indicate two constructions of $G$ satisfying the statement. First, let $\G_1,\G_2,\ldots, \G_n$ be isomorphic copies of $\G$. Let $\G\ca^{\rho_i} \G_i$ be the action by conjugation and let $\G\ca^\rho \G_1\ast \G_2\ast\cdots\ast \G_n=A$ be the action induced by the canonical free product automorphism $\rho_g =\rho_g^1\ast \rho_g^2\ast\cdots\ast \rho_g^n $ for $g\in \G$. Let $G_1= A \rtimes_\rho Q$ be the corresponding semi-direct product. Using formula \eqref{soutform}  and Theorem \ref{computeauta} we obtain that ${\rm sOut} (C^*_r(G_1))\cong_{\rm fk} {\rm Aut} (\G)$. 
Second, let $\tilde \G=\G\times \G$ and assume $\G_0=\G<\tilde \G$ is diagonally embedded. Denote by $I= \tilde \G/\G_0$ and consider  $G_2=A^I\rtimes \tilde \G$ the coinduced group associated with $\G \ca^\rho A$. Then by \eqref{soutform} and Theorem \ref{autcoinduceda} we get ${\rm sOut} (C^*_r(G_2))\cong_{\rm fk} {\rm Aut} (\G)$, as desired.   

2) Since $\G$ is hyperbolic and has property (T) then Paulin's result implies that ${\rm Out}(\G)$ is finite. In particular, ${\rm Aut}(\G)\cong_{\rm fk} \G$ and the conclusion follows from part 1).
\end{proof}

\noindent \emph{Remarks}. If in the prior statement, one considers additionally that $\G$ has trivial abelianization, ${\rm Out}(\G)=1$  and $n=2$ in both constructions of $G$ we actually obtain that ${\rm sOut}(C^*_r(G))$ contains a copy of $\G$ as a subgroup of index two.   
\vskip 0.03in

Now, we turn towards describing the  automorphisms of wreath-products algebras. First, we tackle the von Neumann algebras arising from groups in $\mathscr {GW}_1(A_0)$. For these algebras we have the following von Neumann algebraic counterpart of Houghton's so-called automorphisms $KIB^*$-decomposition in group theory.

\begin{corollary}\label{C^*} For any $G=A_0\wr_I \Gamma\in \mathscr {GW}_1(A_0)$, where  $A_0,\G$ are torsion free we have
\begin{equation*}{\rm Aut}( \cL(G))=\{{\rm ad} (w) \circ  \Phi_\phi \circ \Psi_{\eta,\delta} \,|\,\delta\in {\rm Aut}(G), \eta\in {\rm Char} (G),\phi\in {\rm Aut}( \cL(A_0)), w\in \mathscr U(\cL(G))\}.\end{equation*}
If in addition $A_0$ satisfies ${\rm Out}(\cL(A_0))={\rm  Char}(A_0)\rtimes {\rm Out(A_0)}$ (e.g.\ when is $W^*$-superrigid) then
\begin{equation*}\begin{split}&{\rm Aut}( \cL(G))=\{{\rm ad} (w) \circ  \Phi_v \circ \Psi_{\eta,\delta} \,|\,\delta\in {\rm Aut}(G), \eta\in {\rm Char} (G),v\in {\mathscr U}( \cL(A_0)), w\in \mathscr U(\cL(G))\}.
\end{split}\end{equation*}

\end{corollary}

 Under the result's assumptions, the first part combined with Theorem \ref{autgenbernoulli} gives that 
 ${\rm Out}(\cL(G)\cong {\rm Aut}(\cL(A_0))\times ({\rm Char}(\G)\rtimes {\rm Aut}_{\G_0} (\G)).$
 Under the second assumption we have 
 ${\rm Out}(\cL(G)\cong ({\rm Inn}(\cL(A_0)) $ $({\rm Char} (A_0)\rtimes {\rm Out}(A_0)) \times ({\rm Char}(\G)\rtimes {\rm Aut}_{\G_0} (\G).$
 Thus, the ``non-discrete'' part in the outer automorphisms arises from the unitary group of $\cL(A_0)$. Iterating the construction provided in the Corollary \ref{C^*} one can obtain more and more complicated outer automorphism groups containing new unitary groups of various II$_1$ factors at every level. Specifically, fix $G_1= A_0\wr_I\G  \in \mathscr {GW}_1(A_0)$ and for every $k\geq 2$ we can define inductively  $G_k:= G_{k-1} \wr_I \G\in \mathscr {GW}_1(G_k)$. Therefore, we have the following recursion formula ${\rm Out}(\cL(G_k)\cong {\rm Aut}(\cL(G_{k-1}))\times ({\rm Char}(\G)\rtimes {\rm Aut}_{\G_0} (\G))= ({\rm Inn}(\cL(G_{k-1})) {\rm Out}(\cL(G_{k-1}))\times ({\rm Char}(\G)\rtimes {\rm Aut}_{\G_0} (\G))$  which implies our claim.  
 
 \vskip 0.03in
Corollary \ref{C^*-superrigid5} also sheds some light towards understanding the automorphism of the reduced group $C^*$-algebra of groups in class $\mathscr {GW}_1(A_0)$. More precisely, let $G= A_0\wr_I \G$ where $A_0$ is $W^*$-superrigid group with trivial amenable radical---for instance for all groups $A_0$ covered by Corollary \ref{C^*-superrigid5}. Thus, the conclusion of the same corollary implies that \begin{equation}{\rm Aut}( C^*_r(G))\subseteq \{{\rm ad} (w) \circ  \Phi_v \circ \Psi_{\eta,\delta} \,|\,\delta\in {\rm Aut}(G), \eta\in {\rm Char} (G),v\in \mathscr U( \cL(A_0)), w\in \mathscr U(\cL(G))\}.\end{equation}

Observe that ${\rm Aut}( C^*_r(G))$ contains all automorphisms ${\rm ad} (w) \circ  \Phi_v \circ \Psi_{\eta,\delta}$ whenever $v\in \mathscr U_{\rm wInn}( \cL(A_0))$, $w\in \mathscr U_{\rm wInn}(\cL(G))$, and we conjecture these should be all the automomorphisms. However, we could not determine whether this condition is automatic only from the fact that ${\rm ad} (w) \circ  \Phi_v \circ \Psi_{\eta,\delta}$ invaries $C^*_r(G)$ and further analysis is needed. We leave this as an open problem.
 \vskip 0.05in

In the final part of the paper we highlight some computations of automorphisms groups of algebras associated with coinduced groups.

\begin{corollary}   Assume $\Gamma_0\car A_0$ is a group action by automorphisms satisfying the following: 
\begin{enumerate}
    \item $\tilde\Gamma=\Gamma\times\Gamma$, where $\Gamma$ is an icc, torsion free, weakly amenable, bi-exact, property (T) group. 
    \item $\Gamma_0<\tilde\Gamma$ is the diagonal subgroup defined by $\Gamma_0=\{(g,g)\,|\,g\in\Gamma\}.$
    \item One of the following two properties holds  \begin{enumerate}
\item The action $\G_0\ca A_0$ is trivial;
    \item $A_0$ is icc, torsion free, bi-exact and it contains an infinite property (T) subgroup $K_0$ that is $\Gamma_0$-invariant and has trivial virtual centralizer. \end{enumerate}
    \end{enumerate}
    If $G=A_0^I\rtimes \tilde\Gamma$ is the group coinduced from $\Gamma_0\car A_0$ then $${\rm Aut}( \cL(G))=\{{\rm ad} (w) \circ  \Phi_\phi \circ \Psi_{\eta,\delta} \,|\,\delta\in Aut(G), \eta\in {\rm Char} (G),\phi\in C_{{\rm Aut}( \cL(A_0))}(\G_0), w\in \mathscr U(\cL(H))\}.$$
\end{corollary} 

Here, we denoted by $C_{{\rm Aut}( \cL(A_0))}(\G_0)$ the subgroup of  ${\rm Aut}( \cL(A_0))$ that contains all the $\Gamma_0$-equivariant automorphisms.

\end{document}